\documentclass[12pt, a4paper, english]{article}

\usepackage[utf8]{inputenc}
\usepackage[T1]{fontenc}
\usepackage[nomath]{kpfonts}

\DeclareUnicodeCharacter{1ECD}{\d{o}}

\usepackage{amssymb}
\usepackage{mathtools}

	\DeclareMathOperator{\tr}{Tr}
	\DeclareMathOperator{\hess}{Hess}
	\DeclareMathOperator{\cv}{cv}

	\DeclareMathOperator{\id}{Id}
	\DeclareMathOperator{\grph}{Graph}

	\DeclareMathOperator{\wind}{Wind}
	
	\DeclareMathOperator{\rank}{rank}
	\DeclareMathOperator{\dive}{div}

	\DeclareMathOperator{\sub}{sub}
	\DeclareMathOperator{\cov}{cov}

\usepackage{caption}

\usepackage{nicematrix}

\usepackage{hyperref}
	\hypersetup{
		breaklinks=true,   
		pdfusetitle=true,  
	}

\usepackage[backend=biber,style=trad-abbrv]{biblatex}

\usepackage{authblk}

\usepackage{geometry}

\usepackage{tikz}
	\usetikzlibrary{cd}

\newcommand{\quotemks}[1]{``#1''}

\newcommand{\barg}[1][\hbar]{\mathcal{B}_{#1}}
\newcommand{\statesp}[1][]{\mathcal{H}^{#1}_{\hbar}}

\newcommand{\Fio}{\mathfrak{U}}

\newcommand{\diag}[1][]{\mathcal{D}_{#1}}

\newcommand{\Cm}{\mathbb{C}}
\newcommand{\Nm}{\mathbb{N}}

\newcommand{\Rm}{\mathbb{R}}
\newcommand{\Sm}{\mathbb{S}}

\newcommand{\Zm}{\mathbb{Z}}

\newcommand{\Ec}{\mathcal{E}}

\newcommand{\Hc}{\mathcal{H}}

\newcommand{\Kc}{\mathcal{K}}
\newcommand{\Lc}{\mathcal{L}}

\newcommand{\Nc}{\mathcal{N}}

\newcommand{\Pc}{\mathcal{P}}

\newcommand{\Wc}{\mathcal{W}}

\newcommand{\lpar}{\left(}
\newcommand{\rpar}{\right)}
\newcommand{\lacc}{\left\{}
\newcommand{\racc}{\right\}}
\newcommand{\lver}{\left|}
\newcommand{\rver}{\right|}
\newcommand{\lnor}{\left\|}
\newcommand{\rnor}{\right\|}
\newcommand{\lbra}{\left[}
\newcommand{\rbra}{\right]}

\usepackage{babel}[english]
	\usepackage{csquotes}

\usepackage{amsthm}
	\newtheorem{theorem}{Theorem}[section]
	\newtheorem*{theorem*}{Theorem}
	\newtheorem{corollary}{Corollary}[theorem]
	\newtheorem{lemma}[theorem]{Lemma}
	\newtheorem{proposition}[theorem]{Proposition}
	\theoremstyle{definition}
	\newtheorem{definition}[theorem]{Definition}
	\newtheorem{example}[theorem]{Example}
	\newtheorem{remark}[theorem]{Remark}
	\newtheorem{hypothesis}[theorem]{Hypothesis}

\title{Eigenvalues of non self-adjoint Toeplitz operators near an elliptic critical value with analytic regularity}
\author{Nathan Reguer\footnote{Univ Rennes, CNRS, IRMAR - UMR 6625, F-35000 Rennes, France.\\\strut\hfill MSC codes: 81Q12,81Q20,58J40,58J50,35A20}}
\date{}

\begin{document}

\maketitle
\begin{abstract}
In this article, we determine the spectrum of real-analytic, non self-adjoint Toeplitz operators on compact Kähler manifolds and on the complex plane, on neighbourhoods of critical values of the symbol. We consider specifically critical values of the symbol on which its Hessian is elliptic and we get asymptotic expansion on eigenvalues in a neighbourhood with quantisation conditions similar to Bohr-Sommerfeld. To do so, we recall and further develop analytic semiclassical tools, in particular the symbolic calculus of complex Fourier integral operators using contour deformation. We detail the well known case of operators with quadratic symbols, and we treat a general case through normal form reduction. Finally, we prove resolvent estimates on norms with weights that come from the non-real part of the symbol.
\end{abstract}
\tableofcontents

\section{Introduction}

\label{sec_intro}

A ubiquitous problem in quantum mechanics is to find the energy levels of a system, which means to get the discrete spectrum of operators. There exists no known systematic way to get exact formulas for the spectrum of an operator. Nevertheless, we can use semiclassical analysis to get asymptotic equalities with respect to the parameter $\hbar$. The essential tool is then the quantisation of symbols, which are functions defined on the phase space and satisfying specific regularity hypothesis. If the operator represent the quantum system, the symbol then corresponds to its classical counterpart. In particular, the symbol give information on the operator in the semiclassical limit. For instance, if $f$ is in a specific class of symbols, recall that its Weyl quantisation is
\[
Op_{\hbar}^W(f)u(x) = \int_{\Rm^2} e^{\frac{i}{\hbar}(x-y)\xi} f\lpar\frac{x+y}{2},\xi\rpar u(y) \frac{dyd\xi}{2\pi\hbar}
\]
on a domain of $L^2(\Rm^d)$. $f$ may depend on $\hbar$, but we call principal symbol the term $f_0$ such that $f\sim f_0 + O(\hbar)$ as $\hbar\rightarrow 0$. In this framework, we seek quasimodes, which are states $u_{\hbar}$ and values $\lambda_{\hbar}$ such that $Op_{\hbar}^W(f)u_{\hbar} = \lambda_{\hbar} u_{\hbar} + O_{\hbar}(\hbar^k)$ for $k\in\Nm$. If the operator is self-adjoint, or just normal, then for $z$ in the resolvent set,

\begin{equation}
\label{eq_res_norm}
\lnor\lpar Op_{\hbar}^W(f)-z\id \rpar^{-1}\rnor = \frac{1}{d(z,\sigma(Op_{\hbar}^W(f)))}.
\end{equation}
Hence, quasimodes give elements of the spectrum up to a small remainder with respect to $\hbar$. Many results are based on this idea, see for instance \cite{robe81,sjos92,coli94,vun00}.

However, some problems require understanding the spectrum of non self-adjoint operators. For instance, some models in fluid dynamics and financial mathematics use advection-diffusion operators with variable coefficients or on domains with boundaries. Orr-Sommerfeld and Airy operators also appear in hydrodynamic stability. More operators appear in fluid dynamics with turbulence, 
magnetohydrodynamics and the stability of some numerical analysis of PDEs. These examples are presented by Trefethen and Embree in \cite{tref05} with references for each of them. The difficulty with non self-adjoint operators is that the spectrum is not localised on the real line, and \eqref{eq_res_norm} no longer stands.

In this article, we consider a sole operator in dimension $1$. In this context, every system is completely integrable. Moreover, we study Toeplitz operators. This class of operators is interesting as it is well-defined on both $\Cm$ and compact phase spaces, and it is equivalent to pseudodifferential operators on the complex plane. On $\Cm$, Toeplitz operators are defined on the Bargmann space

\[
\barg = \lacc e^{-\frac{|x|^2}{2\hbar}}u(x) \middle/\; u \text{ holomorphic on }\Cm\racc \cap L^2(\Cm).
\]
Then, denoting $\Pi_{\hbar}$ the projection from $L^2(\Cm)$ to $\barg$, for any bounded function $f$ on $\Cm$, the Toeplitz operator of symbol $f$ is $T_{\hbar}(f)u = \Pi_{\hbar}(fu)$. See \cite{mart02} for a description of $\barg$, and \cite{sjos96,zwor12} for the definition and properties of Toeplitz operators. If $f$ is in the right class of symbols, then $T_{\hbar}(f)$ is unitarily equivalent to a pseudodifferential operator with the same principal symbol. For compact manifolds, if they have suitable structures, a similar scheme can be done, where $\barg$ is replaced by a quantum space built from geometric considerations, see \cite{lef18} for instance.

Even in dimension $1$, determining the spectrum of non self-adjoint operators is non-trivial. One way is via perturbative methods, but it is limited to symbols with a small imaginary part, see \cite{roub16} for instance. For example, consider the explicit toy model $-\hbar^2\Delta + x^2$ called the Harmonic oscillator. The perturbation $-\hbar^2\Delta + e^{i\theta}x^2$ for $0<\theta<\pi$ has purely discrete spectrum along $e^{i\frac{\theta}{2}}[0,+\infty[$, but quasimodes can be found for all complex numbers with arguments in $[0,\theta]$, and the spectrum is no longer discrete if $\theta=\pi$, see \cite{davi99}. Moreover, the known methods for self-adjoint operators cannot be directly applied to the non self-adjoint ones. For instance, an already well-developed strategy for a real-valued symbol $f$ is to find a local diffeomorphism $\kappa$ from a neighbourhood of $f^{-1}(\{E\})$ to a neighbourhood of the circle of radius $\sqrt{E}$ that transforms $f$ into a function of the symbol of the harmonic oscillator. The next step is to quantise $\kappa$ into an operator $\Fio$, and use an adequate Egorov's theorem to relate $\Fio T_{\hbar}(f)\Fio^{-1}$ to $T_{\hbar}(f\circ\kappa^{-1})$. This method gives the so-called Bohr-Sommerfeld quantisation conditions: there exists a function $\mu$ such that the eigenvalues are asymptotically given by $\mu(\hbar\Nm)$ in a neighbourhood of $E$. Many articles compute these quantisation conditions for different frameworks, for example \cite{coli94,char03a,lef14}. However, if $f$ has complex values, the diffeomorphism $\kappa$ only makes sense in the complexified space, which changes the construction of $\Fio$ and $\mu$.

In this article, we consider Toeplitz operators on $M = \Cm$ or a compact manifold. We fix an energy $E$, and our purpose is to find a class of symbols for which we can prove similar results to what we know on self-adjoint operators. We consider that $f^{-1}(\{E\})$ is made of critical points, with an ellipticity criterion. Our result will cover operators which are far from being self-adjoint, and we will get a quantisation condition similar to Bohr-Sommerfeld:

\begin{theorem}
\label{th_intro}
Let $M$ be either a compact Kähler manifold of complex dimension $1$ or $\Cm$, $\mathfrak{e}\in \Cm$, and $f$ a real-analytic function on $M$. We consider Hypothesis \ref{hyp_global} to be satisfied. It means that $f$ is in a class of analytic symbols, and writing $f_0$ its principal symbol, there exists $k\in\Nm$ points $x_1,\dots,x_k\in M$ such that for $1\le n\le k$, $f_0(x_n)=E$. Moreover $f(x)\neq E$ for $x\notin \lacc x_1,\dots,x_k\racc$, $df_0(x_n)=0$, and the quadratic form $H_n = \hess_{x_n}f$ is elliptic, meaning $H_n(x,y)=0 \Rightarrow x=0=y$, and satisfies $H_n(T_{x_n}M) \neq \Cm$, for all $n\in\{1,\dots,k\}$. If $M=\Cm$, there exists $\rho >0$ such that $f$ admits a holomorphic extension $\widetilde{f}$ on $\lacc \lpar x,\overline{y}\rpar \in\Cm^2 /\; |x-y|\le \rho \racc$, meaning $\widetilde{f}(x,\overline{x}) = f(x)$, and there exists $c>0$ and an order function $m$ such that $\lver \widetilde{f}\rver \le cm$, and outside a neighbourhood of $x_0$, $\lver \widetilde{f}\rver \ge \frac{m}{c}$.\\
Then, there exists $C,A,r>0$ and real-analytic functions $\mu_{j,n}$ such that $|\mu_{j,n}(\xi)|\le Cj!A^j$ for $\xi$ in a fixed neighbourhood of $0$, and such that the spectrum of $T_N(f)$ satisfies

\begin{gather*}
\sigma(T_N(f)) \cap B_r(\mathfrak{e}) = \bigcup_{1\le n\le k} \lacc \lambda_{n,l} \middle/\; l\in\Nm^*\racc \bigcap B_r(\mathfrak{e}),\\
\lambda_{n,l} = \mu_n^c(l) + O(e^{-\frac{c}{\hbar}}),
\end{gather*}
for all $n,l\in\Nm^2$ and for $\hbar$ small enough. Moreover,there exists $\rho,C>0$ with $B_r\subset f(B_{\rho})$, such that for any $\lambda\in B_r\cap\sigma(T_N(f))^c$

\[
\lnor (T_N(f)-\lambda)^{-1} \rnor_{L^2\lpar B_{\rho}(x_n),e^{-\frac{\Wc_n}{\hbar}}\rpar\rightarrow L^2\lpar B_{\rho}(x_n),e^{-\frac{\Wc_{n,i}}{\hbar}}\rpar} \le \frac{C}{d(\lambda,\sigma(T_N(f))},
\]
where $\Wc_n$,$\Wc_{n,i}$ are positive functions that vanish at order $1$ at $x_n$, for $1\le n\le k$. Moreover, $\Wc_n$ and $\Wc_{n,i}$ vanish at order $2$ at $x_n$ if and only if the matrices $\Re(H_n)$ and $\Im(H_n)$ are collinear.
\end{theorem}

Since \eqref{eq_res_norm} is not satisfied for non self-adjoint operators in general, one can look instead for subsets of $f(\Rm^2)$ where the resolvent is small, these are often called pseudo-spectra in the literature. Sjöstrand \cite{sjos09} made a retrospective on spectral problems for non self-adjoint operators, we want to come back to some results. \cite{agmo62} proves that elliptic linear differential operators satisfy estimates of the form $\lnor (P-\lambda\id)^{-1}\rnor \le \frac{c}{|\lambda|}$ for $\lambda$ big enough and on specific lines of $\Cm$. \cite{davi99} studies semiclassical Schrödinger operators with complex potentials, and proves that the norm of the resolvent $(Op_{\hbar}^W(f)-rz\id)^{-1}$ diverges as $r\to\infty$ at different speeds depending on $z\in\Cm$. The result is refined in \cite{zwor01}. The same operators are studied in \cite{denc03}, where they build functions $u_{\hbar}$ such that $\lnor Op_{\hbar}^W(f)u_{\hbar}-zu_{\hbar}\rnor = O(\hbar^{\infty})\lnor u_{\hbar}\rnor$ for $z\in f(\Rm^2)$ that satisfy a geometric condition. In the same article, the authors prove the same result with $O(\hbar^{\infty})$ replaced by $O(e^{-\frac{c(z)}{\hbar}})$ if the functions are taken analytic. Our aim is to prove a uniform bound in $z$, which we obtain in Theorem \ref{th_intro}. It is also done in \cite{dura25a} for self-adjoint pseudo-differential operators on $\Rm$ with analytic data, near a regular value of the symbol, using an isometry between $L^2(\Rm)$ and the Bargmann space $\barg$.

From what can be seen in the literature, it often seems necessary to make analyticity assumptions on the data to study spectral properties of non self-adjoint operators. For instance, \cite{hitr03,hitr05,roub16} consider small non self-adjoint perturbation of a self-adjoint pseudifferential operator, with real-analytic symbols. They obtain asymptotic expressions of the eigenvalues with an $\hbar^{\infty}$ remainder, near regular values of the symbols. \cite{hitr04} also compute such an expansion near elliptic critical energies, for non self-adjoint operators. In addition, he links the first term of the asymptotic to the determinant of the Hessian of the symbol. \cite{dele25} proves a similar result, but they obtain asymptotic expansion with a $O(e^{-\frac{c}{\hbar}})$ remainder for Toeplitz operators on compact Kähler manifolds.
Hitrik and Zworski considered in \cite{hitr24} a similar problem as in this article but for pseudo differential operators on $\Rm$, and treated it with slightly different tools.
Both in their article and the present one, the Hessian of the principal symbol plays an important role as it is the limiting factor to diagonalise the operator. Indeed, in order to reduce the operator to its normal form, we first use an operator $\Fio$ such that $\Fio T_{\hbar}(f) = T_{\hbar}(\mu(H))\Fio + O\lpar e^{-\frac{c}{\hbar}}\rpar$ with $\mu(E)=\id+O(E^2)$, hence $\Fio$ is close to the identity. Then we apply a second operator $U$ such that $U T_{\hbar}(H) = T_{\hbar}(d_0(x^2+y^2))U + O\lpar e^{-\frac{c}{\hbar}}\rpar$, with $d_0\in\Cm$. Since $H$ can have values far from the real line, the norm of $U$ can be very big, which forces us to use $L^2$ norms with weights in Theorem \ref{th_intro}.

The article will be organised as follows. In Section \ref{sec_quadratic}, we consider a quadratic symbol $f$, and construct a symplectomorphism to reduce it to $d_0(x^2+y^2)$, $d_0\in\Cm$. This construction is explicit, but it can only be done on the complexification of the space and with the holomorphic extensions of the functions since the symbol can have values far from the real line. Then, we build a Fourier Integral Operator (FIO) associated to the symplectomorphism and get that, under specific assumptions on $f$, the spectrum of $T_{\hbar}(f)$ is made of the simple eigenvalues:

\[
\lacc \hbar \lpar \sqrt{\det (f)} \lpar k+\frac{1}{2}\rpar + \frac{\tr(f)}{4} \rpar \;\middle/ k\in\Nm \racc.
\]

Section \ref{sec_toolbox} summarises the required tools of analytic semiclassical analysis. First, we detail how the analytic pseudo-spectrum will allow us to get asymptotic expression of the eigenvalues, despite the size of the resolvent. In a second part, following ideas from \cite{sjos82}, we compute the products of FIOs with general phases. Then, we detail the analytic symbolic calculus of Toeplitz operators on $\Cm^d$, since the case on manifolds can be reduced to the analyticity of the Bergman kernel \cite{roub20,char19,dele21,heza21,dele20a}. These computations require contour deformations on complex integrals, the necessary results are detailed in Section \ref{sec_contour}.

In Section \ref{sec_local}, we consider a general symbol $f : M\rightarrow \Cm$ with a unique well, which means that $k=1$ in the hypothesis of Theorem \ref{th_intro}. We conjugate $T_N(f)$ to its normal form using several operators. First, we use an analytic Morse Lemma to get a symplectomorphism that turns $f$ into  $\mu(H) + O(\hbar)$ with $\mu$ an analytic symbol, and we quantise it using results from \cite{dele25}. Then, the new symbol is $\mu(H) + O(\hbar)$ on $\Cm$, and we gain one order of $\hbar$ in the remainder term using conjugation by Toeplitz operators. The condition on the Hessian appears at this step as an application of Section \ref{sec_quadratic}. The last step consists of a Moser's trick on Toeplitz operators symbols to get an exponentially small remainder. Then, we combine these operators to obtain the quantisation conditions of Theorem \ref{th_intro} when $k=1$. We end the Section with the computation of $\mu^{-1}$ at order $1$ in $\hbar$.

In Section \ref{sec_global}, we extend the results of Section \ref{sec_local} to the general hypothesis of Theorem \ref{th_intro}. Combining the information on the operator, we finally get the general shape of the spectrum, summarised in figure \ref{fig_shape_spectrum}.

\subsection*{Notations}

In this article we mainly work on $\Cm^d$ with $d\in\Nm$, we denote

\[
y\cdot z = \sum\limits_{1\le j\le d} y_j z_j, \quad y^2 = y\cdot y, \quad |y|^2 = y\cdot \overline{y},
\]
for $(y,z)\in(\Cm^d)^2$, and we use the distance $|\cdot|$. Let $\rho>0$ and $f$ be a $C^1$ function on $\Cm^d$, then

\begin{itemize}
\item $B_{\rho},\widetilde{B_{\rho}}$ are the open balls of centre $0$ and radius $\rho$ in $\Cm^d$ and $\Cm^{2d}$ respectively.
\item $\diag = \lacc (z,\overline{v}) \in \Cm^{2d}\middle/ z=v \racc$, and $\diag[\rho] = \diag \cap \widetilde{B_{\rho}}$.
\item $\partial_{x_j}f = \frac{1}{2}\lpar \frac{df}{d\Re(x_j)}-i\frac{df}{d\Im(x_j)}\rpar$, and $\overline{\partial_{x_j}}f = \frac{1}{2}\lpar \frac{df}{d\Re(x_j)}+i\frac{df}{d\Im(x_j)}\rpar$ for $1\le j\le d$.
\end{itemize}
Let $f$ be a real-analytic function on $B_{\rho}$, we write $\widetilde{f}$ a holomorphic extension, meaning a holormohic function on $B_{\eta}$ with $\eta\in ]0,\rho]$, such that $\widetilde{f}(z,\overline{z}) = f(z)$, see Lemma \ref{le_holo_ext}. Then, we denote

\[
\partial_1 \widetilde{f}(x,\overline{y}) = \partial_x \widetilde{f}(x,\overline{y}), \quad \partial_2 \widetilde{f}(x,\overline{y}) = \overline{\partial_y} \widetilde{f}(x,\overline{y}).
\]
We will also denote $\Pc$ the Bergman kernel, that we introduce in Proposition \ref{prop_local_berg}.

\subsection*{Acknowledgements}

This work was supported by the ANR-24-CE40-5905-01 “STENTOR” project. The author thanks Alix Deleporte and San Vũ Ng\d{o}c for their guidance throughout this project and their proof-readings.

\section{Quadratic symbol}

\label{sec_quadratic}

As said in the introduction, the strategy for a general symbol $f$ will be to find a FIO $\Fio$ that reduces the symbol to $\mu(|z|^2)$. Actually, we will see in Section \ref{sec_local} that it is done in two steps. First, we reduce $f$ to $\mu(H)$, with $H$ the Hessian of $f$ at the critical point, seen as a quadratic function. Then, we reduce $H$ to a constant times $|z|^2$. From a geometric point of view, the first step is simpler, as the corresponding FIO will be locally close to the identity, meanwhile the kernel of the second FIO can be localised far from the diagonal of $\Cm^2$. Hence, we study in this section the quadratic case in detail, which will give the conditions on the Hessian of $f$ to make Theorem \ref{th_intro} work. At the same time, we recall the spectral description of a class of quadratic operators, although it is already known, see for example \cite{sjos74,prav07,hitr11}.

\subsection{Normal form}

\label{subsec_quad_symbol}

We consider in this Section Toeplitz operators with quadratic symbols on $\Cm$. Recall that the space of states is $\barg$, and for $f:\Cm \rightarrow \Cm$ in the right space of symbol,

\[
T_{\hbar}(f)u(x) = \int_{\Cm} e^{\frac{2x\overline{y}-|x|^2-|y|^2}{2\hbar}} F(y) u(y) \frac{dy}{\pi\hbar}.
\]
In particular, $T_{\hbar}(z^{\alpha}\overline{z}^{\beta}): e^{-\frac{|x|^2}{2\hbar}}v(x) \mapsto \hbar^{\beta} e^{-\frac{|x|^2}{2\hbar}} \partial_z^{\beta}(z^{\alpha}v(z))|_{z=x}$ for all $\alpha,\beta \in \Nm^2$.

In this section, we write the symbol in terms of the real variables $(p,q)$ such that $z=\frac{p+iq}{\sqrt{2}}$. Let $f(p,q) = ap^2+bq^2+2cpq$ with $a=a_0+ia_p$, $b=b_0+ib_p$ and $c=c_0+ic_p$. We recall how to build a symplectomorphism $\kappa$ such that $f\circ\kappa^{-1}(p,q) = d(p^2+q^2)$ on the complexified space. This Section follows the construction of \cite{prav07}, we wanted to detail it as the limiting condition of the spectral study appear in the computations. For it to work, the imaginary part of $f$ has to be \quotemks{small} compared to the real part, and the right condition is actually

\begin{align}
\label{eq_complex_quad}
& \det(\Re f) + \det(\Im f) + i\sqrt{\lver \Im(\det f)^2-4\det(\Re f)\det(\Im f)\rver} \notin \Rm_{-},\\\nonumber
& f(\Rm^2) \neq \Cm.
\end{align}
The next Lemma gives a condition equivalent to the first line, but more meaningful. We keep this expression nonetheless as it will directly appear in the computations.

\begin{lemma}
The first line of equation \eqref{eq_complex_quad} is equivalent to $f$ elliptic, meaning that $f(x,y)=0$ implies $x=0=y$.
\end{lemma}

\begin{proof}
The first equation in \eqref{eq_complex_quad} is equivalent to $\det\Re(f) + \det\Im(f) > 0$ or $(\Im\det(f))^2 \neq 4\det\Re(f)\det\Im(f)$. The first inequality is equivalent to $\Re(f)$ or $\Im(f)$ elliptic. The second condition is equivalent to $\Re(f)$ and $\Im(f)$ not collinear and $f$ elliptic. In order to prove it, we write $f_1=a_0b_p-b_0a_p$, $f_2=a_0c_p-c_0a_p$, $f_3 = b_0c_p-c_0b_p$ and $F=f_1^2+4f_2f_3= \Im(\det f)^2-4\det(\Re f)\det(\Im f)$. With these notations, $f(p,q)=0$ implies

\[
\left\{\begin{array}{c}
f_2 p^2 + f_3 q^2 = 0\\
f_1 p^2 -2f_3 pq=0\\
f_1 q^2 +2f_2 pq = 0
\end{array}\right.
\]
which is equivalent to

\begin{equation}
\label{eq_elliptic_equi}
\left\{\begin{array}{c}
f_1 p = 2f_3 q\\
f_1 q = -2f_2 p\\
F pq = 0
\end{array}\right.
\end{equation}
Thus, if $F\neq 0$ then either $f(p,q)=0$ implies $p=0=q$, or two of the coefficients $f_1,f_2,f_3$ vanish, which implies that $\Re(f)$ and $\Im(f)$ are collinear. Conversely, if $F=0$ then either $f_1=f_2=f_3=0$, meaning that $\Re(f)$ and $\Im(f)$ are collinear, or $f_1$, $f_2$ and $f_3$ are different from $0$, in which case \eqref{eq_elliptic_equi} is equivalent to $f(p,q)=0$, so $f$ is not elliptic.
\end{proof}

It is possible to simplify even further condition \eqref{eq_complex_quad}.

\begin{proposition}[\cite{prav07} Proposition 2.1.1]
Let $q : \Rm^2 \rightarrow \Cm$ a complex-valued elliptic quadratic form. If $q(\Rm^{2}) \neq \Cm$, then there exists $\delta\in \Cm^*$ such that $\Re(\delta q)$ is a positive definite quadratic form.
\end{proposition}

The method we are using works only if this hypothesis is fulfilled. Without it, the symbol can be hyperbolic, which leads to a totally different kind of study, see for instance \cite{lef14}. For simplicity, we assume directly $\Re(f)$ positive definite, which implies $a_0,b_0,a_0b_0-c_0^2 > 0$.

First, consider the symplectomorphism
\[
\kappa_1 : (p,q) \mapsto (p_1,q_1) = \lpar \sqrt{\frac{a_0}{d_0}}p + \frac{c_0}{\sqrt{d_0a_0}}q,\sqrt{\frac{d_0}{a_0}}q \rpar
\]
with $d_0=\sqrt{a_0b_0-c_0^2} \in\Rm$, since we took $a_0b_0-c_0^2>0$. Using these new coordinates we get

\begin{align*}
f\circ \kappa_1^{-1}~(p_1,q_1) & = d_0(p_1^2+q_1^2+i(\alpha p_1^2 +\beta q_1^2 +2\gamma p_1q_1)) \text{, where}\\
\alpha & = \frac{a_p}{a_0}\\
\gamma & = \frac{c_pa_0-a_pc_0}{d_0a_0}\\
\beta & = \frac{1}{d_0^2}\lpar a_0b_p -2c_pc_0+\frac{c_0^2a_p}{a_0}\rpar
\end{align*}
Then, consider the unitary transformation
\[
\kappa_2 = \begin{pmatrix}
\frac{\sqrt{2}\gamma}{\sqrt{\Delta(\Delta+\beta-\alpha)}} & \pm\frac{\sqrt{2}\gamma}{\sqrt{\Delta(\Delta-(\beta-\alpha))}} \\
\frac{\sqrt{\Delta+\beta-\alpha}}{\sqrt{2\Delta}} & \mp\frac{\sqrt{\Delta-(\beta-\alpha)}}{\sqrt{2\Delta}}
\end{pmatrix}
\]
where $\Delta = \sqrt{(\beta-\alpha)^2+(2\gamma)^2} = \frac{\sqrt{\Im(\det f)^2-4\det(\Re f)\det(\Im f)}}{\det(\Re f)}$ and the sign chosen such that the determinant is $1$. The terms $\Delta$, $\Delta + (\beta-\alpha)$ and $\Delta - (\beta-\alpha)$ vanish only if $\gamma=0$ in which case there is no need for this step, and we can take $\kappa_2=\id$. We get
\[
f\circ (\kappa_2 \circ \kappa_1)^{-1}(p_2,q_2) = d_0\lpar 1+i\frac{\beta+\alpha+\Delta}{2}\rpar p_2^2+ d_0\lpar 1+i\frac{\beta+\alpha-\Delta}{2} \rpar q_2^2.
\]

We write $r = d_0\lpar 1+i\frac{\beta+\alpha-\Delta}{2}\rpar$ and $
\zeta = \frac{1+i\frac{\beta+\alpha+\Delta}{2}}{1+i\frac{\beta+\alpha-\Delta}{2}}$ so that $f\circ (\kappa_2 \circ \kappa_1)^{-1}(p_2,q_2) = r(\zeta p_2^2+q_2^2)$. Then, it is natural to consider the transformation $(p,q) \mapsto (\zeta^{\frac{1}{4}}p,\zeta^{-\frac{1}{4}}q)$, however it does not preserve $\Rm^2$. Instead, we consider the holomorphic extensions $\widetilde{f}$, $\widetilde{\kappa_1}$ and $\widetilde{\kappa_2}$ on the complexification $\Cm^2$ of $\Rm^2$, which then satisfy

\begin{align*}
\widetilde{f}(z,\overline{v})
	= & a\frac{(z+\overline{v})^2}{2}-b\frac{(z-\overline{v})^2}{2}+ci(z+\overline{v})(z-\overline{v}),\\
\widetilde{f}\circ\lpar\widetilde{\kappa_2}\circ\widetilde{\kappa_1}\rpar^{-1}(z,\overline{v})
	= & r\lpar \frac{\zeta-1}{2}z^2 +\frac{\zeta-1}{2}\overline{v}^2 +(\zeta+1)z\overline{v}\rpar.
\end{align*}
Then denote $\widetilde{\kappa_3}$ the holomorphic symplectomorphism of $\Cm^2$ corresponding to the matrix
\[
\frac{1}{2}
\begin{pmatrix}
\zeta^{\frac{1}{4}}+\zeta^{-\frac{1}{4}} & \zeta^{\frac{1}{4}}-\zeta^{-\frac{1}{4}}\\
\zeta^{\frac{1}{4}}-\zeta^{-\frac{1}{4}} & \zeta^{\frac{1}{4}}+\zeta^{-\frac{1}{4}}
\end{pmatrix}.
\]
The complex $\zeta^{\frac{1}{4}}$ is taken with positive real part, which is possible because

\begin{align*}
\zeta
	& = \frac{1+ \alpha\beta-\gamma^2 + i\Delta}{1+\lpar \frac{\beta+\alpha-\Delta}{2} \rpar^2}
	= \frac{1+ \frac{\det(\Im f)}{\det(\Re f)} + i\Delta}{1+\lpar \frac{\beta+\alpha-\Delta}{2} \rpar^2}\\
	& = \frac{\det(\Re f) + \det(\Im f) + i\sqrt{\Im(\det f)^2-4\det(\Re f)\det(\Im f)}}{\det(\Re f) \lpar 1+\lpar \frac{\beta+\alpha-\Delta}{2} \rpar^2 \rpar},
\end{align*}
and because of \eqref{eq_complex_quad}. We could also say that $\kappa_3: (p,q) \mapsto (\zeta^{\frac{1}{4}}p,\zeta^{-\frac{1}{4}}q)$ is defined from $\Rm^2$ to a subset of $\Cm^2$ and $\widetilde{\kappa_3}$ is its holomorphic extension. Hence, composing the three transformations we get $\widetilde{\kappa} = \widetilde{\kappa_3} \circ \widetilde{\kappa_2} \circ \widetilde{\kappa_1}$ such that $\widetilde{f}\circ \widetilde{\kappa}^{-1}(z,\overline{v}) = 2r\sqrt{\zeta}z\overline{v} = d_0z\overline{v}$ where $d_0= \sqrt{\det (f)}$ is the principal root, which is non-ambiguous as we supposed $\Re(f)$ definite positive, thus $\det(f)\notin \Rm_-$.

Here, we supposed $\Re(f)$ positive definite for simplicity, however we saw that in general this property is satisfied by $\delta f$ with $\delta\in\Sm$. Therefore, doing the same construction for $\delta f$ and dividing the result by $\delta$, we get $d_0 = \frac{\sqrt{\det (\delta f)}}{\delta}$.

\subsection{Quantisation formula}

\label{subsec_quad_fio}

We transformed the initial symbol to a constant times $(z,\overline{v}) \mapsto z\overline{v}$ which is the holomorphic extension of the harmonic oscillator. Now, we want to quantise $\widetilde{\kappa}$, meaning to build a FIOs such that the corresponding conjugation action change the principal symbols into their composition with the symplectomorphism.

\begin{lemma}
There exists a function $\phi$ associated to $\widetilde{\kappa}$, which means that for all $(x,y) \in \Cm^2$
\[
\widetilde{\kappa}(x,\partial_x \phi(x,\overline{y})) = (\overline{\partial_y} \phi(x,\overline{y}),\overline{y}).
\]
Moreover, $\phi$ is a phase function, meaning that for all $x\in \Cm$ the map

\begin{align*}
\overline{\Cm} & \rightarrow \Rm\\
\overline{y} & \mapsto 2\Re \phi(x,\overline{y}) -|y|^2,
\end{align*}
has a unique critical point $\overline{y_0(x)}\in\overline{\Cm}$ of signature $(0,-2)$, see Definition \ref{def_phase}.
\end{lemma}

\begin{proof}
We can write $\widetilde{\kappa}$ as
$
\begin{pmatrix}
a & b\\
c & d
\end{pmatrix}
$
with $b \neq 0$ and $ad-cb = 1$. Then $\phi(x,y) = \frac{1}{2d}(2x\overline{y} + b\overline{y}^2 -cx^2)$ satisfy the equality with $\widetilde{\kappa}$, we prove that it is a phase function. Indeed, the Hessian of $\overline{y}\mapsto 2\Re \phi(x,\overline{y}) - |y|^2$ is

\[
\frac{1}{2}
\begin{pmatrix}
-1+\Re \lpar\frac{b}{d}\rpar & \Im \lpar\frac{b}{d}\rpar\\
\Im \lpar\frac{b}{d}\rpar & -1-\Re \lpar\frac{b}{d}\rpar
\end{pmatrix},
\]

and it has a critical point of signature $(0,-2)$ if and only if the determinant of the Hessian is positive, which is equivalent to $\lver\frac{b}{d}\rver<1$. For $\widetilde{\kappa_1}$ and $\widetilde{\kappa_2}$ the condition is always satisfied, meanwhile for $\widetilde{\kappa_3}$ the condition is equivalent to

\[
\frac{\Re\lpar\lpar\frac{\zeta}{|\zeta|}\rpar^{\frac{1}{2}}\rpar}{|\zeta^{\frac{1}{4}}-\zeta^{-\frac{1}{4}}|} >0
\]
with $\Re\lpar\lpar\frac{\zeta}{|\zeta|}\rpar^{\frac{1}{2}}\rpar = \frac{1}{\sqrt{2}}\sqrt{1+\Re\lpar\frac{\zeta}{|\zeta|}\rpar}$. Then, the determinant of the Hessian of the composition $\widetilde{\kappa_3}\circ\widetilde{\kappa_2}\circ\widetilde{\kappa_1}$ is positive if and only if $\sqrt{1+\frac{\Re(\zeta)}{|\zeta|}}>0$, which is satisfied because of \eqref{eq_complex_quad}.
\end{proof}

We will detail this notion of phase function in Section \ref{subsec_symbol_calcul}. We want to pinpoint the fact that this lemma gives a sufficient condition on the symbol such that the symplectomorphism is quantisable, this is the point where \eqref{eq_complex_quad} is crucial. We now use this function $\phi$ to define a FIO that will simplify $f$.

\begin{theorem}
\label{th_spec_quad}
Let $f$ be an elliptic quadratic function on $\Rm^2$ such that $f(\Rm^2)\neq \Cm^2$, then there exists $\delta\in\Sm$ such that $\Re\lpar\delta f\rpar$ is positive definite. Furthermore, the spectrum of the associated Toeplitz operator $T_{\hbar}(f)$ is made of isolated simple eigenvalues

\[
\sigma \lpar T_{\hbar}(f)\rpar = \sigma_d\lpar T_{\hbar}(f)\rpar = \lacc \hbar \lpar \frac{\sqrt{\det (\delta f)}}{\delta} \frac{2k+1}{2} + \frac{\tr f}{2} \rpar \;\middle/ k\in\Nm \racc,
\]
denoting $\tr f = a+b$, the trace of the matrix representation of $f$.
\end{theorem}

\begin{proof}
By definition of $\phi$,

\[
I_{\hbar,\phi}u : x\mapsto \int_{\Cm} e^{\frac{2\phi(x,y)-|y|^2-|x|^2}{2\hbar}} u(y) \frac{dy}{\pi\hbar}
\]
is bounded from $\barg\cap L^2\lpar\Cm,e^{w\frac{|y|^2}{\hbar}}dz\rpar$ to $\barg$ with $w = \frac{1}{2}|\partial_y^2\phi(0,0)| <1$. We compute,

\begin{align*}
T_{\hbar}(f) I_{\phi} u(x)
	= & \int_{\Cm} e^{\frac{2x\overline{y}-|y|^2-|x|^2}{2\hbar}} f(y) \int_{\Cm} e^{\frac{2\phi(y,\overline{z})-|z|^2-|y|^2}{2\hbar}} u(z) \frac{dz}{\pi\hbar} \frac{dy}{\pi\hbar}\\
	= & \int_{\Cm} e^{\frac{-|z|^2-|x|^2}{2\hbar}} u(z) \int_{\Cm} e^{\frac{\phi(y,\overline{z})+x\overline{y}-|y|^2}{\hbar}} f(y) \frac{dy}{\pi\hbar} \frac{dz}{\pi\hbar}\\
\end{align*}
and notice that

\[
\phi(y,\overline{z})+x\overline{y}-|y|^2  = \phi(x,\overline{z}) -(y-x)\lpar \overline{y} - \frac{\partial_x^2\phi}{2}(y-x) - \partial_x\phi(x,\overline{z})\rpar,
\]
hence,

\begin{align*}
	& T_{\hbar}(f) I_{\phi} u(x)\\
	= & \int_{\Cm} e^{\frac{2\phi(x,\overline{z})-|z|^2-|x|^2}{2\hbar}} u(z) \int_{\Cm} e^{\frac{-(y-x)\lpar \overline{y} - \frac{\partial_x^2\phi}{2}y - \frac{\partial_x^2\phi}{2}x - \partial_x\phi(x,\overline{z}) \rpar}{\hbar}} f(y) \frac{dy}{\pi\hbar} \frac{dz}{\pi\hbar}\\
	= & \int_{\Cm} e^{\frac{2\phi(x,\overline{z})-|z|^2-|x|^2}{2\hbar}} u(z) \int_{\Cm} e^{\frac{-y\overline{y}}{\hbar}} \widetilde{f} \lpar y+x,\overline{y}+\frac{\partial_x^2\phi}{2}y+2\frac{\partial_x^2\phi}{2}x+ \partial_x\phi(x,\overline{z}) \rpar \frac{dy}{\pi\hbar} \frac{dz}{\pi\hbar}\\
	= & \int_{\Cm} e^{\frac{2\phi(x,\overline{z})-|z|^2-|x|^2}{2\hbar}} u(z) c(x,z) \frac{dz}{\pi\hbar}
\end{align*}
by a linear contour deformation, see Proposition \ref{prop_linear_contours}. Then,

\begin{align*}
c(x,z)
	= & \sum\limits_{k=0}^1 \frac{\hbar^k}{k!} \partial_z^k\overline{\partial_z}^k \lpar \widetilde{f}\lpar y+x,\overline{y}+\frac{\partial_x^2\phi}{2}y+2\frac{\partial_x^2\phi}{2}x + \partial_x\phi(x,\overline{z}) \rpar\rpar (0)\\
	= & \widetilde{f}(x,\partial_x\phi(x,\overline{z})) + \partial_1\partial_2\widetilde{f}(x,\partial_x\phi(x,\overline{z})) + \frac{\partial_x^2\phi}{2}\partial_2^2\widetilde{f}(x,\partial_x\phi(x,\overline{z})),
\end{align*}
this computation is explicit because $f$ is quadratic, but we will generalise it to analytic functions on a ball in Lemma \ref{le_int_gauss}. In the other way, we consider the operator $T_{\hbar,\cov}(x\overline{y}): u(x) \mapsto xT_{\hbar}(\overline{y})u(x)$,

\begin{align*}
I_{\phi} T_{\hbar,\cov}(x\overline{y}) u(x)
	= & \int_{\Cm} e^{\frac{2\phi(x,\overline{y}-|y|^2-|x|^2}{2\hbar}} \int_{\Cm} e^{\frac{2y\overline{z}-|z|^2-|y|^2}{2\hbar}} y\overline{z} u(z) \frac{dz}{\pi\hbar} \frac{dy}{\pi\hbar}\\
	= & \int_{\Cm} e^{\frac{-|z|^2-|x|^2}{2\hbar}} \overline{z} u(z) \int_{\Cm} e^{\frac{\phi(x,\overline{y})+y\overline{z}-|y|^2}{\hbar}} y \frac{dy}{\pi\hbar} \frac{dz}{\pi\hbar}\\
	= & \int_{\Cm} e^{\frac{2\phi(x,\overline{z})-|z|^2-|x|^2}{2\hbar}} \overline{z} u(z) \int_{\Cm} e^{\frac{-\lpar y-\frac{\partial_{\overline{y}}^2\phi}{2}(\overline{y}-\overline{z})-\partial_{\overline{y}}\phi(x,\overline{z}) \rpar(\overline{y}-\overline{z})}{\hbar}} y \frac{dy}{\pi\hbar} \frac{dz}{\pi\hbar}\\
	= & \int_{\Cm} e^{\frac{2\phi(x,\overline{z})-|z|^2-|x|^2}{2\hbar}} \overline{z} u(z) \int_{\Cm} e^{\frac{-|v|^2}{\hbar}} \lpar v+\frac{\partial_{\overline{y}}^2\phi}{2}\overline{v}+\partial_{\overline{y}}\phi(x,\overline{z}) \rpar \frac{dv}{\pi\hbar} \frac{dz}{\pi\hbar}\\
	= & \int_{\Cm} e^{\frac{2\phi(x,\overline{z})-|z|^2-|x|^2}{2\hbar}} \overline{z} \partial_{\overline{y}}\phi(x,\overline{z}) u(z) \frac{dz}{\pi\hbar}
\end{align*}
Since $\widetilde{f}(x,\partial_x\phi(x,\overline{z})) = \widetilde{f}\circ \widetilde{\kappa} (\partial_{\overline{y}}\phi(x,\overline{z}),\overline{z}) = d_0\partial_2\phi(x,\overline{z}) \overline{z}$ by definition, and $T_{\hbar,\cov}(x\overline{y}) =  T_{\hbar}(|y|^2-\hbar)$,

\[
T_{\hbar}(f) I_{\phi} = I_{\phi} T_{\hbar}\lpar d_0 |y|^2 - d_0\hbar +  \partial_1\partial_2\widetilde{f}(x,\partial_x\phi(x,\overline{z})) + \frac{\partial_x^2\phi}{2}\partial_2^2\widetilde{f}(x,\partial_x\phi(x,\overline{z})) \rpar.
\]
Furthermore, differentiating $\widetilde{f}(x,\partial_x\phi(x,\overline{y}))=d_0\overline{y}\overline{\partial_y}\phi(x,\overline{y})$ with respect to $x$ and $\overline{y}$ gives

\[
\partial_x \overline{\partial_y}\phi(x,\overline{y})(\partial_1+\partial_x^2\phi(x,\overline{y})\partial_2)\partial_2\widetilde{f} = d_0\partial_x \overline{\partial_y}\phi(x,\overline{y}),
\]
and combined with the previous equality,

\begin{align*}
T_{\hbar}(f) I_{\phi}
	= & I_{\phi} T_{\hbar}\lpar d_0|y|^2-d_0\hbar + \frac{1}{2}(\partial\overline{\partial}f+d_0) \rpar\\
	= & I_{\phi} T_{\hbar}\lpar d_0|y|^2+\frac{\hbar\tr f}{2} -\frac{\hbar d_0}{2}\rpar.
\end{align*}
The operator $T_{\hbar}\lpar d_0|y|^2+\frac{\hbar\tr f}{2} -\frac{\hbar d}{2}\rpar$ has pure discrete spectrum, with simple eigenvalues $\hbar\lpar d_0(k+1)+\frac{\hbar\tr f}{2}-\frac{d_0}{2}\rpar$ for $k\in\Nm$, and eigenfunctions

\[
e^{-\frac{|z|^2}{2\hbar}}z^k \in \barg\cap L^2\lpar\Cm,e^{w\frac{|y|^2}{\hbar}}dz\rpar
\]
since $w<1$. Then, $I_{\phi}\lpar e^{-\frac{|z|^2}{2\hbar}}z^k \rpar \in \barg$ is an eigenfunction of $T_{\hbar}(f)$ with the same eigenvalue for every $k\in\Nm$. In the same manner, $I_{\phi}^{-1}$ turns every eigenfunction of $T_{\hbar}(f)$ into $e^{-\frac{|z|^2}{2\hbar}}z^k$ for a $k\in\Nm$, hence the result.
\end{proof}

We did not prove that $I_{\phi}$ was invertible, as we will detail its parametrix for a wider framework in Section \ref{sec_toolbox}. Since $f$ was quadratic here, the computations were explicit. The general conjugation action of $I_{\phi}$ will be described in Theorem \ref{th_FIO_normal} and used in Section \ref{sec_local}.

\section{Analytic semiclassical tools}

\label{sec_toolbox}

In this section, we uncover why spectral estimates up to an $\hbar^{\infty}$ remainder are not suited to study non self-adjoint operators in general, and why an analytic pseudo-spectrum is more relevant. In order to find this pseudo spectrum, we will reduce $T_N(f)$ to a function of the harmonic oscillator, up to an exponential remainder. To do so, we will consider FIOs first, then use the symbolic calculus of Toeplitz operators. Hence, we recall and detail the construction of FIOs with complex phases in analytic regularity. Then, the symbolic calculus of Toeplitz operators, in particular the topology of a class of analytic symbols. Notice that the analytic symbolic calculus is also described in \cite{bout67,bout72} for pseudodifferential operators, and in \cite{sjos82,dele24} for Toeplitz operators and complex FIOs. These two steps require computations of integrals over contours with a stationary phase method,  the results used in this section are provided in Section \ref{sec_contour}.

\subsection{Analytic pseudo spectrum}

In Section \ref{sec_quadratic}, the hypothesis on the symbol made the computations explicit, but from now on, we consider a general symbol $f$. A method is to look for quasimodes instead, meaning values $\nu_{\hbar}$ and states $u_{\hbar}$ such that

\[
T_{\hbar}(f) u_{\hbar} = \nu_{\hbar} u_{\hbar} + O(\hbar^{\infty}).
\]
Nevertheless, we sill see that it is not enough to find the spectrum if $f$ is not real-valued.

Let $f\in C^{\infty}(\Cm)$ be real valued and such that all its derivatives grow at most polynomially. Then $T_{\hbar}(f)$ is a self-adjoint Toeplitz operator on $\barg$, and a very useful result is that for all $\lambda$ in the resolvent set of $T_{\hbar}(f)$

\begin{equation}
\label{eq_spectral_th}
\lnor \lpar T_{\hbar}(f)-\lambda\id \rpar^{-1}\rnor_{\barg\rightarrow\barg} = \frac{1}{d(\lambda,\sigma(T_{\hbar}(f)))}.
\end{equation}
Hence, quasimodes correspond to values at distant $\hbar^{\infty}$ to the spectrum. Furthermore, $T_{\hbar}(f)$ also satisfies

\begin{equation}
\label{eq_pseudo_estim}
\frac{1}{\lnor \lpar T_{\hbar}(f)-\lambda\id \rpar^{-1}\rnor_{\barg\rightarrow\barg}} = \inf\limits_{x\in M} |f(x)-\lambda| + O(\hbar^{\frac{1}{2}}) 
\end{equation}
see \cite{bort03} Theorem 1.1, hence, combined with \eqref{eq_spectral_th}, it implies that the spectrum fills the classical range $\Sigma(f) = \overline{\lacc f_0(z) /\; z\in \Cm \racc}$ as $\hbar\to 0$, with $f_0$ the principal symbol of the operator. Notice that the classical range is a subset of $\Rm$ for self-adjoint operators.

Now, if $T_{\hbar}(f)$ is no longer self-adjoint, then \eqref{eq_pseudo_estim} still holds, but \eqref{eq_spectral_th} does not. In consequence, the classical range can be much larger than the spectrum, and the norm of the resolvent can be of size $\hbar^{-\infty}$ even far away from the spectrum. More precisely, \cite{bort03} Theorem 1.1 gives that, for all $\lambda \in\Cm$, if there exists $x_0\in M$ such that $\lambda=f(x_0)$, and $\lacc \Re(f),\Im(f) \racc (x_0) < 0$ then
\[
\frac{1}{\lnor \lpar T_{\hbar}(f)-\lambda\id \rpar^{-1}\rnor_{\barg\rightarrow\barg}} = O(\hbar^{\infty}).
\]

Consider two elementary differential operators on $\Rm$, the harmonic oscillator $H=-\hbar^2\Delta+x^2$, and for $-\pi <\theta<\pi$, the rotated harmonic oscillator $H_{\theta} = -\hbar^2\Delta+e^{i\theta}x^2$. Then the spectrum of $H$ is $\lacc \hbar(2k+1) /\;k\in\Nm \racc$ and the spectrum of $H_{\theta}$ is simply a rotation of the latter $\lacc e^{i\frac{\theta}{2}}\hbar(2k+1) /\;k\in\Nm \racc$. Meanwhile, the classical range of $H$ is $\Rm_+$, but the classical range of $H_{\theta}$ is $\lacc re^{is} /\; r\in\Rm_+, -\theta\le s\le\theta \racc$, as shown in Figure \ref{fig_rotate_oscil}. We refer to \cite{prav07,hitr11} for these results.

\begin{figure}[h]
\centering
\begin{minipage}[b]{.49\textwidth}
\begin{tikzpicture}[scale = 1.2]
	\draw[->,thick] (0,0)--(6,0);
	\draw[fill=black] (0,0) circle (0.3mm);
	\draw[fill=yellow,fill opacity=0.5] (1,0) circle (0.2cm);
	\draw (1,0.5) node{$\hbar$};
	\draw[fill=yellow,fill opacity=0.5] (3,0) circle (0.2cm);
	\draw (3,0.5) node{$3\hbar$};
	\draw[fill=yellow,fill opacity=0.5] (5,0) circle (0.2cm);
	\draw (5,0.5) node{$5\hbar$};
	\draw plot[only marks,mark=x,mark size=2pt] coordinates {(1,0) (3,0) (5,0)};
	\draw[blue,<->] (2.8,-0.3) -- (3.2,-0.3);
	\draw[blue] (3,-0.6) node{$2\hbar^{\infty}$};
\end{tikzpicture}
\end{minipage}
\hfill
\begin{minipage}[b]{.49\textwidth}
\begin{tikzpicture}[scale = 1.2]
\draw[->,thick] (0,0)--(5.5,0);
\draw[->,thick] (0,0)--(0,3);
\fill[color=lightgray,opacity=0.2] (0,0) -- (5.5,0) -- (5.5,3) -- (3.58,3) -- cycle;
\draw[gray] (0,0) -- (3.58,3);
\draw[color=gray] (1.5,2.5) node{$\Sigma(\xi^2+e^{i\theta}x^2)$};
\pgftransformrotate{20}
	\draw[fill=black] (0,0) circle (0.3mm);
	\draw[red,thick] (0,0)--(5.8,0) node[right] {$e^{i\frac{\theta}{2}}\Rm$};
	\draw (1,0.5) node{$e^{i\frac{\theta}{2}}\hbar$};
	\draw[fill=yellow,fill opacity=0.5] (1,0) circle (0.15cm);
	\draw (3,0.5) node{$3e^{i\frac{\theta}{2}}\hbar$};
	\draw[fill=yellow,fill opacity=0.5] (3,0) circle (0.15cm);
	\draw (5,0.5) node{$5e^{i\frac{\theta}{2}}\hbar$};
	\draw[fill=yellow,fill opacity=0.5] (5,0) circle (0.15cm);
	\draw plot[only marks,mark=x,mark size=2pt] coordinates {(1,0) (3,0) (5,0)};
	\draw[blue,<->] (2.85,-0.3) -- (3.15,-0.3);
	\draw[blue] (3,-0.6) node{$2e^{-\frac{c}{\hbar}}$};
\end{tikzpicture}
\end{minipage}
\caption{On the left: spectrum of $H$, the yellow balls are the sets where the norm of the resolvent is greater than $\hbar^{-\infty}$.\\On the right: Spectrum of $H_{\theta}$, the yellow balls are now the $c$-analytic pseudo-spectrum, and the grey area is the classical range.}
\label{fig_rotate_oscil}
\end{figure}
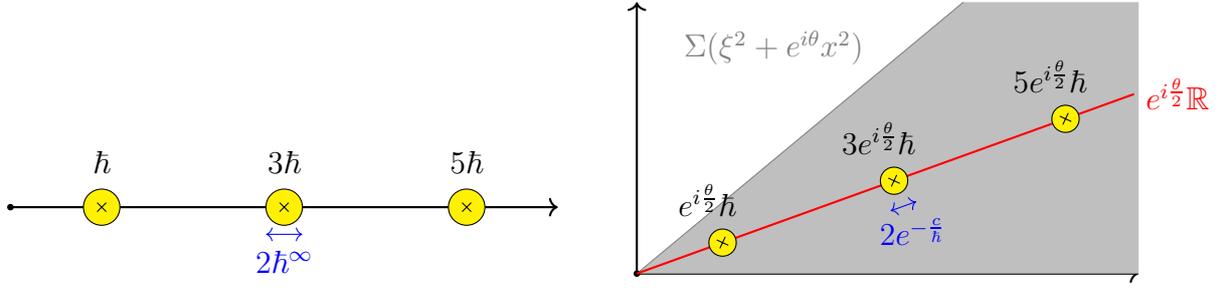
However, we will prove in Section \ref{sec_local} that the $c$-analytic pseudo-spectrum of $T_{\hbar}(f)$ with $c>0$, as defined below, consists of a reunion of small balls around each eigenvalue.

\begin{definition}
Let $P_{\hbar}$ be an operator on a Hilbert space $H$ with domain $D$, let $c>0$, then for a fixed $\hbar>0$, the $c$-analytic pseudo spectrum of $P_{\hbar}$ is
\[
\sigma_c(P_{\hbar}) = \lacc \lambda\in\Cm /\; \inf_{u\in D\backslash\{0\}} \frac{\lnor \lpar P_{\hbar}-\lambda\id \rpar u \rnor_{H}}{\lnor u \rnor_{H}} \le e^{-\frac{c}{\hbar}} \racc.
\]
\end{definition}
Hence, the aim will be to find the smallest analytic pseudo spectrum possible, the ideal scenario being when each connected component contains a unique eigenvalue. This expectation comes from the study of the quadratic case, see \cite{prav07,viol13,hitr11}, and the following.

\begin{lemma}[\cite{tref05} Theorem 4.3]
\label{le_pseudo_spec}
Let $c>0$, every bounded connected component of $\sigma_{c}(P)$ contains at least one point of $\sigma(P)$.
\end{lemma}
In order to get estimates on the operator with analytic accuracy, we consider data with analytic regularity in all this article.

\subsection{Complex phases Fourier Integral Operators}

The existing results and Section \ref{sec_quadratic} reveal that, when treating with non self-adjoint operators, one has to consider FIOs with complex-valued phase functions. The symbolic calculus and boundedness of this kind of operators is already well known, as it was proved by Melin, Sjöstrand and Hörmander in the smooth regularity \cite{meli75,horm83}. It was also proved in analytic regularity by Sjöstrand \cite{sjos82}. However, we will not write FIOs in the well-known way with an auxiliary variable and an integral over a contour. Indeed, recent studies \cite{dele24,dele25,dura25a} revealed a more convenient way to define them with a kernel in WKB form without caustic, meaning that in charts, it is of the form

\[
u \mapsto \int_{\Cm^d} e^{\frac{\phi(x,y)}{\hbar}} a(x,y) u(y) dy, \text{ instead of } \int_{\Gamma} e^{\frac{\psi(x,y,\theta)}{\hbar}} b(x,y,\theta) u(y) dyd\theta,
\]
with $\Gamma$ a contour, and $\phi,\, \psi,\, a,\, b$ complex valued function. For the sake of completeness, we provide in this subsection the definition and symbolic calculus of complex FIOs in any dimension, even if we will only consider the dimension $1$ in the rest of the article. We first need two computational results.

\begin{lemma}
\label{le_Cauchy_ana}
Let $0<\eta<\rho$, and $f$ be an analytic function on $B_{\rho}$ with a holomorphic extension $\widetilde{f}$ defined on $\widetilde{B_{\rho}}$, then for all $\alpha,\beta\in\Nm$, $\epsilon \in ]0,\rho-\eta[$, and $z\in \overline{B_{\eta}}$,
\[
\partial^{\alpha}\overline{\partial}^{\beta} f(z)
	= \frac{-\alpha!\beta!}{4\pi^2} \int_{C(z,\epsilon)} \frac{\widetilde{f}(\xi,\zeta)}{(\xi-z)^{\alpha+1}(\zeta-\overline{z})^{\beta+1}} d\xi d\zeta
\]
where $C(z,\epsilon) = \lacc (\xi,\zeta)\in\Cm^2 /\; \lver\xi-z\rver=\lver\zeta-\overline{z}\rver=\epsilon \racc$. Therefore,
\[
\sup_{z\in \overline{B_{\eta}}}\lver \partial^{\alpha}\overline{\partial}^{\beta} f(z) \rver
	\le \alpha!\beta! \lnor f \rnor_{L^{\infty}\lpar \overline{B_{\eta+\epsilon}}\rpar} \epsilon^{-(\alpha+\beta)}.
\]
\end{lemma}

\begin{lemma}[\cite{sjos82} Chapter 2]
\label{le_int_gauss}
Let $F(x,\overline{y})$ be holomorphic on $\widetilde{B_{\rho}}$, and $\eta\in]0,\rho[$, then for all $N\in\Nm$ 

\[
\int_{B_{\eta}} e^{-\frac{|y|^2}{\hbar}} F(y,\overline{y}) \frac{dy}{\pi\hbar} = \sum\limits_{k=0}^{N-1} \frac{\hbar^k}{k!} \lpar \partial_x\partial_{\overline{y}}\rpar^k F(0) + R_N(\hbar)
\]
with

\[
\lver R_N(\hbar) \rver = O \lpar \lnor F \rnor_{L^{\infty}\lpar \widetilde{B_{\rho-\epsilon}}\rpar} N^2 N! \lpar \frac{\hbar}{\min(\rho-\eta,\eta)} \rpar^N \rpar.
\]
Using the Stirling formula, when $N\rightarrow \infty$ we have

\[
N^2 N! \lpar\frac{\hbar}{\min(\rho-\eta,\eta)}\rpar^N \sim C N^{\frac{5}{2}} \lpar \frac{N\hbar}{e\min(\rho-\eta,\eta)} \rpar^N
\]
with $C>0$ a constant. The optimal $N$ to minimise this remainder is $\frac{\min(\rho-\eta,\eta)}{\hbar}$, in which case

\[
R_N(\hbar) = O\lpar \hbar^{-\frac{5}{2}} e^{-\frac{\min(\rho-\eta,\eta)}{\hbar}} \rpar.
\]
Hence, for any $0<\delta< 1$, we can take $N = \delta \frac{\min(\rho-\eta,\eta)}{\hbar}$ and get an exponentially small error.
\end{lemma}

\begin{proof}
By hypothesis, $(z,y) \mapsto u(z)v(\overline{y})$ is holomorphic on $\widetilde{B_{\rho}}$, so according to Lemma \ref{le_Cauchy_ana},

\[
\lver \partial_x^k\partial_{\overline{y}}^l F(0,0) \rver
	\le \frac{k!l!}{\eta^{k+l}}\lnor F \rnor_{L^{\infty}\lpar\widetilde{B_{\eta}}\rpar},
\]
Hence, for all $|z|>\eta$

\[
\lver \sum\limits_{k+l\le N-1} \lpar \partial_x\partial_{\overline{y}}\rpar^k F(0,0) \frac{z^k \overline{z}^l}{k!l!} \rver \le \frac{N^2}{2} \lpar \frac{|z|}{\eta} \rpar^{N} \lnor F \rnor_{L^{\infty}\lpar\widetilde{B_{\eta}}\rpar}.
\]
Since $F$ is holomorphic on $\widetilde{B_{\rho}}$, if $|z|\le\eta$ using Taylor's formula and Lemma \ref{le_Cauchy_ana}:

\begin{align*}
	& \lver F(z,\overline{z}) - \sum\limits_{k+l\le N-1} \partial_x^k\partial_{\overline{y}}^l F(0,0) \frac{z^k \overline{z}^l}{k!l!} \rver\\
	\le & \frac{1}{N!}\sum\limits_{k+l=N} \lnor \partial_x^k\partial_{\overline{y}}^l F \rnor_{L^{\infty}\lpar\widetilde{B_{\eta}}\rpar} |z|^N\\
	\le & \frac{|z|^N}{N!}\sum\limits_{k+l=N} k! l! \lnor F\rnor_{L^{\infty}\lpar\widetilde{B_{\rho-\epsilon}}\rpar} (\rho-\eta-\epsilon)^{-k-l}\\
	\le & (N+1) \lnor F\rnor_{L^{\infty}\lpar\widetilde{B_{\rho-\epsilon}}\rpar} \lpar\frac{|z|}{\rho-\eta-\epsilon}\rpar^N\\
\end{align*}
with $\epsilon>0$ such that $\eta<\rho-\epsilon$. We compute by induction

\[
\int_{\Cm} e^{-\frac{|z|^2}{\hbar}} z^{k} \overline{z}^l \frac{dz}{\pi\hbar} = \delta_{k=l} k!\hbar^k
\]
and then,

\[
\int_{\Cm} e^{-\frac{|z|^2}{\hbar}} \sum\limits_{k+l\le 2N-1} \partial_x^k\partial_{\overline{y}}^l F(0,0) \frac{z^k \overline{z}^l}{k!l!} \frac{dz}{\pi\hbar}
	= \sum\limits_{0\le k\le N-1} \frac{\hbar^k}{k!} \partial_x^k\partial_{\overline{y}}^l F(0,0).
\]
Combining the previous computations gives

\begin{align*}
	& \lver \int_{B_{\eta}} e^{-\frac{|z|^2}{\hbar}} Fz,\overline{z}) \frac{dz}{\pi\hbar} - \sum\limits_{0\le k\le N-1} \frac{\hbar^k}{k!} \partial_x^k\partial_{\overline{y}}^l F(0,0) \rver\\
	\le & \int_{z\in B_{\eta}} \lver e^{-\frac{|z|^2}{\hbar}} F(z,(\overline{z}) - \sum\limits_{k+l\le 2N-1} \partial_x^k\partial_{\overline{y}}^l F(0,0) \frac{z^k\overline{z}^l}{k!l!} \rver \frac{dz}{\pi\hbar}\\
	& + \int_{|z|>\eta} \lver \sum\limits_{k+l\le 2N-1} \partial_x^k\partial_{\overline{y}}^l F(0,0) \frac{z^k\overline{z}^l}{k!l!} \rver \frac{dz}{\pi\hbar} \\
	\le & C N^2 \lnor F \rnor_{L^{\infty}\lpar\widetilde{B_{\rho-\epsilon}}\rpar} \min(\rho-\epsilon-\eta,\eta)^{-2N} \int_{\Cm} e^{-\frac{|y|^2}{\hbar}} |y|^{2N} \frac{dy}{\pi\hbar}\\
	\le & C \lnor F \rnor_{L^{\infty}\lpar\widetilde{B_{\rho}}\rpar} N^2 N! \lpar \frac{\hbar}{\min(\rho-\eta,\eta)} \rpar^N
\end{align*}
\end{proof}

In the end, we get the remainder $O(e^{-c\frac{\min(\rho-\eta,\eta)}{\hbar}})$. However, we will write it as $O(e^{-c\frac{\min(\eta)}{\hbar}})$, since we are more interested in the asymptotic $\eta \searrow 0$ than $\eta \nearrow \rho$. The optimisation in $N$ that yields an exponentially small remainder may appear as a calculus trick, but it will actually be essential for the definition of analytic symbols.

Recall that a real valued function $W$ on an open set of $\Cm^d$ is plurisubharmonic, respectively strictly plurisubharmonic, if the matrix $\lpar\partial_{x_j}\partial_{\overline{x_k}}W\rpar_{1\le j,k\le d}$ is positive semidefinite, respectively positive definite.

\begin{definition}[Function state spaces]
Let $d\in\Nm$, we fix $U$ an open set of $\Cm^d$, and $\Phi_M$ an analytic real-valued plurisubharmonic function on $U$. Let $W$ be a weight, meaning an analytic real-valued function, we suppose that $\Phi_M-W$ is strictly plurisubharmonic. Then, the modified Bargmann space

\[
\statesp[\Phi_M,W](U) = L^2(U,e^{-\frac{W(z)}{\hbar}}dz) \bigcap \lacc e^{-\frac{\Phi_M}{2\hbar}}f /\; f \text{ holomorphic on } U \racc.
\]
is a Hilbert space for the $L^2(U,e^{-\frac{W(z)}{\hbar}}dz)$ scalar product, see for instance \cite{roub20}.

In practice, we will work with local coordinates on a manifold $M$, and $\Phi_M$ will be a Kähler potential. If there is no ambiguity on the Kähler potential, we will omit it and write $\statesp[W](U)$ instead. If additionally $W=0$, notice that $\statesp[0](U)=\barg(U)$ is the usual Bargmann space.

Recall that $m\in C^{\infty}(\Cm)$ is an order function if $m\ge 1$ and there exists $C>0,\, N\in\Nm^*$ such that $m(x) \le C(1+|x-y|^2)^{\frac{N}{2}} m(y)$ for all $(x,y)\in\Cm^2$. If $M=\Cm$, $W$ is a weight, and $m$ is an order function, then we also consider the following spaces

\[
\statesp[W,m] = L^2(\Cm,m(z)^2 e^{-\frac{W(z)}{\hbar}}dz) \bigcap \lacc e^{-\frac{|z|^2}{2\hbar}}f /\; f \text{ holomorphic on } U \racc.
\]

\end{definition}

Actually, $\Phi_M$ can be seen as a Kähler potential because $\statesp[\Phi_M,0]$ is linked to the space of states on the corresponding manifold. First, recall the definition of Toeplitz operators on a manifold. Let $M$ be a real-analytic, quantisable, Kähler manifold, $L\rightarrow M$ a prequantised line bundle, and $\Hc_N = H^0(M,L^{\otimes N})$ the quantum space on $M$ for $N\in\Nm^*$, meaning the space of holomorphic sections of $L^{\otimes N}\rightarrow M$. Then, for $f\in L^{\infty}(M)$ the associated Toeplitz operator is $T_N(f):\; u \mapsto \Pi_N(fu)$, with $\Pi_N$ the projection from $L^2(M)$ to $\Hc_N$. Now, if $U$ is an open set of $M$ such that we have the local chart $l : U \rightarrow V\subset \Cm^d$, and a Kähler potential $\Phi_M$ on $V$, then

\begin{gather*}
\Hc_N \hookrightarrow \statesp[\Phi_M,0](V)\\
s \mapsto s\circ l^{-1}|_V,
\end{gather*}
is an injection, for $\hbar = \frac{1}{N}$. We will also see in Proposition \ref{prop_local_berg} that the Bergman kernel is an analytic symbol, which means that there exists a symbol $B$ such that locally $T_N(f)s = T_{\hbar}(Bf)\lpar s\circ l^{-1}|_V\rpar + O(e^{-\frac{c}{\hbar}})\lnor s\rnor_{L^2(U)}$.

Moreover, we will quickly reduce the problem to $\Cm^d$, in which case we choose $\Phi_M = |z|^2$. For that reason, we consider data on $\Cm^d$ from now on in this Section. For clarity, we recall the Kähler structure of $\Cm$, and the corresponding quantum space.

\begin{example}[Kähler structure of $\Cm$]
\label{ex_kahler_Cm}
Let $M = (\Rm^2,J,\omega)$, with complex structure $J$ and symplectic form $\omega$ both represented by the matrix

\[
J = \begin{pmatrix}
 0 & -1\\
 1 & 0
\end{pmatrix}.
\]
The complex tangent bundle is then $TM\otimes\Cm = T^{1,0}M\oplus T^{0,1}M = Vect \lpar\frac{(1,i)}{\sqrt{2}}\rpar \oplus Vect \lpar\frac{(1,-i)}{\sqrt{2}}\rpar$. The metric is given by $g(X,Y) = \omega(-JX,Y) = X^{\intercal} Y$. They correspond to the coordinates $z=\frac{1}{\sqrt{2}}(1,-1)$ and $\overline{z} = \frac{1}{\sqrt{2}}(1,1)$. Then, a prequantum line bundle is given by $L=(\Rm^2\otimes \Cm,\nabla,h)$ with metric $h(u,v) = u\overline{v}$. Since

\[
\omega = dx \wedge d\xi = \lpar \frac{dz+d\overline{z}}{\sqrt{2}} \rpar\wedge\lpar \frac{dz-d\overline{z}}{\sqrt{2}i} \rpar = idz\wedge d\overline{z} = i\partial\overline{\partial} |z|^2,
\]
necessarily $\nabla$ is of the form $d+azd\overline{z}-b\overline{z}dz$ with $a+b=1$, but it also satisfies $d(h(a,b)) = h(\nabla a,b) + h(a,\nabla b)$ so $a=b=1$. Hence, $\nabla^{0,1} = \overline{\partial} + \frac{z}{2}d\overline{z}$, and the quantum space is
\[
\Hc_N = L^2 \cap \lacc f(z) e^{-N\frac{|z|^2}{2}} /\; f \text{ holomorphic } \racc = \barg
\]
for $\hbar = \frac{1}{N}$.
\end{example}

Our main concern is to be able to quantise a symplectomorphism $\widetilde{\kappa}$ of $\widetilde{M}$. With this in mind, we consider $\mathfrak{I}$ a section of $\widetilde{L}\times \widetilde{\overline{L}} \rightarrow \widetilde{M}\otimes\widetilde{\overline{M}}$ which will be the operator kernel. It has to be holomorphic, thus locally of the form $\mathfrak{I} = e^{-\widetilde{\Phi_M}(x,\overline{y})-\widetilde{\Phi_M}(z,\overline{v})+\phi(x,\overline{v})}$ with $\Phi_M$ a Kähler potential, and needs to be constant along $\grph(\widetilde{\kappa})$, which is equivalent to

\[
\lacc \lpar\widetilde{\nabla}\otimes \id + \id\otimes\widetilde{\overline{\nabla}}\rpar \mathfrak{I} = 0 \racc = \grph\lpar\widetilde{\kappa}\rpar,
\]
since $\grph\lpar\widetilde{\kappa}\rpar$ is a Lagrangian. As said before, we consider that we can reduce the problem to $\Cm^{2d}$, hence the following definition.

\begin{definition}
Let $U$,$V$ be neighbourhoods of $0$ in $\Cm^{2d}$, and $\widetilde{\kappa}: U\rightarrow V$ be a symplectormophism. Let $W$,$B$ be neighbourhoods of $0$ in $\Cm^{2d}$ and $\Cm^d$ respectively, we say that a holomorphic function $\phi : W \rightarrow B$ is associated to $\widetilde{\kappa}$ if

\begin{multline}
\label{eq_phase_symp}
\grph(\widetilde{\kappa}) = \lacc (x,\overline{y},z,\overline{v}) \in U\times V \middle/\; \right.\\
\left. \partial_x\phi(x,\overline{v}) = \partial_1\widetilde{\Phi_M}(x,\overline{y}),\; \partial_{\overline{v}}\phi(x,\overline{v}) = \partial_2\widetilde{\Phi_M}(z,\overline{v}) \racc.
\end{multline}
These are the so-called Hamilton-Jacobi equations. In the other way, if $\phi : W \rightarrow B$ is a holomorphic function, then there exists an associated symplectomorphism satisfying \eqref{eq_phase_symp}, we write it $\widetilde{\kappa_{\phi}}$. Notice that if $\Phi_M(x) = |z|^2$ then \eqref{eq_phase_symp} is equivalent to

\[
\widetilde{\kappa}(x,\partial_x\phi(x,\overline{y})) = \lpar\overline{\partial_y}\phi(x,\overline{y}),\overline{y}\rpar,
\]
for all $(x,\overline{y})\in W$. Furthermore $\partial_x\overline{\partial_y}\phi(x,\overline{y}) \in Gl(\Cm^d)$, since $\widetilde{\kappa}$ is a symplectomorphism.
\end{definition}

The function $\phi$ does not always exist, \cite{dele25} gives a sufficient condition on the Lagrangian $\grph(\kappa)$. For us, it will be equivalent to \eqref{eq_complex_quad}. If it exists, the good operator candidate is

\[
I_{\phi}(a)u(x) = \int_{\Cm} e^{\frac{2\phi(x,\overline{y})-\Phi_M(x)-\Phi_M(y)}{2\hbar}} a(x,y) u(y) \frac{dy}{\pi\hbar}
\]
with $a \in L^{\infty}(\Cm\otimes\overline{\Cm})$. In order to prove that it is well-defined, we need to make hypothesis and prove a few results on the function $\phi$.

\begin{definition}[Phase functions]
\label{def_phase}
Let $W$ be a real-valued analytic function on $U$, and $\phi:(x,\overline{y})\mapsto\phi\lpar x,\overline{y}\rpar$ be a holomorphic function on $U\times\overline{U}$ such that $\partial_x\overline{\partial_y}\phi \in Gl(\Cm^d)$ everywhere. We suppose that $\Phi_M-W$ is strictly plurisubharmonic. $\phi$ is called a phase function for $W$ if for all $x\in M$ the map

\begin{align*}
\overline{\Cm}^d & \rightarrow \Rm\\
\overline{y} & \mapsto 2\Re \phi(x,\overline{y}) -\Phi_M(y)+W(y),
\end{align*}
has a unique critical point $\overline{y_0(x)}\in\overline{\Cm}^d$ of signature $(0,-2d)$. We write

\begin{align*}
W_l : \Cm^d \rightarrow & \Rm\\
x \mapsto & \cv_{y\in\Cm^d} \lbra 2\Re (\phi(x,\overline{y})) -\Phi_M(\overline{y})+W(\overline{y}) \rbra - \Phi_M(x)\\
	& = 2\Re\lpar \phi\lpar x,\overline{y_0(x)}\rpar \rpar -\Phi_M\lpar\overline{y_0(x)}\rpar+W\lpar\overline{y_0(x)}\rpar- \Phi_M(x)
\end{align*}
and we call it the transformed weight by $\phi$. By definition, for $y$ near $y_0(x)$

\begin{equation}
\label{eq_ineq_phase}
2\Re\phi(x,\overline{y}) -\Phi_M(y)-\Phi_M(x)+W(\overline{y})-W_t(x) \le -C\lver y-y_0(x) \rver^2.
\end{equation}
\end{definition}

We use this notation and vocabulary, as $\Phi_M+W_l$ can be seen as a deformed Legendre transformation of $\Phi_M-W$.

\begin{lemma}
\label{le_transformed_potential}
If $\phi$ is a phase function for $W$ then the transformed weight $W_l$ is such that $\Phi_M+W_l$ is strictly plurisubharmonic.
\end{lemma}

\begin{proof}
The proof is merely an adaptation of \cite{sjos82} Lemma 3.2. We Look locally near a point $x_0$ which can be taken as $0$ by a change of variable, thus we suppose the functions to be quadratic. Writing $\phi(x,y) = \frac{1}{2}\sum a_{j,k}\overline{y_j}\overline{y_k}+b_{j,k}\overline{y_j}x_k$ implies that $\Re\phi(x,\overline{y}) = \frac{1}{2}Y^{\intercal}N_aY + Y^{\intercal}N_bX$ where

\begin{align*}
N_a & = \begin{pmatrix}
\ddots & \vdots & \iddots\\
\hdots & 
	\begin{bmatrix}
	\Re(a_{j,k}) & -\Im(a_{j,k})\\
	-\Im(a_{j,k}) & -\Re(a_{j,k})
	\end{bmatrix}
& \hdots\\
\iddots & \vdots & \ddots\\
\end{pmatrix}
\\
X & = \begin{pmatrix}
\vdots\\
\Re(x_j)\\
\Im(x_j)\\
\vdots
\end{pmatrix}
,\;
Y = \begin{pmatrix}
\vdots\\
\Re(y_j)\\
-\Im(y_j)\\
\vdots
\end{pmatrix}.
\end{align*}
We forget the $x^2$ part of $\phi$ as it will not appear in the computations, and write $\Phi_M-W(y) = Y^{\intercal}MY$ with $M$ definite positive by hypothesis, thus

\[
\nabla_Y\lpar 2\Re\phi(x,\overline{y})-\Phi_M(y)+W(y) \rpar = 2Y^{\intercal} N_a + 2X^{\intercal} N_b -2Y^{\intercal}M.
\]
The phase condition is then equivalent to $M-N_a$ positive definite, the critical point is

\[
Y_c(x) = \lpar M-N_a \rpar^{-1} N_b X
\]
and the transformed weight $W_l$ is such that

\[
W_l(x)+\Phi_M(x) = X^{\intercal} N_b^{\intercal} \lpar M-N_a \rpar^{-1} N_b X + 2\Re\phi(x,0).
\]
Since $(b_{j,k}) = \partial_x\overline{\partial_y}\phi \in Gl(\Cm^d)$, we also have $N_b \in Gl(\Rm^{2d})$, and since $M-N_a$ is positive definite, so is its inverse. Hence, $N_b^{\intercal} \lpar M-N_a \rpar^{-1} N_b$ is positive definite, and so is $\partial\overline{\partial}\lpar W_l+\Phi_M\rpar$.
\end{proof}

We compute an explicit transformed weight, that will be useful in Theorem \ref{th_FIO_normal}.

\begin{lemma}
\label{le_phase_linear}
Let  $\Phi_M(x) = |z|^2$, and $\widetilde{\kappa}$ be a linear symplectomorphism, represented by the matrix $\begin{pmatrix} a & b\\ c & d \\ \end{pmatrix}$. Then, there exists a function associated to $\widetilde{\kappa}$ if and only if $d\neq 0$. Moreover, $\phi$ is a phase function for the weight $W=0$ if and only if $|d|>|b|$, and if it satisfied, the transformed weight is 
\[
W_l(x) = \frac{1-(|d|^2-|b|^2)}{|d|^2-|b|^2}|x|^2+\Re \lpar\frac{a\overline{b}-\overline{d}c}{|d|^2-|b|^2}x^2\rpar.
\]
\end{lemma}

\begin{proof}
Here, \eqref{eq_phase_symp} is equivalent to $d \phi(x,\overline{v}) = -\frac{c}{2}x^2 + x\overline{v} + \frac{b}{2}\overline{v}^2 + d\phi(0,0)$, thus the condition, and when it exists $\phi(x,\overline{v}) = -\frac{c}{2d}x^2 + \frac{1}{d}x\overline{v} + \frac{b}{2d}\overline{v}^2 + \phi(0,0)$. We take $\phi(0,0) = 0$, as it will only change the result by a constant. The Hessian of $\lpar \Re(y),-\Im(y) \rpar \mapsto 2\Re \phi(x,\overline{y}) -|y|^2$ is
\[
2\begin{pmatrix}
-1+\Re\lpar\frac{b}{d}\rpar & -\Im\lpar\frac{b}{d}\rpar\\
-\Im\lpar\frac{b}{d}\rpar & -1 - \Re\lpar\frac{b}{d}\rpar
\end{pmatrix}
\]
and its determinant is $4\lpar 1-\lver \frac{b}{d} \rver^2\rpar$. Hence, $\phi$ is a phase for $0$ if and only if $|d|>|b|$, and when it is satisfied, $y_0(x) = \frac{\overline{d}x+b\overline{x}}{|d|^2-|b|^2}$ and

\begin{align*}
W_l(x)
	= & \Re \lpar-\frac{c}{d}x^2 + \frac{2}{d}x\overline{y_0}(x) + \frac{b}{d}\overline{y_0}(x)^2\rpar - |y_0(x)|^2\\
	= & \frac{1-(|d|^2-|b|^2)}{|d|^2-|b|^2}|x|^2+\Re \lpar\frac{a\overline{b}-\overline{d}c}{|d|^2-|b|^2}x^2\rpar.
\end{align*}
\end{proof}

\begin{example}
\label{ex_symp}
Let us see the transformed weights of specific symplectormophisms for $\Phi_M(y) = |y|^2$.
\begin{itemize}
\item If $\kappa$ is a real symplectomorphism on $\Rm^2$ given by

	\[\begin{pmatrix}
	\alpha & \beta\\
	\gamma & \delta
	\end{pmatrix},\]
	then its analytic extension is

	\[
	\widetilde{\kappa} = \frac{1}{2} \begin{pmatrix}
	\alpha+\delta+i(\gamma-\beta) & \alpha-\delta+i(\gamma+\beta)\\
	\alpha-\delta-i(\gamma+\beta) & \alpha+\delta-i(\gamma-\beta)
	\end{pmatrix} = \begin{pmatrix}
	a & b\\
	c & d
	\end{pmatrix},
	\]
	with $\widetilde{\kappa}(z,\overline{z}) = \kappa(\Re(z),\Im(z))$ for all $z\in\Cm$. Hence, $|d|^2-|b|^2=1$ and $b\overline{d}-a\overline{c}=0$, and the transformed weight of $0$ is itself.
\item If $\widetilde{\kappa} = \textup{diag}(\theta,\theta^{-1})$, the corresponding function $\phi$ is a phase for any weight $W$ and the transformed $W_l$ is such that $\partial W_l(x) = \theta \overline{\lpar\id -2\partial W\rpar^{-1}(\theta x)} - \overline{x}$.
\item We will compute in Theorem \ref{th_FIO_normal} the transformed weight of $0$ for the simplectomorphism from Section \ref{subsec_quad_symbol}.
\end{itemize}
\end{example}

Let $\Phi_M= |z|^2$, and $W$ be a weight such that $I_d + \partial\overline{\partial}W$ is definite positive, then

\[
W(z) +\Phi_M(z) = B + \Re{bz} + \frac{1+A}{2}|z|^2 + \Re(az^2) + O(|z|^3)
\]
where $A>-1$. Then, the constant $B$ is not relevant since we can multiply $\statesp[W]$ by $e^{\frac{2B}{\hbar}}$ to get rid of it, and it won't change the operators. Furthermore, we will mainly consider weights written near $0$, and we will suppose that $\partial W(0)=0$, this property being invariant by the transformations we will consider, thus we get rid of $b$. Consequently, all the information is given by the Lagrangian

\[
\Lambda_{W} =\lacc \lpar x,\overline{x}+\partial_x W(x) \rpar /\; x\in \Cm^d \racc,
\]
which by hypothesis, contains $(0,0)$. This Lagrangian makes the computation of $W_l$ easier.

\begin{lemma}
\label{le_link_weigts}
Let $W$ a weight, and $\phi$ a phase function associated to $W$ with transformed weight $W_l$, then $\widetilde{\kappa}(\Lambda_{W_l}) = D\overline{\Lambda_{-W}}$, where
$
D = \begin{pmatrix}
	0 & I_d\\
	I_d & 0
	\end{pmatrix}
$.
\end{lemma}

\begin{proof}
Let $x\in\Cm^d$, with the notations of Definition \ref{def_phase}
\begin{align*}
\widetilde{\kappa}\lpar x,\partial W_l(x)+\overline{x} \rpar
	= & \widetilde{\kappa}\lpar x,\partial_x\phi\lpar x,\overline{y_0(x)}\rpar \rpar\\
	= & \lpar \overline{\partial_y}\phi\lpar x,\overline{y_0(x)}\rpar,\overline{y_0(x)} \rpar\\
	= & D \lpar \overline{y_0(x)},y_0(x)-\overline{\partial}W\lpar\overline{y_0(x)}\rpar \rpar,
\end{align*}
so $\widetilde{\kappa}(\Lambda_{W_l}) \subset D\overline{\Lambda_{-W}}$, which implies equality since the spaces are Lagrangians.
\end{proof}

\begin{proposition}
\label{prop_phase_inverse}
Let $\phi$ be a phase for $W$ with transformed $W_l$, and $\widetilde{\kappa_{\phi}}$ the associated symplectomorphism. Then, $\phi^{\dag}:(x,\overline{y})\mapsto \overline{\phi(y,\overline{x})}$ is a phase for $-W_l$ with transformed $-W+c$ where $c\in\Cm$, and the associated symplectomorphism is $\widetilde{\kappa_{\phi^{\dag}}} = \overline{ D \widetilde{\kappa_{\phi}}^{-1} D} : (x,\overline{y}) \mapsto D\overline{\widetilde{\kappa}^{-1}D(\overline{x},y)}$. To summarise

\begin{align*}
W_l(x) +\Phi_M(x) & = \cv_{\overline{y}\in\Cm^d} \lpar 2\Re \phi(x,\overline{y}) -\Phi_M(y) + W(y) \rpar,\\
-W(x)+c +\Phi_M(x)& = \cv_{\overline{z}\in\Cm^d} \lpar 2\Re  \phi^{\dag}(x,\overline{z}) -\Phi_M(z) - W_l(z) \rpar\\
			& = \cv_{\overline{z}\in\Cm^d} \lpar 2\Re  \phi(z,\overline{x}) - \Phi_M(z)-W_l(z) \rpar\\
			& = 2\Re \phi(zl(x),\overline{x}) - \Phi_M(zl(\overline{x}))-W_l(zl(\overline{x})).
\end{align*}
Besides, there is a unique $z_0(x) \in\Cm^d$ on which it is attained, with $z_0(y_0(x)) = x$, therefore both functions $y_0$ and $z_0$ are diffeomorphisms.

\end{proposition}

\begin{proof}
We can apply Lemma \ref{le_transformed_potential} to $\overline{y}\mapsto 2\phi(x,\overline{y})-\Phi_M(y)+W(y)-\phi(x,0)$ for $x$ small enough, which gives that $z\mapsto W_l(z)+\Phi_M(z)-2\Re\phi(z,0)$ has a positive definite Hessian. In other words, for $\overline{x}$ small enough $z\mapsto 2\Re\phi(z,\overline{x}) -W_l(z)-\Phi_M(z)$ has a negative definite hessian, so it admits a unique critical point of signature $(0,-2d)$. According to Lemma \ref{le_link_weigts}

\begin{align*}
D\widetilde{\kappa_{\phi}}^{-1} D \overline{(x,\partial_x \phi^{\dag}(x,\overline{y}))} = & D\widetilde{\kappa_{\phi}}^{-1} (\partial_{\overline{x}} \phi(y,\overline{x}),\overline{x}) = D(y,\partial_y\phi(y,\overline{x}))\\
	= & (\partial_y\phi(y,\overline{x}),y) = \overline{(\overline{\partial_y}\phi^{\dag}(x,\overline{y}),\overline{y})},
\end{align*}
thus, the symplectomorphism associated to $\phi^{\dag}$ is $\overline{ D \widetilde{\kappa_{\phi}}^{-1} D}$. Now, we denote

\[
-W_2(y) = \cv_{z\in\Cm^d} \lpar 2\Re \phi(z,\overline{x}) - \Phi_M(z) - W_l(z) \rpar,
\]
and we want to prove that $W_2-W$ is constant. Since $\phi^{\dag}$ is associated to $\overline{ D \widetilde{\kappa_{\phi}}^{-1} D}$, and according to Lemma \ref{le_link_weigts}
\[
\overline{D\widetilde{\kappa_{\phi}}^{-1} D \overline{\Lambda_{-W_2}}} = \overline{D \Lambda_{W_l}} = \overline{D \kappa_{\phi}^{-1} D \overline{\Lambda_{-W}}}
\]
meaning $\Lambda_{-W}=\Lambda_{-W_2}$. Hence, $W_2-W$ is a real-valued anti-holomorphic function, which means it is constant. Then, for all $x\in\Cm^d$, by definition of $y_0(x)$,

\begin{align*}
\widetilde{\kappa_{\phi}}^{-1} \widetilde{\kappa_{\phi}}\lpar x,\overline{x}+\partial W_l(x) \rpar
	= & \widetilde{\kappa_{\phi}}^{-1} \widetilde{\kappa_{\phi}}\lpar x,\partial_x\phi\lpar x,\overline{y_0(x)} \rpar \rpar\\
	= & \widetilde{\kappa_{\phi}}^{-1} \lpar \overline{\partial_y}\phi\lpar x,\overline{y_0(x)} \rpar,\overline{y_0(x)} \rpar\\
	= & \widetilde{\kappa_{\phi}}^{-1} \lpar \overline{\partial_y\phi^{\dag}\lpar y_0(x),\overline{x} \rpar},\overline{y_0(x)} \rpar\\
	= & \overline{\lpar\overline{\widetilde{\kappa}^{-1}}\rpar D \lpar y_0(x), \partial_y\phi^{\dag}\lpar y_0(x),\overline{x} \rpar \rpar}\\
	= & D\overline{\lpar \partial_{\overline{x}}\phi^{\dag}\lpar y_0(x),\overline{x} \rpar, \overline{x} \rpar}\\
	= & \lpar x, \overline{\partial_y\phi^{\dag}\lpar y_0(x),\overline{x} \rpar}\rpar
\end{align*}
meaning $\partial_{\overline{x}}\phi^{\dag}\lpar y_0(x),\overline{x} \rpar = x+ \overline{\partial}W_l(x)$, thus $x$ is the unique critical point $z_0(y_0(x))$, so $z_0 \circ y_0 = \id$. A similar computation gives $y_0 \circ z_0 = \id$.
\end{proof}

One might notice that the adjoint of $I_{\phi}(a)$ is equal to

\[
I_{\phi^{\dag}}(a^{\dag}) : \statesp[-W_l] \rightarrow \statesp[-W]
\]
where $a^{\dag}(x,\overline{y}) = \overline{a(y,\overline{x})}$. We will not consider adjoints of FIOs in this article, although we will use the fact that $y_0$, $z_0$ are diffeomorphisms to do changes of variables in some proofs.

\begin{theorem}
\label{th_def_FIO}
Let $\Phi_M$ be a plurisubharmonic function, $W$ a weight, $\phi$ a phase function for $W$, and $W_l$ the transformed weight. Let $a_{\hbar}\in L^{\infty}(\Cm^{2d})$, then there exists $\rho,\sigma>0$ such that for all $u\in\statesp[W](B_{\rho})$ the function

\[
I_{\hbar,\phi}(a_{\hbar})u : x\mapsto \int_{B_{\rho}} e^{\frac{2\phi(x,y)-\Phi_M(y)-\Phi_M(x)}{2\hbar}} a_{\hbar}(x,y) u(y) \frac{dy}{(\pi\hbar)^d}
\]
is well-defined and lies in $\statesp[W_l](B_{\sigma})$. We call the operators $u \mapsto I_{\hbar,\phi}(a_{\hbar})u$ complex FIOs, and we write them $I_{\phi}(a)$ for simplicity. More precisely, there exists $C(d,\phi)>0$ such that for all $u\in\statesp[W](B_{\rho})$
\[
\lnor I_{\phi}(a) u \rnor_{\statesp[W_l](B_{\sigma})} \le C(d,\phi) \lnor a \rnor_{L^{\infty}(\Cm^{2d})} \lnor u \rnor_{\statesp[W](B_{\rho})}.
\]
Furthermore, if $\rho'\in ]0,\rho[$, there exists $c>0$ and $\sigma'\in ]0,\sigma[$ such that for all  $u\in\statesp[W](B_{\rho})$

\begin{align*}
	& \lnor \int \chi_{\rho'\le|z|\le\rho} e^{\frac{2\phi(x,y)-\Phi_M(y)-\Phi_M(x)}{2\hbar}} a_{\hbar}(x,y) u(y) \frac{dy}{(\pi\hbar)^d} \rnor_{\statesp[W_l](B_{\sigma'})}\\
	\le & C(d,\phi) \lnor a \rnor_{L^{\infty}(\Cm^{2d})} \lnor u \rnor_{\statesp[W](B_{\rho})} e^{-\frac{c}{\hbar^2}}.
\end{align*}

If $\phi$  is such that inequality \eqref{eq_ineq_phase} is satisfied everywhere, for example if $\phi$, $\Phi_M$ and $W$ are quadratic, then $I_{\phi}(a)$ is well-defined for any $\rho$ and $\sigma$. Taking the integral over $\Cm^d$, $I_{\phi}(a)$ is bounded from $\statesp[W](\Cm^d)$ to $\statesp[W_l](\Cm^d)$.
\end{theorem}

\begin{proof}
We will prove it by taking it down to Schur's test. Let $u\in\statesp[W](B_{\rho})$, then

\begin{align*}
I_{\phi}(a)u(x) = & e^{\frac{W_l}{2\hbar}}\int_{\Cm^d} K(x,y) r(y) dy \text{, with}\\
r = & e^{\frac{-W}{2\hbar}} \chi_{B_{\rho}} u,\\
K(x,y) = & \chi_{B_{\sigma}}(x) \chi_{B_{\rho}}(y) e^{\frac{2\phi(x,\overline{y})-\Phi_M(y)-\Phi_M(x)+W(y)-W_l(x)}{2\hbar}} a_{\hbar}(x,y) \frac{1}{(\pi\hbar)^d}.
\end{align*}
Thus, it is equivalent to prove that $s \mapsto \int_{\Cm^d} K(x,y) r(y) dy$ is bounded on $L^2(\Cm^d)$, which is equivalent to $\sup_x \int_{\Cm^d} \lver K(x,y)\rver dy$ and $\sup_y \int_{\Cm^d} \lver K(x,y)\rver dx$ bounded. By hypothesis, there exists $\rho,\sigma,C >0$ such that for all $\lnor x\rnor\le\sigma,\lnor y\rnor\le\rho$,

\[
2\Re\phi(x,y)-\Phi_M(y)-\Phi_M(x)+W(y)-W_l(x) \le -C|y-y_0(x)|^2
\]
hence,

\begin{align*}
\int_{\Cm^d} \lver K(x,y)\rver dy
	\le & \lnor a \rnor_{L^{\infty}(\Cm^{2d})} \int_{\Cm^d} e^{-C\frac{|y-y_0(x)|^2}{2\hbar}} \frac{dy}{(\pi\hbar)^d} = \lnor a \rnor_{L^{\infty}(\Cm^{2d})} C^{-d}\\
\int_{\Cm^d} \lver K(x,y)\rver dx
	\le & \lnor a \rnor_{L^{\infty}(\Cm^{2d})} \int_{B_{\sigma}} e^{-C\frac{|y-y_0(x)|^2}{2\hbar}} \frac{dx}{(\pi\hbar)^d}\\
	\le & \lnor a \rnor_{L^{\infty}(\Cm^{2d})} \sup_{y_0(B_{\sigma})} \lnor \det \nabla y_0^{-1} \rnor \int_{y_0(B_{\sigma})} e^{-C\frac{|x|^2}{\hbar}} \frac{dx}{(\pi\hbar)^d}\\
	\le & \lnor a \rnor_{L^{\infty}(\Cm^{2d})} \sup_{y_0(B_{\sigma})} \lnor \det \nabla y_0^{-1} \rnor C^{-d}.
\end{align*}
This gives the following inequality on $I_{\phi}(a)$

\[
\lnor I_{\phi}(a) u \rnor_{\statesp[W_l](B_{\sigma})} \le C^{-d} \lnor a \rnor_{L^{\infty}(\Cm^{2d})} \sqrt{\sup_{y_0(B_{\sigma})} \lnor \det \nabla y_0^{-1} \rnor} \lnor u \rnor_{\statesp[W](B_{\rho})}.
\]
Now, if $\rho'\in ]0,\rho[$, there exists $\sigma'\in ]0,\sigma[$ and $\epsilon>0$ such that for all $|x|<\sigma'$, $|y_0(x)|<\rho'-\epsilon$. We do a Schur's test again but now

\[
K(x,y) = \chi_{B_{\sigma'}}(x) \chi_{B_{\rho}\backslash B_{\rho'}}(y) e^{\frac{\phi(x,y)-\Phi_M(y)-\Phi_M(x)+W(y)-W_l(x)}{\hbar}} a_{\hbar}(x,y) \frac{1}{(\pi\hbar)^d}
\]
which gives

\begin{align*}
\sup_{x\in \Cm^d} \int_{\Cm^d} \lver K(x,y)\rver dy
	\le & \lnor a \rnor_{L^{\infty}(\Cm^{2d})} \sup_{z\in B_{\rho'-\epsilon}} \int_{\rho'<|y|} e^{-C\frac{|y-z|^2}{\hbar}} \frac{dy}{(\pi\hbar)^d}\\
	\le & \lnor a \rnor_{L^{\infty}(\Cm^{2d})} C^{-d} e^{-C\frac{\epsilon^2}{\hbar}}\\
\sup_{y\in \Cm^d}\int_{\Cm^d} \lver K(x,y)\rver dx
	\le & \lnor a \rnor_{L^{\infty}(\Cm^{2d})} \sup_{\rho'<\lnor y\rnor<\rho} \int_{B_{\sigma'}} e^{-C\frac{|y-y_0(x)|^2}{\hbar}} \frac{dx}{(\pi\hbar)^d}\\
	\le & \lnor a \rnor_{L^{\infty}(\Cm^{2d})} \sup_{y_0(B_{\sigma})} \lnor \det \nabla y_0^{-1} \rnor \sup_{\rho'<\lnor y\rnor} \int_{B_{\rho'-\epsilon}} e^{-C\frac{|y-x|^2}{\hbar}} \frac{dx}{(\pi\hbar)^d}\\
	\le & \lnor a \rnor_{L^{\infty}(\Cm^{2d})} \sup_{y_0(B_{\sigma})} \lnor \det \nabla y_0^{-1} \rnor C^{-d} e^{-C\frac{\epsilon^2}{\hbar}}.
\end{align*}
Therefore,

\begin{align*}
	& \lnor \int \chi_{\rho'\le|z|\le\rho} e^{\frac{\phi(x,y)-\Phi_M(y)-\Phi_M(x)}{\hbar}} a_{\hbar}(x,y) u(y) \frac{dy}{(\pi\hbar)^d} \rnor_{\statesp[W_l](B_{\sigma'})}\\
	\le & C^{-d} \sqrt{\sup_{y_0(B_{\sigma})} \lnor \det \nabla y_0^{-1} \rnor} \lnor a \rnor_{L^{\infty}(\Cm^{2d})} \lnor u \rnor_{\statesp[W](B_{\rho})} e^{-C\frac{\epsilon^2}{\hbar^2}}.
\end{align*}
\end{proof}

We now define the class of symbols that we will define in the rest of the article.

\begin{definition}[\cite{bout67}, \cite{bout72}]
Fix $\hbar>0$, and consider a formal symbol $\sum \hbar^k f_k$ on $\Cm^d$, we say that it is analytic if for every $k$, there exists $C,A >0$ such that the holomorphic extension $\widetilde{f_k}$ of $f_k$ satisfies on a neighbourhood of $0$

\[
\lver\widetilde{f_k}\rver \le C A^k k!.
\]
A function $f : B_{\rho} \rightarrow \Cm$ is an analytic symbol if it is asymptotically equivalent to an analytic formal symbol $\sum \hbar^k f_k$, meaning there exists $C,A>0$ such that $\lver\widetilde{f_k}\rver \le C A^k k!$ on $\widetilde{B_{\rho}}$ for all $k\in\Nm$, and

\[
f(x) = \sum\limits_{0\le k\le \frac{1}{A\hbar}} \hbar^k f_k(x)
\]
which we write $f \sim \sum \hbar^k f_k$. We denote $S(B_{\rho})$ the space of analytic symbols on $B_{\rho}$.

If $M=\Cm$, a function $f$ is a global analytic symbol if it is uniformly equivalent to an analytic formal symbol, meaning that there exists $\rho$ such that for any compact $K\subset \Cm$, there exists $A = A(K)>0$ such that $\widetilde{f}(x,\overline{y}) = \sum\limits_{0\le k\le \frac{1}{A\hbar}} \hbar^k \widetilde{f_k}$ with $\widetilde{f_k}$ analytic on $K + \widetilde{B_{\rho}}$ and $\lver \widetilde{f_k}(x,\overline{y}) \rver \le A^{k+1} k!$ for all $(x,\overline{y}) \in K+\widetilde{B_{\rho}}$. Let $m$ be an order function, then a global analytic symbol $f$ is of order $m$ if there exists $C>0$ such that $|\widetilde{f}(x,\overline{y})| \le Cm(x)$ for all $(x,\overline{y})\in \diag+\widetilde{B_{\rho}}$. We write $S(\Cm,m)$ the space of global analytic symbols of order $m$.
\end{definition}

\begin{remark}
Let $a_{\hbar} = \sum\limits_{k=0}^{\frac{c}{A\hbar}} \hbar^k a_k(x): B_{\rho} \rightarrow \Cm$ be an analytic symbol, $\Phi_M$ be a plurisubharmonic function, $W,W_l$ be weights, and $\phi$ be a phase like in Theorem \ref{th_def_FIO}. Denote $\rho,\sigma$ the constants given by the Theorem. Then

\[
I_{\phi}(a_{\hbar}): u \mapsto \lpar x\mapsto \int_{B_{\rho}} e^{\frac{2\phi(x,\overline{y})-\Phi_M(y)-\Phi_M(x)}{2\hbar}} f(x,y)u(y) \frac{dy}{(\pi\hbar)^d} \rpar
\]
is still well-defined from $\statesp[W](B_{\rho})$ to $\statesp[W_l](B_{\sigma})$.
\end{remark}

The purpose of the analytic symbols is to enable the use of analytic symbolic calculus over FIOs. For instance, it is possible to compute any function of the form $I_{\phi}(a_{\hbar})\lpar e^{-\frac{\Phi_M(y)}{2\hbar}} e^{\frac{\psi(y)}{\hbar}} v_{\hbar}(y)\rpar(x)$ with a WKB method, using contour deformations. It has been done before in \cite{lef14a,dele25}.

\begin{lemma}
\label{le_WKB}
Let $\Phi_M$, $W$ be a plurisubharmonic function and a weight, $\phi$ a phase, $a_{\hbar}$ an analytic symbol, $\psi$, $v_{\hbar}$ holomorphic functions, and suppose $e^{-\frac{\Phi_M(y)}{2\hbar}}v_{\hbar}(y) \in\barg$, then

\[
I_{\phi}(a_{\hbar})\lpar e^{-\frac{\Phi_M(y)}{2\hbar}} e^{\frac{\psi(y)}{\hbar}} v_{\hbar}(y)\rpar(x)
	= e^{-\frac{\Phi_M(x)}{2\hbar}} e^{\frac{\sigma(x)}{\hbar}} w_{\hbar}(x)
\]
with $\sigma$ a holomorphic function and $e^{-\frac{\Phi_M(y)}{2\hbar}}w_{\hbar}(y) \in\barg$. Moreover, for all $x$ there exists $y_c(x),\overline{v_c}(x)$ satisfying

\[
\begin{cases}
\partial \psi(y_c(x)) = \partial_1\widetilde{\Phi_M}(y_c(x),\overline{v_c}(x))\\
\partial_{\overline{v}}\phi(x,\overline{v_c}(x)) = \partial_2\widetilde{\Phi_M}(y_c(x),\overline{v_c}(x))
\end{cases}
\]
and such that $\sigma$ is defined by

\begin{gather*}
\widetilde{\kappa_{\phi}}(\Lambda_{\sigma}) = \widetilde{\kappa_{\phi}}\lacc  \widetilde{\nabla} \lpar e^{\frac{2\sigma(x)-\widetilde{\Phi_M}\lpar x,\overline{v}\rpar}{2\hbar}} \rpar = 0\racc = \lacc  \widetilde{\nabla} \lpar e^{\frac{2\psi(y)-\widetilde{\Phi_M}\lpar y,\overline{z}\rpar}{2\hbar}} \rpar = 0\racc = \Lambda_{\psi},\\
\sigma(0) = \phi(0,\overline{v_c}(0))+\psi(y_c(0))-\widetilde{\Phi_M}(y_c(0),\overline{v_c}(0)).
\end{gather*}
Moreover,

\[
w_0(x) = \widetilde{a_0}(x,\overline{x},y_c(x),\overline{v_c}(x)) v_0(y_c(x))\det\nabla K_x(0,0).
\]
where $K_x$ is the local diffeomorphism such that
\[
\lpar \phi(x,\overline{v})+\psi(x)-\widetilde{\Phi_M}(x,\overline{v}) \rpar \lpar K_x(z,\overline{v})\rpar = -z\overline{v} + \sigma(x).
\]
\end{lemma}

\begin{proof}
Applying Proposition \ref{prop_contour_aext} to $\phi_1(x,z) = \phi(x,\overline{z}) - \frac{1}{2}\Phi_M(z)$ and $\phi_2(x,z) = \psi(z) - \frac{1}{2}\Phi_M(z)$ gives

\begin{align*}
	& I_{\phi}(a_{\hbar})\lpar e^{-\frac{\Phi_M(y)}{2\hbar}} e^{\frac{\psi(y)}{\hbar}} v_{\hbar}(y)\rpar(x)\\
	= & e^{-\frac{\Phi_M(x)}{2\hbar}} \int_{B_{\rho}} e^{\frac{\phi(x,\overline{z}) +\psi(z) - \Phi_M(z)}{\hbar}} a_{\hbar}(x,z) v_{\hbar}(z) \frac{dz\wedge d\overline{z}}{\pi\hbar}\\
	= & e^{-\frac{\Phi_M(x)}{2\hbar}} e^{\frac{\sigma(x)}{\hbar}} \lpar \int_{B_{\rho}} e^{-\frac{|y|^2}{\hbar}} \det\nabla K_x(y,\overline{y}) \widetilde{a_{\hbar}}\lpar x,\overline{x}, K_x(y,\overline{y})\rpar v_{\hbar}\lpar  K_x(y,\overline{y})\rpar \frac{dy}{\pi\hbar} + O\lpar e^{-\frac{c}{\hbar}}\rpar \rpar
\end{align*}
where $\sigma(x) = \phi\lpar x,\overline{v_c}(x)\rpar +\psi(y_c(x)) - \widetilde{\Phi_M}\lpar y_c(x),\overline{v_c}(x)\rpar$. Now, if $(x,\overline{y},z,\overline{v})\in\grph(\widetilde{\kappa_{\phi}})$ and $(z,\overline{v}) \in \Lambda_{\psi}$ then,

\begin{align*}
\partial_x\phi(x,\overline{v}) = & \partial_1 \widetilde{\Phi_M}(x,\overline{y})\\
\partial_{\overline{v}}\phi(x,\overline{v}) = & \partial_2 \widetilde{\Phi_M}(z,\overline{v})\\
\partial\psi(z) = & \partial_1 \widetilde{\Phi_M}(z,\overline{v})
\end{align*}
thus, by construction, $z=y_c(x)$ and $\overline{v}=\overline{v_c}(x)$. Therefore,

\[
\partial\sigma(x) = \partial_x\phi(x,\overline{v_c}(x)) = \partial_1 \widetilde{\Phi_M}(x,\overline{y}),
\]
in other words $(x,\overline{y})\in\Lambda_{\sigma}$, which proves that $\widetilde{\kappa_{\phi}}(\Lambda_{\sigma}) = \Lambda_{\psi}$.

Next, using Lemma \ref{le_int_gauss}

\begin{align*}
	& I_{\phi}(a_{\hbar})\lpar e^{-\frac{\Phi_M(y)}{2\hbar}} e^{\frac{\psi(y)}{\hbar}} v_{\hbar}(y)\rpar(x)\\
	 = & e^{-\frac{\Phi_M(x)}{2\hbar}} e^{\frac{\sigma(x)}{\hbar}} \sum\limits_{0\le k\le \frac{\rho}{\hbar}} \frac{\hbar^k}{k!} (\partial_y\overline{\partial_y})^k \Bigr( \det\nabla K_x(y,\overline{y}) \widetilde{a_{\hbar}}\lpar x,\overline{x},K_x(z,\overline{z})\rpar v_{\hbar}\lpar K_x(z,\overline{z})\rpar\Bigr)(0,0)
\end{align*}
up to an exponentially small remainder, so $w_{\hbar}$ is an analytic symbol and

\[
w_0(x)
	= \widetilde{a_0}\lpar x,\overline{x},y_c(x),\overline{v_c}(x)\rpar v_0(y_c(x))\det\nabla K_x(0,0).
\]
\end{proof}

In order to use the results of this section with Toeplitz operators on manifolds, it is necessary to write them as analytic FIOs, which means to write the Bergman kernel as an analytic symbol. This result has been proved around the same time by different groups, and with different methods. Rouby, Sjöstrand and Vũ Ng\d{o}c \cite{roub20} proved it with a Kuranishi trick, and Deleporte \cite{dele21} proved it with a specific class of symbols and associated norms. The proof was simplified by Hezari and Xu \cite{heza21}, and by Charles \cite{char19} who also detailed the corresponding algebra of symbols. Then, another method was used by Deleporte, Hitrik and Sjöstrand \cite{dele20a}. We recall the result.

\begin{proposition}
\label{prop_local_berg}
Let $\Phi_M$ be a Kähler potential of $M$, and $\widetilde{\Phi_M}$ be its holomorphic extension. Then, in local charts, there exists $\rho,\eta>0$ and $\Pc\in S\lpar\widetilde{B_{\rho}}\rpar$ such that the Bergman projection $\Pi_{N} : L^2(M) \rightarrow \Hc_N$ is given by
\begin{align*}
\Pi_{N}(u)(x)
	= & \int_{B_{\rho}} e^{\frac{2\widetilde{\Phi_M}(x,\overline{z})-\Phi_M(z)-\Phi_M(x)}{2\hbar}} u(z) \Pc(x,\overline{y}) \frac{dz}{\pi\hbar}  + O(e^{-\frac{c}{\hbar}})\lnor u\rnor_{L^2(B_{\rho})}\\
	= & I_{\widetilde{\Phi_M}}(\Pc)u(x) + O(e^{-\frac{c}{\hbar}})\lnor u\rnor_{L^2(B_{\rho})}
\end{align*}
for all $x\in B_{\eta}$, $u\in L^2(B_{\rho})$, and with $\hbar=\frac{1}{N}$.
\end{proposition}

Hence, the Toeplitz operators are locally FIOs of the form $I_{\widetilde{\Phi_M}}(\Pc f)$. For that reason, we write them in local charts and see them as operators on $\statesp[\Phi_M]$.

\begin{definition}
Let $f\in S(B_{\rho})$, the Toeplitz operator associated to $f$ is

\[
T_{\hbar}(f)u(x) = \int_{B_{\rho}} e^{\frac{2\widetilde{\Phi_M}(x,\overline{y})-\Phi_M(y)-\Phi_M(x)}{2\hbar}} \Pc(x,\overline{y})f(y) u(y) \frac{dy}{\pi\hbar} = I_{\widetilde{\Phi_M}}(Bf)u(x)
\]
bounded on $\statesp[\Phi_M]$. It will also be useful to consider covariant symbols, in other words symbols $F\in S\lpar\widetilde{B_{\rho}}\rpar$ and operators

\[
T_{\cov,\hbar}(F)u(x) = \int_{B_{\rho}} e^{\frac{2\widetilde{\Phi_M}(x,\overline{y})-\Phi_M(y)-\Phi_M(x)}{2\hbar}} F(x,y) \Pc(x,\overline{y}) u(y) \frac{dy}{\pi\hbar}.
\]
Furthermore, as in Theorem \ref{th_def_FIO}, if $\rho'\in ]0,\rho[$ and $u\in\barg$

\begin{multline*}
\lnor T_{\cov,\hbar}(F)u(x) - \int_{B_{\rho'}} e^{\frac{2\widetilde{\Phi_M}(x,\overline{y})-\Phi_M(y)-\Phi_M(x)}{2\hbar}} F(x,y) \Pc(x,\overline{y}) u(y) \frac{dy}{\pi\hbar} \rnor_{\statesp(B_{\rho'})}\\
\le C e^{-\frac{(\rho-\rho')^2}{\hbar}} \lnor F \rnor_{L^{\infty}\lpar\widetilde{B_{\rho}}\rpar} \lnor u \rnor_{\statesp(B_{\rho})}
\end{multline*}
thus, forgetting the exponentially small error, we can write the integrals on a smaller ball. Moreover, these definitions coincide with the usual Toeplitz operators if $f,F$ are defined on $M,\widetilde{M}$ respectively. In particular, if $M=\Cm$, $\Phi_M(x) = |x|^2$ we denote

\[
T_{\hbar}(f)u(x) = \int_{B_{\rho}} e^{\frac{2x\overline{y}-|y|^2-|x|^2}{2\hbar}} f(y) u(y) \frac{dy}{\pi\hbar} ,\; T_{\cov,\hbar}(F)u(x) = \int_{B_{\rho}} e^{\frac{2x\overline{y}-|y|^2-|x|^2}{2\hbar}} F(x,y) u(y) \frac{dy}{\pi\hbar}.
\]
\end{definition}

\begin{theorem}
Let $\phi$ be a phase for $W$ and $W_l$ the transformed weight, let $\psi$ be a phase for $W_l$ with $\Theta = (W_l)_l$ the transformed weight, and let $a,b \in S(B_{\sigma}\times B_{\rho}) \times S(B_{\tau}\times B_{\sigma})$ be such that $I_{\phi}(a), I_{\psi}(b)$ are well-defined. We suppose that $(z,\overline{v})\mapsto \psi(0,\overline{v})+\phi(z,0)-\widetilde{\Phi_M}(z,\overline{v})$ has $(0,0)$ as its unique critical point. Therefore, there exists $\tau'\in ]0,\tau[$, $\rho'\in ]0,\rho[$ such that for $(x,\overline{y})\in B_{\tau'}\times B_{\rho'}$, there exists unique points $(z_c,\overline{v_c})= \lpar z_c(x,\overline{y}),\overline{v_c}(x,\overline{y}) \rpar$ such that

\[
\left\{\begin{array}{l}
\partial_{\overline{v}}\psi(x,\overline{v_c}) = \partial_2\widetilde{\Phi_M}(z_c,\overline{v_c}),\\
\partial_z \phi(z_c,\overline{y}) = \partial_1\widetilde{\Phi_M}(z_c,\overline{v_c}).
\end{array}\right.
\]
Then, the function

\[
\eta(x,\overline{y}) = \psi(x,\overline{v_c})+\phi(z_c,\overline{y})-\widetilde{\Phi_M}(z_c,\overline{v_c})
\]
is a phase for $W$ associated to the symplectomorphism $\kappa_{\eta} = \kappa_{\phi} \circ \kappa_{\psi}$. Furthermore, for all $u\in\statesp[W](B_{\rho})$

\[
I_{\psi}(b) I_{\phi}(a)u = I_{\eta}(c)u + O_{\statesp[\Theta](B_{\tau'})}(e^{-\frac{c}{\hbar}})\lnor u \rnor_{\statesp[W](B_{\rho})},
\]
where $c\in S(\overline{B_{\tau'}\times B_{\rho'}})$. The symbol $c$ is asymptotically equivalent to

\begin{equation}
\label{eq_prod_symb}
\sum \frac{\hbar^k}{k!} (\partial_z\overline{\partial_z})^k \lpar \det\lpar\nabla_{z,\overline{z}}K_{x,\overline{y}}\rpar \widetilde{b}\lpar x,\overline{x},K_{x,\overline{y}}(z,\overline{z})\rpar \widetilde{a}\lpar K_{x,\overline{y}}(z,\overline{z}),y,\overline{y}\rpar \rpar (0)
\end{equation}
where $K_{x,\overline{y}}$ is such that

\[
\lpar (v,\overline{w}) \mapsto \psi(x,\overline{w})+\phi(v,\overline{w})-\widetilde{\Phi_M}(v,\overline{w}) \rpar\lpar K_{x,\overline{y}}(z,\overline{v})\rpar = -z\overline{v} + \eta(x,\overline{y}).
\]
Moreover, if $\Phi_M = |z|^2$ and one of these three conditions is satisfied: $\phi=0$, $\psi=0$, $\phi$ and $\psi$ are quadratic, then

\[
c_0(x,\overline{y}) = \lpar 1- \partial_z^2\phi(z_c,\overline{y})\partial_{\overline{v}}^2\psi(x,\overline{v_c})\rpar^{-\frac{1}{2}} \widetilde{b_0}\lpar x,\overline{x},z_c,\overline{v_c}\rpar \widetilde{a_0}\lpar z_c,\overline{v_c},y,\overline{y}\rpar.
\]
\end{theorem}

A construction of these FIOs using Lagrangian submanifolds can be found in \cite{dele25}, where they prove the symbolic calculus with more geometric arguments.

\begin{proof}
Let $\tau'\in ]0,\tau[$, $\rho'\in ]0,\rho[$ and $\sigma'\in ]0,\sigma[$ that we will fix later. We consider $x\in B_{\tau'}$, and we write $I_{\phi}(a)$ on $B_{\rho'}$, and $I_{\psi}(b)$ on $B_{\sigma'}$, which adds an exponentially small error that we don't write as it will not change the result,
\begin{align*}
I_{\psi}(b) I_{\phi}(a) u(x)
	= & \int_{B_{\sigma'}} e^{\frac{-\Phi_M(x)-\Phi_M(z)+2\psi(x,\overline{z})}{2\hbar}} b(x,z) \int_{B_{\rho'}} e^{\frac{-\Phi_M(z)-\Phi_M(y)+2\phi(z,\overline{y})}{2\hbar}} a(z,y) u(y) \frac{dydz}{\pi\hbar}\\
	= & \int_{B_{\rho'}} e^{\frac{-\Phi_M(x)-\Phi_M(y)}{2\hbar}} u(y) \lpar \int_{{B_{\sigma'}}} e^{\frac{\psi(x,\overline{z})+\phi(z,\overline{y})-\Phi_M(z)}{\hbar}} b(x,z)a(z,y) \frac{dz}{\sqrt{\pi\hbar}} \rpar \frac{dy}{\sqrt{\pi\hbar}}\\
	= & I_{\eta}(c)u(x).
\end{align*}
According to Proposition \ref{prop_phase_inverse}, $z \mapsto 2\Re\phi(z,\overline{y}) -\Phi_M(z)- W_l(z)$ has a unique critical point of signature $(0,-2d)$ for all $x$, then we can apply Proposition \ref{prop_contour_aext} to $\phi_1(x,\overline{y},z) = \phi(z,\overline{y})-\frac{1}{2}\Phi_M(z)-\frac{1}{2}W_l(z)$ and $\phi_2(x,\overline{y},z) = \psi(x,\overline{z}) -\frac{1}{2}\Phi_M(z)+ \frac{1}{2}W_l(z)$, there exists $\tau',\rho',\sigma'$ such that for all $(x,y)\in B_{\tau'}\times B_{\rho'}$,

\begin{align*}
	  & c(x,\overline{y})\\
	= & \int_{B_{\sigma'}} e^{\frac{\psi(x,\overline{z})+\phi(z,\overline{y})-\Phi_M(z)-\eta(x,\overline{y})}{\hbar}} b(x,z)a(z,y) \frac{dz}{\sqrt{\pi\hbar}}\\
	= & \int_{B_{\sigma'}} e^{-\frac{|z|^2}{\hbar}} \det\lpar\nabla_{(w,\overline{w})}K_{x,\overline{y}}\rpar \widetilde{b}\lpar x,\overline{x},K_{x,\overline{y}}(w,\overline{w})\rpar \widetilde{a}\lpar K_{x,\overline{y}}(w,\overline{w}),y,\overline{y}\rpar \frac{dw\wedge d\overline{w}}{\sqrt{\pi\hbar}}\\
	= & \sum\limits_{k\le \frac{1}{3D\hbar}} \frac{\hbar^k}{k!} (\partial_w\overline{\partial_w})^k \lpar \det\lpar\nabla_{(w,\overline{w})}K_{x,\overline{y}}\rpar \widetilde{b}\lpar x,\overline{x},K_{x,\overline{y}}(w,\overline{w})\rpar \widetilde{a}\lpar K_{x,\overline{y}}(w,\overline{w}),y,\overline{y}\rpar\rpar(0,0)
\end{align*}
up to an exponentially small remainder using Lemma \ref{le_int_gauss}, with $D>\sigma'$. If $\Phi_M = |z|^2$ and $\phi=0$ or $\psi=0$ or both are quadratic, then $\det\nabla_{0,0} K_{x,\overline{y}}$ is constant equal to

\begin{align*}
	& \lpar -\det\hess_{z,\overline{v}} \lpar \psi(0,\overline{v})+\phi(z,0)-\widetilde{\Phi_M}(z,\overline{v})\rpar\lpar z_c,\overline{v_c}\rpar \rpar^{-\frac{1}{2}}\\
	= & \lpar 1- \partial_z^2\phi(z_c,\overline{y})\partial_{\overline{v}}^2\psi(x,\overline{v_c})\rpar^{-\frac{1}{2}}
\end{align*}
according to Proposition \ref{prop_contour_aext}.

Now, we prove that $c$ is an analytic symbol. Choosing $D$ big enough, we can write

\[
a = \sum\limits_{0\le k\le \frac{1}{3D\hbar}} \hbar^k a_k, \quad b = \sum\limits_{0\le k\le \frac{1}{3D\hbar}} \hbar^k b_k
\]
up to an exponentially small term. Hence, $c(x,\overline{y}) = \sum\limits_{0\le m\le \frac{1}{D\hbar}} \hbar^m c_m(x,\overline{y})$, where $c_m(x,\overline{y})$ is equal to

\[
\sum\limits_{(j,k,l)\in S_m} \frac{1}{j!} (\partial_w\overline{\partial_w})^j \lpar \det\lpar\nabla_{(w,\overline{w})}K_{x,\overline{y}}\rpar \widetilde{b_k}\lpar x,\overline{x},K_x(w,\overline{w})\rpar \widetilde{a_l}\lpar K_x(w,\overline{w}),y,\overline{y}\rpar\rpar (0,0)
\]
where $S_m = \lacc (j,k,l)\in\Nm^3 \middle/\; j+k+l=m, j,k,l\le\frac{1}{3D\hbar} \racc$. By hypothesis, there exists $A,B,c>0$ such that $|\det\nabla K_x|\le c$ on $B_{\sigma'}\times\overline{B_{\sigma'}}$, $|\widetilde{a_k}|\le k!A^k$ on $B_{\sigma'}\times\overline{B_{\rho'}}$ and $|\widetilde{b_l}|\le l!B^l$ on $B_{\tau'}\times\overline{B_{\sigma'}}$, and because of Proposition \ref{prop_contour_aext}, $z_c(x),\overline{v_c}(x) \in B_{\sigma'}$ for all $(x,y)\in B_{\tau'}\times B_{\rho'}$. Then, writing $C = \min(\sigma-\sigma',\rho-\rho',\tau-\tau',\sigma')^{-2}$, Lemma \ref{le_Cauchy_ana} gives

\begin{align*}
|c_m(x,\overline{y})|
	\le & c\sum\limits_{(j,k,l)\in S_m} j! C^j k!A^k l!B^l\\
	\le & c \max(A,B,C)^m m(m+1)!\\
	\le & c m! D^m
\end{align*}
if $D$ is chosen big enough.
\end{proof}

\begin{corollary}
\label{cor_FIO_para}
Let $\widetilde{\kappa}$ be a symplectormophism on a neighbourhood of $0$ such that there exists holomorphic functions $\phi$, $\phi_i$ associated to $\widetilde{\kappa}$ and $\widetilde{\kappa}^{-1}$ respectively. Suppose that $\phi$ is a phase for $0$ with transformed $W$, then the real-analytic, real-valued function $W_i$ such that $W_i(0)=0$, and

\[
\overline{\partial}W_i(x) = x-\lpar\id+\overline{\partial}W_l\rpar^{-1}(x),
\]
is a weight, and $\phi_i$ is a phase for $W_i$ with a vanishing transformed weight. Furthermore, if $(a,b) \in S(B_{\tau}\times \overline{B_{\sigma}}) \times S(B_{\sigma}\times \overline{B_{\rho}})$ are such that $I_{\phi}(a), I_{\phi_i}(b)$ are well-defined, then there exists $\tau' \in]0,\tau],\, \sigma'\in ]0,\sigma]$, and $F\in S(B_{\tau'}\times \overline{B_{\rho'}})$ such that

\[
I_{\phi}(a) I_{\phi_i}(b) = T_{\cov,\hbar}(F) + O_{\statesp[W_i](B_{\rho'})\rightarrow\statesp[W](B_{\tau'})}(e^{-\frac{c}{\hbar}}).
\]
where $F$ is given at first order by

\[
F_0(x,y)
	= \Pc_0(x,\overline{y})^{-1}\widetilde{a}(x,\overline{x},\kappa_{\phi}(x,\overline{y})) \widetilde{b}(\kappa_{\phi}(x,\overline{y}),y,\overline{y}) \det\nabla_{0,0} K_{x,\overline{y}}.
\]
Moreover, there exists $a_i\in S(B_{\tau'}\times \overline{B_{\sigma'}})$ satisfying

\begin{equation}
\label{eq_symb_inv}
\widetilde{a_i}(\kappa_{\phi}(x,\overline{y}),y,\overline{y}) =  \frac{\Pc_0(x,\overline{y})}{\widetilde{a}(x,\overline{x},\widetilde{\kappa}_{\phi}(x,\overline{y})) \det\nabla_{0,0} K_{x,\overline{y}}} + O(\hbar),
\end{equation}
and

\[
I_{\phi}(a)I_{\phi_i}(a_i) = \id + O_{\statesp[W_i](B_{\rho'})\rightarrow\statesp[W](B_{\tau'})}(e^{-\frac{c}{\hbar}}).
\]
\end{corollary}

This result proves in particular that $\statesp[W_i]$ is embedded in $\statesp[W]$.

\begin{proof}
According to lemma \ref{le_holo_ext}, there exists $a_i\in S(B_{\tau}\times B_{\sigma})$ satisfying \eqref{eq_symb_inv} up to taking $\tau,\sigma$ smaller, which gives $I_{\phi}(a)I_{\phi_i}(a_i) = \id + O(\hbar)$. Though, the operator on the right-hand side is now a Toeplitz operator with symbol $1+O(\hbar)$, which can be inverted with the usual symbolic calculus, see for instance \cite{dele25} or Section \ref{subsec_symbol_calcul}.
\end{proof}

Hitrik and Zworski used the same method in \cite{hitr24} but with a different definition for the spaces $\statesp[\Phi_M]$ and the FIOs. They obtained a similar result but with $W_i=W$ for their weighted $L^2$ spaces. We will see in Theorem \ref{th_elliptic} that we can use localisation to neglect the difference between $W$ and $W_i$.

\begin{corollary}
We suppose here $\Phi_M(z)=|z|^2$. Let $\phi,\psi$ be two phases for a weight $W$, and $(a,f,b) \in S(B_{\tau}\times \overline{B_{\rho}})\times S(B_{\rho}) \times S(B_{\rho}\times \overline{B_{\sigma}})$, then there exists $\tau'\in]0,\tau[,\, \rho'\in]0,\rho[,\, \sigma'\in]0,\sigma[$ and $(c,d)\in S(B_{\tau'}\times \overline{B_{\rho'}})\times S(B_{\rho'}\times \overline{B_{\sigma'}})$ such that

\begin{align*}
I_{\psi}(b)T_{\cov,\hbar}(F) & = I_{\psi}(c) + O_{\statesp[W](B_{\rho'}) \to \statesp[W_l](B_{\sigma'})}(e^{-\frac{c}{\hbar}}),\\
T_{\cov,\hbar}(F)I_{\phi}(a) & = I_{\phi}(d) + O_{\statesp[W](B_{\tau'}) \to \statesp[W_l](B_{\rho'})}(e^{-\frac{c}{\hbar}}).
\end{align*}
Furthermore,

\begin{align*}
c(x,\overline{y})
	= & \sum_k \frac{\hbar^k}{k!} \partial^k\lpar \overline{\partial}+\partial_{\overline{v}}\psi_q(x,\overline{y},\overline{v})\partial\rpar^k|_{\overline{v}=0} \lpar \widetilde{b}(x,\overline{x},\cdot)\widetilde{F}(\cdot,y,\overline{y}) \rpar(\overline{\partial_y}\psi(x,\overline{y}),\overline{y})\\
\psi_q(x,\overline{y},\overline{v})
	= & \frac{\psi(x,\overline{v}+\overline{y})-\psi(x,\overline{y})-\overline{v}\overline{\partial_y}\psi(x,\overline{y})}{\overline{v}}\\
d(x,\overline{y})
	= & \sum_k \frac{\hbar^k}{k!} \lpar\partial+\partial_z\phi_q(x,\overline{y},z)\overline{\partial}\rpar^k|_{z=0}\overline{\partial}^k \lpar \widetilde{F}(x,\overline{x},\cdot)\widetilde{a}(\cdot,y,\overline{y}) \rpar(x,\partial_x\phi(x,\overline{y}))\\
\phi_q(x,\overline{y},z)
	= & \frac{\phi(z+x,\overline{y})-\phi(x,\overline{y})-z\partial_x\phi(x,\overline{y})}{z}.
\end{align*}
\end{corollary}

If we use the formula with $F=1$, it gives an intricate condition for two symbols to have the same operator

\begin{equation}
\label{eq_cond_equal_op}
I_{\phi}(a) = I_{\phi} \lpar \sum_k \frac{\hbar^k}{k!} \partial_y^k(\overline{\partial_y}+\partial_{\overline{v}}\phi_q(x,\overline{y},\overline{v})\partial_y)^k \widetilde{a}(x,\overline{x},\overline{\partial_y}\phi(x,\overline{y}),\overline{y}) \rpar.
\end{equation}

We now prove a result generalising what we found in the proof of Theorem \ref{th_spec_quad}. We will use it in a crucial step of the proof of Theorem \ref{th_elliptic}.

\begin{theorem}
\label{th_FIO_normal}
Suppose $\Phi_M(z)= |z|^2$, and let $H$ a quadratic function satisfying \eqref{eq_complex_quad}, which means

\begin{align*}
& \det(\Re f) + \det(\Im f) + i\sqrt{\lver \Im(\det f)^2-4\det(\Re f)\det(\Im f)\rver} \notin \Rm_{-},\\\nonumber
& f(\Rm^2) \neq \Cm.
\end{align*}
Let $\widetilde{\kappa}$ be the linear symplectormophism from Section \ref{subsec_quad_symbol}, and $\phi$ the associated phase function. Then, $I_{\phi} = I_{\phi}(1): \barg \rightarrow \statesp[W]$ and $I_{\phi_i} =
I_{\phi_i}(1): \statesp[W_i] \rightarrow \barg$ are bounded, with

\begin{align*}
W(x) = & \frac{r-1}{r}|x|^2 + \frac{2j}{r}\Re(x)\Im(x)\\
W_i(x) = & \frac{1-r}{r}|x|^2 + \frac{2j}{r}\Re(x)\Im(x)
\end{align*}
and $r,j\in\Rm_+^*\times\Rm$ are such that $(r+ij)^2 = \frac{\zeta}{|\zeta|}$. Moreover,

\[
T_{\hbar}(H) I_{\phi}(1) = I_{\phi}(1) T_{\cov,\hbar}\lpar d|y|^2 +\frac{\hbar\partial\overline{\partial}}{2}(H(y)-d|y|^2) \rpar+O(e^{-\frac{c}{\hbar}}).
\]
\end{theorem}

\begin{proof}
We saw in Section \ref{subsec_quad_symbol} that $\widetilde{\kappa} = \widetilde{\kappa_3} \circ \widetilde{\kappa_2} \circ \widetilde{\kappa_1}$, where $\kappa_1$ and $\kappa_2$ are real symplectomorphisms, and

\[
\widetilde{\kappa_3} = \frac{1}{2} \begin{pmatrix}
\zeta^{\frac{1}{4}}+\zeta^{-\frac{1}{4}} & \zeta^{\frac{1}{4}}-\zeta^{-\frac{1}{4}}\\
\zeta^{\frac{1}{4}}-\zeta^{-\frac{1}{4}} & \zeta^{\frac{1}{4}}+\zeta^{-\frac{1}{4}}
\end{pmatrix}.
\]
According to Example \ref{ex_symp}, $\widetilde{\kappa_1}$ and $\widetilde{\kappa_2}$ does not change the weight. Then, according to Lemma \ref{le_phase_linear} the transformed weight from $0$ by $\phi$ associated to $\widetilde{\kappa_3}$ is

\[
W(x) = \frac{1-r}{r}|x|^2 + \Re \lpar \frac{-ij}{r} x^2 \rpar = \frac{1-r}{r}|x|^2+\frac{2j}{r}\Re(x)\Im(x)
\]
where

\begin{align*}
r = & \frac{|\zeta^{\frac{1}{4}}+\zeta^{-\frac{1}{4}}|^2-|\zeta^{\frac{1}{4}}-\zeta^{-\frac{1}{4}}|^2}{4} = \Re\lpar\lpar\frac{\zeta}{|\zeta|}\rpar^{\frac{1}{2}}\rpar,\quad\\
j = & i\frac{\lpar \zeta^{\frac{1}{4}}+\zeta^{-\frac{1}{4}} \rpar\lpar \overline{\zeta}^{\frac{1}{4}}-\overline{\zeta}^{-\frac{1}{4}} \rpar - \lpar \zeta^{\frac{1}{4}}-\zeta^{-\frac{1}{4}} \rpar\lpar \overline{\zeta}^{\frac{1}{4}}+\overline{\zeta}^{-\frac{1}{4}} \rpar}{4} = \Im\lpar\lpar\frac{\zeta}{|\zeta|}\rpar^{\frac{1}{2}}\rpar.
\end{align*}
We are looking for a weight $W_i$ such that $\phi_i$ is a phase for $W_i$ with transformed weight $0$. According to Lemma \ref{le_link_weigts}, this is equivalent to $\widetilde{\kappa}^{-1}(\diag) = D \overline{\Lambda_{-W_i}}$, which is the same as $\Lambda_W = D \overline{\Lambda_{-W_i}}$ since $\widetilde{\kappa}(\Lambda_W) = \diag$ by definition. Let $y\in\Cm$, we look for $x\in\Cm$ such that
\[
\left\{\begin{array}{l}
x = y - \overline{\partial}W_i(y),\\
\frac{1}{r}\overline{x} - \frac{ij}{r}x = \overline{y}.
\end{array}\right.
\]
It is satisfied for $x = \frac{1}{r}y - \frac{ij}{r}\overline{y}$, and since $\frac{1-j^2}{r^2} = 1$, we get $\overline{\partial}W_i(y) = -\frac{1-r}{r}y+\frac{ij}{r}\overline{y}$, thus the result.

For the product of operators, we use \eqref{eq_prod_symb} where $K_{x,\overline{y}}(z,\overline{v})=\lpar z,\overline{v}+\frac{\partial_x^2\phi}{2}z\rpar$
 
\[
T_{\hbar}(H) I_{\phi}(a) = aI_{\phi}\lpar \widetilde{H}(x,\partial_x\phi(x,\overline{y})) + \hbar \partial\overline{\partial}H + \hbar \frac{\partial_x^2\phi}{2}\overline{\partial}^2H\rpar+O(e^{-\frac{c}{\hbar}}),
\]
with $\widetilde{H}(x,\partial_x\phi(x,\overline{y})) = \widetilde{H}\circ\widetilde{\kappa}(\overline{\partial_y}\phi(x,\overline{y}),\overline{y}) = d\overline{\partial_y}\phi(x,\overline{y})\overline{y}$. On the other way

\[
I_{\phi}(a) T_{\cov,\hbar}(dx\overline{y}) = aI_{\phi}\lpar d\overline{\partial_y}\phi(x,\overline{y})\overline{y}\rpar+O(e^{-\frac{c}{\hbar}}),
\]
with $T_{\hbar}(|y|^2) = T_{\hbar}(x\overline{y}+\hbar)$ according to \eqref{eq_cond_equal_op}, hence

\[
I_{\phi}(a) T_{\hbar}\lpar |y|^2-d\hbar+\hbar \partial\overline{\partial}H + \hbar \frac{\partial_x^2\phi}{2}\overline{\partial}^2H\rpar = T_{\hbar}(H) I_{\phi}(a) + O(e^{-\frac{c}{\hbar}}).
\]
The result comes from the relation we found in the proof of Theorem \ref{th_spec_quad}:

\[
	\partial\overline{\partial}H+\frac{1}{2}\partial_x^2\phi(x,\overline{y})\overline{\partial}^2H
	= \frac{1}{2}(\partial\overline{\partial}H+d)
	= \frac{\partial\overline{\partial}}{2}(H(x)+d|x|^2).
\]
\end{proof}

Notice that in the proof, the sum from \eqref{eq_prod_symb} is finite as $H$ is quadratic and $a=1$, the remainder comes from the contour deformation of Proposition \ref{prop_contour_aext}.

We can check that the bound $I_{\phi} : \barg \rightarrow \statesp[W_l]$ from Theorem \ref{th_FIO_normal} is optimal. Consider $u(y) = e^{-\frac{|y|^2}{2\hbar}} e^{\frac{i\sigma y^2}{2\hbar}} \in \barg$ for $\sigma\in]-1,1[$. According to Lemma \ref{le_phase_linear}, $\phi(x,\overline{y}) = -\frac{C_-}{2C_+}x^2+\frac{1}{C_+}x\overline{y}+\frac{C_-}{2C_+}\overline{y}^2$ where $C_+ = \frac{\zeta^{\frac{1}{4}}+\zeta^{-\frac{1}{4}}}{2}$ and $C_- = \frac{\zeta^{\frac{1}{4}}-\zeta^{-\frac{1}{4}}}{2}$, then

\[
I_{\phi}(1)u(x) = e^{-\frac{|x|^2 + \frac{C_-}{C_+}x^2}{2\hbar}} \int_{\Cm} e^{\frac{-|y|^2+\frac{1}{C_+}x\overline{y}+\frac{C_-}{2C_+}\overline{y}^2+\frac{i\sigma}{2}y^2}{\hbar}} \frac{dy}{\pi\hbar} = e^{-\frac{|x|^2 + \frac{C_-}{C_+}x^2}{2\hbar}} I'_{\phi}(1)u(x).
\]
Using that $\int_{\Rm} e^{-\frac{At^2}{\hbar}} t^{2k} \frac{dt}{\sqrt{\pi\hbar}} = \frac{(2k)!\hbar^k}{k!4^k} A^{-k-\frac{1}{2}}$, we compute for $l\in\Nm$

\begin{align*}
(\hbar\partial)^l I'_{\phi}(1)u(0)
	& = C_+^{-l} \int_{\Cm} e^{-\frac{\alpha\Re(y)^2}{\hbar}} e^{-\frac{\beta\Im(y)^2}{\hbar}} e^{\frac{2\gamma\Re(y)\Im(y)}{\hbar}} \overline{y}^l \frac{dy}{\pi\hbar}\\
	& = C_+^{-l} \int_{\Cm} e^{-\frac{\alpha r^2}{\hbar}} e^{-\frac{(\alpha\beta-\gamma^2) s^2}{\alpha\hbar}} \lpar r+\lpar \frac{\gamma}{\alpha}-i \rpar s \rpar^l \frac{drds}{\pi\hbar}\\
	& = C_+^{-l} \sum\limits_{k=0}^l \binom{l}{k} \lpar \frac{\gamma}{\alpha}-i \rpar^{l-k} \int_{\Rm} e^{-\frac{(\alpha\beta-\gamma^2) s^2}{\alpha\hbar}} \frac{dr}{\sqrt{\pi\hbar}} \int_{\Rm} e^{-\frac{(\alpha\beta-\gamma^2) s^2}{\alpha\hbar}} \frac{ds}{\sqrt{\pi\hbar}}\\
	& = \begin{cases}
      0 & \text{if $l$ is odd,}\\
      \frac{(2l')!}{l'!} \lpar\frac{\hbar}{4C_+^2}\rpar^{l'} \lpar \frac{\beta-\alpha-2i\gamma}{\alpha\beta-\gamma^2} \rpar^{l'} (\alpha\beta-\gamma^2)^{-\frac{1}{2}} & \text{if } l=2l',
    \end{cases}
\end{align*}
where $\alpha = 1-i\frac{\sigma}{2}-\frac{C_-}{2C_+}$, $\beta = 1+i\frac{\sigma}{2}+\frac{C_-}{2C_+}$, and $\gamma = -i\frac{C_-}{2C_+}-\frac{\sigma}{2}$. Hence,

\begin{align*}
I_{\phi}(1)u(x)
	= & e^{-\frac{|x|^2}{2\hbar}} e^{-\frac{C_-x^2}{2C_+\hbar}} (\alpha\beta-\gamma^2)^{-\frac{1}{2}} e^{\frac{\beta-\alpha-2i\gamma}{4C_+^2(\alpha\beta-\gamma^2)}\frac{x^2}{\hbar}}\\
	= & (\alpha\beta-\gamma^2)^{-\frac{1}{2}} e^{-\frac{|x|^2}{2\hbar}} e^{-\frac{C_-x^2}{2C_+\hbar}} e^{\frac{i\sigma x^2}{2C_+(C_+-iC_-\sigma)\hbar}}.
\end{align*}
Now, consider $\tau\ge 0$ and $\epsilon \in [-\tau,\tau]$ so that the quadratic form $p(x) = \frac{\tau |x|^2}{2} + \epsilon\Re(x)\Im(x)$ is positive,

\begin{align*}
\lver e^{\frac{p(x)}{\hbar}} e^{-\frac{W_l(x)}{2\hbar}} I_{\phi}u(x) \rver
	= & e^{-\frac{|x|^2}{2\hbar}\lpar\frac{1}{r}-\tau\rpar} e^{-\Re\lpar\frac{x^2}{2\hbar}\lpar -\frac{ji}{r}+\frac{C_-}{C_+}-\frac{i\sigma}{C_+(C_+-i\sigma C_-)} +i\epsilon \rpar\rpar}\\
	= & e^{-\frac{|x|^2}{2\hbar}\lpar\frac{1}{r}-\tau\rpar} e^{-\Re\lpar\frac{x^2}{2\hbar}\lpar \frac{\overline{C_-}}{C_+r}-\frac{i\sigma}{C_+(C_+-i\sigma C_-)} +i\epsilon \rpar\rpar}\\
	= & e^{-\frac{q(x)}{\hbar}}
\end{align*}
with $q$ a quadratic form on $\lpar\Re(x),\Im(x)\rpar$. The trace of $q$ is $\tau-\frac{1}{r}$, thus we consider $\tau<\frac{1}{r}$ from now on. The determinant of $q$ is

\[
\lpar\frac{1}{r}-\tau\rpar^2 - \lver \frac{\overline{C_-}}{C_+r}-\frac{i\sigma}{C_+(C_+-i\sigma C_-)} +i\epsilon\rver^2.
\]
We check that for $\tau=\epsilon=0$ the determinant is strictly greater than $0$ for all $-1<\sigma<1$  and equal to $0$ for $\sigma^2=1$.

\begin{align*}
\lver \frac{\overline{C_-}}{C_+r}-\frac{i\sigma}{C_+(C_+-i\sigma C_-)} \rver^2
	= & \frac{|C_-|^2}{|C_+|^2r^2}+\frac{\sigma^2}{|C_+|^2|C_+-i\sigma C_-|^2}-\frac{2}{|C_+|^2r}\Re\lpar\frac{i\sigma C_-}{C_+-i\sigma C_-}\rpar\\
	= & \frac{|C_-|^2}{|C_+|^2r^2}+\frac{\sigma^2r+\sigma j +2\sigma^2|C_-|^2}{r|C_+|^2|C_+-i\sigma C_-|^2}\\
	= & \frac{1}{r^2|C_+|^2}\lpar|C_-|^2+r\frac{\sigma^2|C_+|^2+\sigma^2|C_-|^2+\sigma j}{|C_+|^2+\sigma^2|C_-|^2+\sigma j}\rpar\\
	= & \frac{1}{r^2|C_+|^2}\lpar|C_-|^2+r+r\frac{(\sigma^2-1)|C_+|^2}{|C_+|^2+\sigma^2|C_-|^2+\sigma j}\rpar\\
	= & \frac{1}{r^2} - \frac{(1-\sigma^2)}{r(|C_+|^2+\sigma^2|C_-|^2+\sigma j)}
\end{align*}
hence, the determinant is equal to

\[
\frac{(1-\sigma^2)}{r|C_+-i\sigma C_-|^2} - \frac{2\tau}{r} +\tau^2 -2\epsilon\Im\lpar \frac{\overline{C_-}}{C_+r}-\frac{i\sigma}{C_+(C_+-i\sigma C_-)} \rpar -\epsilon^2.
\] 
For $\sigma=1$ the determinant is equal to

\[
\lpar\tau^2- \frac{2\tau}{r}\rpar - \epsilon^2 +\frac{2\epsilon}{r}\Re\lpar\lpar \frac{C_+-iC_-}{|C_+-iC_-|} \rpar^2\rpar,
\]
and for $\sigma=-1$ it is equal to

\[
\lpar\tau^2- \frac{2\tau}{r}\rpar - \epsilon^2 -\frac{2\epsilon}{r}\Re\lpar\lpar \frac{C_++iC_-}{|C_++iC_-|} \rpar^2\rpar.
\]
One of these two is negative, thus there exists a $|\sigma|<1$ such that the determinant is still negative. For such $\sigma$, the function $u_{\sigma}$ is in $\barg$, and $I_{\phi}(u)$ is in $\statesp[W_l]$ but not in $\statesp[W_l-p]$. Hence, $W_l$ is the optimal weight.
\subsection{Symbolic calculus of Toeplitz operators}
\label{subsec_symbol_calcul}

We recall the tools of analytic symbolic calculus on the Bargmann space on $\Cm^d$. By \emph{analytic}, we mean that the data are taken modulo an exponentially small error instead of $\hbar^{\infty}$. The analytic symbolic calculus is already well known for pseudodifferential operators \cite{bout67,bout72,sjos96}, and for Toeplitz operators on $\Cm$ \cite{sjos82}. As said before to introduce Proposition \ref{prop_local_berg}, it is also known for Toeplitz operators on manifolds thanks to the analyticity of the Bergman kernel \cite{roub20,char19,dele21,heza21,dele20a}.

In this subsection $M=\Cm$, $\Phi_M(z) = |z|^2$, and recall that for $f\in S(B_{\rho})$, the Toeplitz operator associated to $f$ is

\[
T_{\hbar}(f)u(x) = I_{z\overline{v}}(f)u(x) = \int_{B_{\rho}} e^{\frac{2\overline{y}x-|y|^2-|x|^2}{2\hbar}} f(y) u(y) \frac{dy}{\pi\hbar}
\]
which is bounded on $\barg$.

\begin{proposition}
\label{prop_prod_symbol}
If $f,g \in S(B_{\rho})$, then for all $\sigma\in]0,\rho[$ there exists $f\#g \in S(B_{\sigma})$ such that $T_{\hbar}(f)T_{\hbar}(g) = T_{\hbar}(f\# g) + O\lpar e^{-\frac{c}{\hbar}} \rpar$ on $\barg(B_{\sigma})$, and
\[
f\# g \sim \sum\limits_k \frac{(-\hbar)^k}{k!} \partial_z^kf \partial_{\overline{z}}^k g.
\]
\end{proposition}

\begin{proof}
Consider $\sigma\in]0,\rho[$, let $u\in\barg$, recall that $u(y) = e^{-\frac{|y|^2}{2\hbar}}r(y)$ with $r$ holomorphic on $\Cm^d$. The contour $\Gamma(y,\overline{v}) = \lacc (z,\overline{w})\in \Cm^{2d}, \exists t\in\Cm^d /\; \overline{w}=\overline{t}+\overline{v},z=-t+y \racc$ is good for $(z,w)\mapsto (y-z)(\overline{w}-\overline{v})$. Let $x\in B_{\sigma}$, we compute

\begin{align*}
& T_{\hbar}(f)T_{\hbar}(g)u(x)\\
	= & \int_{\diag[\rho]}\int_{\diag[\rho]} e^{\frac{2\overline{v}x-z\overline{v}-|x|^2}{2\hbar}} e^{\frac{2\overline{w}z-z\overline{v}-y\overline{w}}{2\hbar}} \widetilde{f}(z,\overline{v}) \widetilde{g}(y,\overline{w}) e^{-\frac{\overline{w}y}{2\hbar}} r(y) \frac{dy\wedge d\overline{w}\wedge dz\wedge d\overline{v}}{(\pi\hbar)^2}\\
	= & \int_{\diag[\rho]} e^{\frac{2\overline{v}x-y\overline{v}-|x|^2}{2\hbar}} \lpar \int_{\Gamma(y,\overline{v})\cap \widetilde{B_{\rho}}}  e^{-\frac{(y-z)(\overline{w}-\overline{v})}{\hbar}} \widetilde{f}(z,\overline{v}) \widetilde{g}(y,\overline{w}) \frac{dz\wedge d\overline{w}}{\pi\hbar} \rpar e^{-\frac{y\overline{v}}{2\hbar}} r(y) \frac{dy\wedge d\overline{v}}{\pi\hbar}\\
	= & \int_{B_{\rho}} e^{\frac{2\overline{y}x-|y|^2-|x|^2}{2\hbar}} \lpar \int_{\diag\cap\lpar\widetilde{B_{\rho}}+(y,-\overline{y})\rpar}  e^{-\frac{z\overline{w}}{\hbar}} \widetilde{f}(y-z,\overline{y}) \widetilde{g}(y,\overline{w}+\overline{y}) \frac{dz\wedge d\overline{w}}{\pi\hbar} \rpar e^{-\frac{|y|^2}{2\hbar}} r(y) \frac{dy}{\pi\hbar}
\end{align*}
using Proposition \ref{prop_linear_contours}, and up to a $O(e^{-\frac{c}{\hbar}})$ remainder. Now, for $|y|\ge\sigma$, the integral is exponentially small due to $e^{\frac{2\overline{y}x-|y|^2-|x|^2}{2\hbar}}$, and for $|y|<\sigma$, the set $\diag\cap\lpar\widetilde{B_{\rho}}+(y,-\overline{y})\rpar$ contains $\diag[\rho']$ for $\rho'\in]0,\rho[$ small enough. Hence, we can replace the set of integration by $\diag[\rho']$ up to an exponentially small error. Then, we choose a $D>\min (\rho-\rho',\rho')^{-1}$, and we apply Lemma \ref{le_int_gauss} with $N=\frac{1}{3D\hbar}$ and $(x,y)\in B_{\sigma}^2$,

\begin{align*}
f\# g(y) & = \int_{B_{\rho'}}  e^{-\frac{|z|^2}{\hbar}} \widetilde{f}(y-z,\overline{y}) \widetilde{g}(y,\overline{z}+\overline{y}) \frac{dz}{\pi\hbar} + O\lpar e^{-\frac{c}{\hbar}} \rpar\\
				& = \sum\limits_{0\le k\le \frac{1}{3D\hbar}} \frac{\hbar^k}{k!} (\partial_z\overline{\partial_z})^k \lpar \widetilde{f}(y-z,\overline{y}) \widetilde{g}(y,\overline{z}+\overline{y}) \rpar(0) + O\lpar e^{-\frac{c}{\hbar}} \rpar\\
				& = \sum\limits_{0\le k\le \frac{1}{3D\hbar}} \frac{(-\hbar)^k}{k!} \partial_1^k \widetilde{f}(y,\overline{y}) \partial_2^k \widetilde{g}(y,\overline{y}) + O\lpar e^{-\frac{c}{\hbar}} \rpar\\
				& = \sum\limits_{0\le k\le \frac{1}{3D\hbar}} \frac{(-\hbar)^k}{k!} \partial^k f(y) \overline{\partial}^k g(y) + O\lpar e^{-\frac{c}{\hbar}} \rpar.
\end{align*}
Thus,

\[
f\# g \sim \sum\limits_{0\le k\le \frac{1}{3D\hbar}} \frac{(-\hbar)^k}{k!} \partial_z^kf \partial_{\overline{z}}^k g
\]
which is an analytic function on $B_{\sigma}$, let us check that it is an analytic symbol. Since $f,g$ are analytic symbols, there exist $A,B>0$ such that $|f_k|\le k!A^k$ and $|g_k|\le k!B^k$ for all $k\in\Nm$. We can choose $D$ such that $D > A,B$ so that $f = \sum\limits_{0\le k \le \frac{1}{3D\hbar}} \hbar^k f_k + O\lpar e^{-\frac{c}{\hbar}} \rpar$, $g = \sum\limits_{0\le k \le \frac{1}{3D\hbar}} \hbar^k g_k + O\lpar e^{-\frac{c}{\hbar}} \rpar$, which computes

\begin{align*}
f\# g
	& = \sum\limits_{0\le j\le \frac{1}{3D\hbar}} \frac{(-\hbar)^j}{j!} \lpar
\sum\limits_{0\le k\le \frac{1}{3D\hbar}} \hbar^k \partial^j f_k
\sum\limits_{0\le l\le \frac{1}{3D\hbar}} \hbar^l \overline{\partial}^k g_l \rpar + O\lpar e^{-\frac{c}{\hbar}} \rpar\\
	& = \sum\limits_{0\le m\le \frac{1}{D\hbar}} \hbar^m
\sum\limits_{(j,k,l)\in S_m} \frac{(-1)^j}{j!} \partial^j f_k \overline{\partial}^j g_l + O\lpar e^{-\frac{c}{\hbar}} \rpar
\end{align*}
where

\[
S_m = \lacc (j,k,l)\in\Nm^3 /\; j+k+l=m,\; j,k,l \le \frac{1}{3D\hbar}\racc.
\]
We write $D'= \max \lpar (\rho-\sigma)^{-2}, A,B \rpar$, and suppose $D'<D$. Using Lemma \ref{le_Cauchy_ana} and the fact that $f$ and $g$ are analytic symbols, for all $|y|<\sigma$

\begin{align*}
\lver \sum\limits_{(j,k,l)\in S_m} \frac{(-1)^j}{j!} \partial^j f_k(y) \overline{\partial}^j g_l(y) \rver
	& \le C \sum\limits_{(j,k,l)\in S_m} (\rho-\sigma)^{-2j} j! k! A^k l! B^l\\
	& \le C (D')^m \sum\limits_{(j,k,l)\in S_m}  j! k! l!\\
	& \le C m(m+1)! (D')^m\\
	& \le C m! D^m
\end{align*}
so $f\# g$ is an analytic symbol on $B_{\sigma}$.
\end{proof}

This last result shows how complicated the computations on $S(B_{\rho})$ can get. We follow an idea of Boutet de Monvel and Krée to define formal norms on $S(B_{\rho})$ to control the parameters $\rho$ and $A$ with more ease. The results of this Section regarding these formal norms are derived from their works \cite{bout67,bout72}.

\begin{definition}
For $a\in S(B_{\rho})$ we define
\[
\lnor a \rnor_{\rho} = \sum\limits_{\alpha,\beta,k \in\Nm^3} \lpar \frac{2(2d)^{-k}k!}{(k+|\alpha|)!(k+|\beta|)!} \rpar \lver \partial^{\alpha} \overline{\partial}^{\beta} a_k(0) \rver \rho^{2k+|\alpha+\beta|}.
\]
For computations, we also consider these partial expressions
\begin{align*}
\lnor a \rnor_{\rho,s} & = \rho^s \sum\limits_{2k+|\alpha+\beta| = s} \lpar \frac{2(2d)^{-k}k!}{(k+|\alpha|)!(k+|\beta|)!} \rpar \lver \partial^{\alpha} \overline{\partial}^{\beta} a_k(0) \rver\\
\lnor a \rnor_{\rho,s-} & = \sum\limits_{0\le r\le s} \lnor a \rnor_{\rho,r}
\end{align*}
it is clear that for all $\rho$ and $s$, $\lnor \cdot \rnor_{\rho,s}$ satisfies the triangle inequality.
\end{definition}

\begin{lemma}
\label{le_pseudo_norms}
If the formal norm $\lnor a \rnor_{\rho}$ is bounded then $a_k$ is analytic at $0$ with a convergence radius greater or equal to $\rho$ and $\lver \widetilde{a_k} \rver \le \frac{1}{2} \lnor a \rnor_{\rho} (2d)^k k!$ for all $k\in\Nm$. On the other hand, if $a$ is analytic at $0$ with a convergence radius equal to $\rho$ with $\lver \widetilde{a_k} \rver \le c A^k k!$ then $\lnor a \rnor_{\rho'}$ is bounded for all $\rho'\in ]0,\rho[$. Furthermore, $\lnor a \rnor_{\rho}=0$ if and only if $a = O(e^{-\frac{c}{\hbar}})$ on $B_{\rho}$.
\end{lemma}

\begin{proof}
First, consider a sequence of functions $a_k$ such that the term $\lnor a \rnor_{\rho}$ is finite for a fixed $\rho$. Let $\rho'\in ]0,\rho[$ and $k\in\Nm$, then for $\hbar$ small enough and for all $\alpha,\beta\in\Nm^2$:

\[
\frac{1}{|\alpha|!|\beta|!} \hbar^k (\rho')^{|\alpha|+|\beta|} \le \frac{(k!)^2 \rho^{2k+\alpha+\beta}}{(|\alpha|+k)!(|\beta|+k)!}
\]
as the quotient of the left-hand side by the right-hand side is equal to

\[
\binom{\alpha+k}{\alpha} \lpar \frac{\sqrt{\hbar}}{\rho} \rpar^k \lpar \frac{\rho'}{\rho} \rpar^{\alpha} \binom{\beta+k}{\beta} \lpar \frac{\sqrt{\hbar}}{\rho} \rpar^k \lpar \frac{\rho'}{\rho} \rpar^{\beta} \le \lpar \frac{\sqrt{\hbar}+\rho'}{\rho} \rpar^{\alpha+\beta+2k} \le 1
\]
if $\hbar \le (\rho-\rho')^2$. Therefore,

\begin{align*}
	& \sum\limits_{\alpha,\beta\in\Nm^2} \hbar^k \frac{\lver \partial^{\alpha} \overline{\partial}^{\beta} a_k(0) \rver}{|\alpha|!|\beta|!} (\rho')^{|\alpha+\beta|}\\
	\le & \sum\limits_{\alpha,\beta\in\Nm^2} \frac{\lver \partial^{\alpha} \overline{\partial}^{\beta} a_k(0) \rver (k!)^2}{(|\alpha|+k)!(|\beta|+k)!} \rho^{2k+|\alpha+\beta|}\\
	\le & \frac{1}{2} k! (2d)^k \sum\limits_{\alpha,\beta \in\Nm^2} \lpar \frac{2(2d)^{-k}k!}{(k+|\alpha|)!(k+|\beta|)!} \rpar \lver \partial^{\alpha} \overline{\partial}^{\beta} a_k(0) \rver \rho^{2k+|\alpha+\beta|}\\
	\le & \frac{1}{2} k! (2d)^k \lnor a \rnor_{\rho} < +\infty
\end{align*}
so, $a_k$ is analytic at $0$ with a convergence radius greater or equal to $\rho$, and we have $|\widetilde{a_k}(z,\overline{v})| \le \frac{1}{2}\lnor a\rnor_{\rho} k! (2d)^k$ for all $|z|,|v|\le \rho$.

Now, let us suppose that for all $k\in\Nm$, $a_k$ is analytic with convergence radius greater or equal to $\rho$ and $|\widetilde{a_k}| \le c A^k k!$ on $\widetilde{B_{\rho}}$. Notice that $A$ can be taken arbitrarily large. Let $\rho' \in]0,\rho[$, then using Lemma \ref{le_Cauchy_ana} on $B_{\rho}$ for all $k\in\Nm$,
\[
\lver \partial^{\alpha} \overline{\partial}^{\beta} a_k (0) \rver \le C \alpha! \beta! A^k k! \rho^{-|\alpha+\beta|}.
\]
Then,
\begin{align*}
\lnor a \rnor_{\rho'} & \le C \sum\limits_{\alpha,\beta,k \in\Nm^3} \frac{2(2d)^{-k}k!}{(k+|\alpha|)!(k+|\beta|)!} \alpha! \beta! k! (\rho'^2 A)^k \lpar \frac{\rho'}{\rho} \rpar^{|\alpha+\beta|}\\
							& \le C \sum\limits_{\alpha,\beta,k \in\Nm^3} \lpar \frac{\rho'^2 A}{2d} \rpar^k \lpar \frac{\rho'}{\rho} \rpar^{|\alpha+\beta|}\\
							& < +\infty
\end{align*}
if $\rho'< \min\lpar \rho,\sqrt{\frac{2d}{A}} \rpar$, so it is satisfied if $A$ is taken large enough.
\end{proof}

The coefficients in the definition of $\lnor a\rnor_{\rho}$ are actually chosen to make the next lemma work.

\begin{lemma}
\label{le_ana_prod}
For all $f,g\in S(B_{\rho})$,
\[
\lnor f g \rnor_{\rho,s-} \le \lnor f\# g \rnor_{\rho,s-} \le \lnor f \rnor_{\rho,s-} \lnor g \rnor_{\rho,s-}.
\]
\end{lemma}

\begin{proof}
The first inequality is due to the definition of $f\#g$ and the formal norms. Let $k\in\Nm$, then

\[
\lpar \sum \frac{\hbar^j}{j!} \partial^j f \overline{\partial}^j g \rpar_k = \sum\limits_{0\le j \le k} \sum\limits_{ 0\le l \le k-j} \frac{1}{j!} \partial^jf_l \overline{\partial}^j g_{k-j-l}
\]
then we can compute

\begin{align*}
	& \lnor f\# g \rnor_{\rho,r}\\
	\le & \rho^r \sum\limits_{S_i} \frac{2(2d)^{-k}k!}{(k+|\alpha|)!(k+|\beta|)!} \binom{\alpha}{m} \binom{\beta}{n} \frac{1}{j!} \lver \partial^{j+m} \overline{\partial}^{n}f_l(0) \rver\lver \partial^{\alpha-m} \overline{\partial}^{j+\beta-n} g_{k-j-l}(0) \rver\\
	\le & \sum\limits_{S_f} \rho^{2k_2+\alpha_2+\beta_2} \rho^{2k_1+\alpha_1+\beta_1} \mu_{1,2} C_{\alpha_1,\beta_1,k_1} \lver \partial^{\alpha_1} \overline{\partial}^{\beta_1} f_{k_1}(0) \rver C_{\alpha_2,\beta_2,k_2} \lver \partial^{\alpha_2} \overline{\partial}^{\beta_2} g_{k_2}(0) \rver\\
	\le & \sum\limits_{0\le t \le r} \lnor f \rnor_{\rho,t} \lnor g \rnor_{\rho,r-t}.
\end{align*}
where we wrote $C_{\alpha,\beta,k} = \frac{2(2d)^{-k}k!}{(k+|\alpha|)!(k+|\beta|)!}$,

\begin{multline*}
\mu_{1,2} = \frac{1}{2} \sum\limits_{\gamma=0}^{\min(\beta_1,\alpha_2)} (2d)^{-\gamma}\binom{k_1+\gamma}{\gamma} \binom{\alpha_1+\alpha_2-\gamma}{\alpha_2-\gamma} \binom{k_2+\alpha_2+k_1+\alpha_1}{k_2+\alpha_2}^{-1}\\
\binom{k_2+\beta_2+k_1+\beta_1}{k_2+\beta_2}^{-1} \binom{k_1+k_2+\gamma}{k_2} \binom{\beta_1+\beta_2-\gamma}{\beta_2},
\end{multline*}
and
\begin{align*}
S_i & = \lacc (k,\alpha,\beta,j,l,m,n)\in\Nm^7 /\; 2k+\alpha+\beta=r,j\le k,l+j\le k,m\le \alpha, n\le \beta \racc\\
S_f & = \lacc (k_1,\alpha_1,\beta_1,k_2,\alpha_2,\beta_2,\gamma)\in\Nm^7 /\; 2k_1+2k_2+\alpha_1+\alpha_2+\beta_1+\beta_2 = r, \gamma \le \beta_1,\alpha_2 \racc.
\end{align*}
The inequality comes from two points, first $\mu_{1,2} \le 1$ because $\binom{m}{k}\binom{n}{l} \le \binom{m+n}{k+l}$, and $\binom{m-j}{k-j} \le \binom{m}{k}$ for all $k\le m$, $l\le n$, $l\le k,m$. The second point is that we have the following injection

\begin{align*}
S_i \hookrightarrow & S_f\\
(k,\alpha,\beta,j,l,m,n) \mapsto & (k-j-l,\alpha-m,j+\beta-n,l,j+m,n,j)\\
			& = (k_1,\alpha_1,\beta_1,k_2,\alpha_2,\beta_2,\gamma).
\end{align*}
Then, summing the inequalities over $0\le r\le s$ gives
\[
\lnor f\# g \rnor_{\rho,s-} \le \sum\limits_{0\le t+r \le s} \lnor f \rnor_{\rho,t} \lnor g \rnor_{\rho,r} \le \lnor f \rnor_{\rho,s-} \lnor g \rnor_{\rho,s-}.
\]
\end{proof}

From now on we will write

\begin{equation}
\label{eq_Poisson}
\lacc f,g \racc = i\lpar \partial g \overline{\partial} f - \overline{\partial} g \partial f \rpar
\end{equation}
the Poisson bracket, and $ \partial_{\theta} f = i\lpar z\partial f - \overline{z} \overline{\partial} f \rpar = \lacc |z|^2,f \racc$. If we consider $z = \frac{x-iy}{\sqrt{2}}$ then this notation is consistent with the usual Poisson bracket, and $\partial_{\theta}f$ can be seen as the differential of $f$ with respect to the variable $\theta= \arg(z)$. We also consider a symbolic bracket as $[f,g]_{\#} = f\#g - g\#f$. We will now prove a series of Lemmas in order to solve a cohomology equation of the form $(\partial\overline{\partial}\Phi_M)^{-1}\{\mu,b\} = g$ with unknown $b$ in a space of analytic symbols, as we will encounter it multiple times in the article. Notice that it can also be written as a transport equation along the flow of $\mu$ on a manifold of Kähler potential $\Phi_M$, however this point of view is inconvenient for analytic symbols.

\begin{lemma}
\label{le_ana_zero_avrg}
Let $\rho>0$ and $f\in S(B_{\rho})$ such that $\lnor f\rnor_{\rho}<\infty$ and $\partial^{\alpha}\overline{\partial}^{\alpha}f(0)=0$ for all $\alpha\in\Nm$. Then $\lnor f \rnor_{\rho,s} \le \lnor \partial_{\theta} f \rnor_{\rho,s}$ for all $s\in\Nm$. Moreover, there exists a unique $g\in S(B_{\rho})$, solution to $\partial_{\theta} g =f$ such that $\partial^{\alpha}\overline{\partial}^{\alpha}g(0)=0$ for all $\alpha\in\Nm$, we write it $\int f \frac{d\theta}{2\pi}$, and we deduce that
\[
\lnor \int f\frac{d\theta}{2\pi} \rnor_{\rho,s} \le \lnor f \rnor_{\rho,s}.
\]
\end{lemma}

\begin{proof}
For any $\alpha,\beta \in \Nm^2$, we compute 

\begin{align*}
\partial^{\alpha} \overline{\partial}^{\beta} \partial_{\theta} f(0)
	& = i \lpar \overline{\partial}^{\beta} \sum\limits_{0\le k\le \alpha} \binom{\alpha}{k} \partial^k z \partial^{\alpha+1-k} f \rpar (0) - i \lpar \partial^{\alpha} \sum\limits_{0\le k\le \beta} \binom{\beta}{k} \overline{\partial}^k \overline{z} \overline{\partial}^{\beta+1-k} f \rpar (0)\\
	& = i(\alpha-\beta) \partial^{\alpha} \overline{\partial}^{\beta} f(0).
\end{align*}
By hypothesis, $\partial^{\alpha} \overline{\partial}^{\alpha} f(0) = 0 = \partial^{\alpha} \overline{\partial}^{\alpha} \partial_{\theta}f(0)$ for all $\alpha\in\Nm$, and if $\alpha, \beta \in \Nm^2$ are distinct, $\lver \partial^{\alpha} \overline{\partial}^{\beta} f(0) \rver = \frac{1}{\alpha-\beta} \lver \partial^{\alpha} \overline{\partial}^{\alpha} \partial_{\theta}f(0) \rver \le \lver \partial^{\alpha} \overline{\partial}^{\alpha} \partial_{\theta}f(0) \rver $, hence $\lnor f \rnor_{\rho,s} \le \lnor \partial_{\theta} f \rnor_{\rho,s}$.

According to Lemma \ref{le_pseudo_norms}, for all $z\in B_{\rho}$

\[
f(z) = \sum\limits_{\alpha,\beta\in\Nm^2} \frac{\partial^{\alpha}\overline{\partial}^{\beta}f(0)}{\alpha!\beta!} z^{\alpha} \overline{z}^{\beta},
\]
and by hypothesis $\partial^{\alpha}\overline{\partial}^{\alpha}f(0)=0$ for all $\alpha\in\Nm$. Hence, the expression

\[
g(z) = \sum\limits_{\alpha,\beta\in\Nm^2} \frac{\partial^{\alpha}\overline{\partial}^{\beta}f(0)}{i(\alpha-\beta)\alpha!\beta!} z^{\alpha} \overline{z}^{\beta},
\]
is well-defined on $B_{\rho}$.

\end{proof}

\begin{remark}
It is possible to choose more general formal norms by taking a function of $z\in\Omega$
\[
\lnor a \rnor_{\rho} (z) = \sum\limits_{\alpha,\beta,k \in\Zm^3} \lpar \frac{2(2n)^{-k}k!}{(k+|\alpha|)!(k+|\beta|)!} \rpar \lver \partial^{\alpha} \overline{\partial}^{\beta} a_k(z) \rver \rho^{2k+|\alpha+\beta|}
\]
with $\Omega \subset \Cm^d$ an open set. Although, this last lemma only works by taking its evaluation at $0$.
\end{remark}

\begin{lemma}
\label{le_ana_sympl}
If $\kappa$ is a linear symplectomorphism, then
\[
\lnor \widetilde{f} \circ \widetilde{\kappa}|_{\diag} \rnor_{\rho,s} \le \lpar 2\sqrt{2} \lnor \kappa \rnor \rpar^s \lnor f \rnor_{\rho,s}
\]
therefore, there exists $\tau\in ]0,\rho[$ such that
\[
\lnor \widetilde{f} \circ \widetilde{\kappa}|_{\diag} \rnor_{\tau,s} \le  \lnor f \rnor_{\rho,s}
\]
\end{lemma}

\begin{proof}
We denote $(\kappa_{j,k})_{1\le j,k\le 2} = \lpar\lver \partial_k \widetilde{\kappa}_j(0) \rver\rpar_{1\le j,k\le 2}$, then we compute
\[
\lver \partial_1^{\alpha} \partial_2^{\beta} (\widetilde{f_k}\circ\widetilde{\kappa})(0) \rver
	\le \sum\limits_{0\le j\le \alpha} \sum\limits_{0\le l\le \beta} \binom{\alpha}{j} \binom{\beta}{l} \kappa_{1,1}^j \kappa_{2,1}^{\alpha-j} \kappa_{1,2}^l \kappa_{2,2}^{\beta-l} \lver\partial_1^{j+l}\partial_2^{\alpha+\beta-j-l} \widetilde{f_k}(0)\rver,
\]
thus
\[
\lnor \widetilde{f}\circ\widetilde{\kappa} \rnor_{\rho,s}
	\le \rho^s \sum\limits_{2k+\alpha+\beta=s} \sum\limits_{0\le j\le \alpha} \sum\limits_{0\le l\le \beta} C_{\alpha,\beta,k} \binom{\alpha}{j} \binom{\beta}{l} \kappa_{1,1}^j \kappa_{2,1}^{\alpha-j} \kappa_{1,2}^l \kappa_{2,2}^{\beta-l} \lver \partial_1^{j+l}\partial_2^{\alpha+\beta-j-l} \widetilde{f_k}(0) \rver.
\]
We apply the change of variable $(j,l,\alpha,\beta) \mapsto (j,\beta-l,j+l,\alpha+\beta-j-l)=(j',l',\alpha',\beta')$, which gives

\begin{align*}
	& \lnor \widetilde{f}\circ\widetilde{\kappa} \rnor_{\rho,s}\\
	= & \rho^s \sum C_{\beta'+j'-l',\alpha'+l'-j',k} \binom{\beta'+j'-l'}{j'} \binom{\alpha'+l'-j'}{\alpha'-j'} \kappa_{1,1}^{j'} \kappa_{2,1}^{\beta'-l'} \kappa_{1,2}^{\alpha'-j'} \kappa_{2,2}^{l'} \lver \partial_1^{\alpha'}\partial_2^{\beta'} \widetilde{f_k}(0) \rver
\end{align*}
where the sum is over the indexes such that $2k+\alpha'+\beta'=s$, $0\le j'\le \alpha'$ and $0\le l'\le \beta'$. We then use the inequality
\begin{align*}
	& \binom{\beta'+j'-l'}{j'} \binom{\alpha'+l'-j'}{\alpha'-j'}  \frac{C_{\beta'+j'-l',\alpha'+l'-j',k}}{C_{\alpha',\beta',k}}\\
	= & \binom{\alpha'}{j'} \binom{\beta'}{l'} \binom{k+\alpha'}{k} \binom{k+\alpha'+l'-j'}{k}^{-1} \binom{k+\beta'}{k} \binom{k+\beta'+j'-l'}{k}^{-1}\\
	\le & \binom{\alpha'}{j'} \binom{\beta'}{l'} 2^{k+\alpha'} \binom{k+\alpha'+l'-j'}{k}^{-1} 2^{k+\beta'} \binom{k+\beta'+j'-l'}{k}^{-1}\\
	\le & \binom{\alpha'}{j'} \binom{\beta'}{l'} 2^s
\end{align*}
then
\begin{align*}
2^s \sum\limits_{0\le j'\le \alpha'} \sum\limits_{0\le l'\le \beta'} \binom{\alpha'}{j'} \binom{\beta'}{l'} \kappa_{1,1}^{j'} \kappa_{2,1}^{\beta'-l'} \kappa_{1,2}^{\alpha'-j'} \kappa_{2,2}^{l'}
	& = 2^s (\kappa_{1,1}+\kappa_{1,2})^{\alpha'} (\kappa_{2,1}+\kappa_{2,2})^{\beta'}\\
	& \le 2^s \lpar 2 \lnor \kappa \rnor^2 \rpar^{\frac{\alpha'+\beta'}{2}}\\
	& \le \lpar 2\sqrt{2} \lnor \kappa \rnor \rpar^s
\end{align*}
and combining the inequalities
\[
\lnor \widetilde{f}\circ\widetilde{\kappa} \rnor_{\rho,s} \le \lpar 2\sqrt{2}\lnor \kappa \rnor \rpar^s \lnor f \rnor_{\rho,s}.
\]
\end{proof}

\begin{lemma}
\label{le_ana_inverse}
Let $\rho>0$ and $f\in S(B_{\rho})$ such that $\lnor f\rnor_{\rho} <\infty$ and $f_0(0) \neq 0$, then there exists $\tau\in ]0,\rho]$ such that $\lnor f^{-1} \rnor_{\tau} \le \frac{1}{1-\lnor f-f_0(0)\rnor_{\tau}}$.
\end{lemma}

\begin{proof}
$f^{-1}$ is well-defined near $0$ for $\hbar$ small enough since $f_0(0) \neq 0$. Moreover, there exists $\tau\in ]0,\rho]$ such that $\lnor f-f_0(0) \rnor_{\tau} <1$ since $\lim\limits_{\tau\to 0}\lnor f-f_0(0) \rnor_{\tau} = 0$, and the inequality comes from the power series expansion of $x\mapsto \frac{1}{1-x}$.
\end{proof}

\begin{proposition}
\label{prop_poiss_quad}
Let $\rho>0$, $g\in S(B_{\rho})$, $\Phi_M$ be a real-analytic plurisubharmonic function independent of $\hbar$, $H$ be a non-degenerate complex quadratic function, and $\mu$ be holomorphic near $0$. Suppose that $\lnor g\rnor_{\rho}, \lnor\mu(H)\rnor_{\rho}<\infty$, and $\mu_0 = \id +O(|z|^2)$. Then there exists $C_M,C>0$ depending only on $\Phi_M$ and $H$ respectively, $\tau\in ]0,\rho[$, $r$ an analytic function, and $b\in S(B_{\tau})$ such that
\[
(\partial\overline{\partial}\Phi_M)^{-1}\lacc \mu(H),b \racc = g-r(H).
\]
Moreover $\lnor r(H)\rnor_{\tau,s} \le C_MC^s\lnor g\rnor_{\tau,s}$, and
\[
\lnor \widetilde{b}\circ\widetilde{\kappa}|_{\diag} \rnor_{\tau,s-}
	\le \frac{C_M}{1-C^s\lnor (\mu-\id)(H)\rnor_{\tau,s-}} \lnor \widetilde{g}\circ\widetilde{\kappa}|_{\diag} \rnor_{\tau,s-}
\]
for all $s\in\Nm$, with $\tau=\rho$, $C=1$ if $H=|z|^2$, and $C_M=1$ if $\Phi_M=|z|^2$.
\end{proposition}

\begin{proof}
There exists a symplectomorphism $\widetilde{\kappa}$ of $\Cm^2$ such that $\widetilde{H}\circ \widetilde{\kappa}(z,\overline{v}) = dz\overline{v}$, thus the equation becomes

\begin{align*}
	& \lpar\widetilde{\partial\overline{\partial}\Phi_M}\circ\widetilde{\kappa}\rpar^{-1}\lacc \mu\lpar\widetilde{H}\rpar,\widetilde{b} \racc \circ \widetilde{\kappa} = \widetilde{g}\circ \widetilde{\kappa}-r(dz\overline{v})\\
	\Leftrightarrow & \lacc \mu\lpar\widetilde{H}\circ \widetilde{\kappa}\rpar,\widetilde{b}\circ \widetilde{\kappa} \racc = \partial_1\partial_2\widetilde{\Phi_M}\circ\widetilde{\kappa} \lpar \widetilde{g}\circ \widetilde{\kappa}-r(dz\overline{v})\rpar\\
	\Leftrightarrow & \mu'(d|z|^2)\partial_{\theta} \lpar \widetilde{b}\circ \widetilde{\kappa}|_{\diag} \rpar = \partial_1\partial_2\widetilde{\Phi_M}\circ\widetilde{\kappa}|_{\diag} \lpar \widetilde{g}\circ \widetilde{\kappa}|_{\diag}-r(d|z|^2) \rpar\\
	\Leftrightarrow & \partial_{\theta} \lpar \widetilde{b}\circ \widetilde{\kappa}|_{\diag} \rpar = \partial_1\partial_2\widetilde{\Phi_M}\circ\widetilde{\kappa}|_{\diag} \frac{\widetilde{g}\circ\widetilde{\kappa}|_{\diag}-r(d|z|^2)}{\mu'(d|z|^2)}.
\end{align*}
Hence, we take

\[
r(d|z|^2) = \lpar \int_{|z|\Sm} \partial_1\partial_2\widetilde{\Phi_M}\circ \widetilde{\kappa} \frac{d\theta}{2\pi} \rpar^{-1} \int_{|z|\Sm} \lpar\partial_1\partial_2\widetilde{\Phi_M}\widetilde{g}\rpar\circ \widetilde{\kappa} \frac{d\theta}{2\pi},
\]
where we write $\int_{|z|\Sm} \widetilde{f} \frac{d\theta}{2\pi} = \sum\limits_{\alpha\in\Nm} \frac{\partial^{\alpha}\overline{\partial}^{\alpha}f(0)}{(\alpha!)^2} |z|^{2\alpha}$. According to Lemma \ref{le_pseudo_norms}, $\int_{|z|\Sm} \partial_1\partial_2\widetilde{\Phi_M}\circ \widetilde{\kappa} \frac{d\theta}{2\pi}$ and $\int_{|z|\Sm} \lpar\partial_1\partial_2\widetilde{\Phi_M}\widetilde{g}\rpar\circ \widetilde{\kappa} \frac{d\theta}{2\pi}$ are well-defined for all $z\in B_{\rho}$. Now, since $\Phi_M$ is analytic and independent of $\hbar$, there exists $\tau>0$ such that $\lnor \partial\overline{\partial}\Phi_M \rnor_{\tau}<\infty$, and since $\partial\overline{\partial}\Phi_M(0)\neq 0$ by hypothesis, according to Lemma \ref{le_ana_inverse},

\[
\lnor \lpar \int_{|z|\Sm} \partial_1\partial_2\widetilde{\Phi_M}\circ \widetilde{\kappa} \frac{d\theta}{2\pi} \rpar^{-1} \rnor_{\tau}
	\le \frac{1}{1-\lnor \partial_1\partial_2\widetilde{\Phi_M}-\partial_1\partial_2\widetilde{\Phi_M}(0) \rnor_{\tau}},
\]
for $\tau\in ]0,\rho[$ small enough. Furthermore, according to Lemma \ref{le_ana_sympl}, there exists $C>0$ such that

\[
\lnor \int_{|z|\Sm} \lpar\partial_1\partial_2\widetilde{\Phi_M}\widetilde{g}\rpar\circ \widetilde{\kappa} \frac{d\theta}{2\pi} \rnor_{\tau,s}
	\le C^s \lnor \partial\overline{\partial}\Phi_M \rnor_{\tau,s} \lnor g \rnor_{\tau,s},
\]
for $\tau$ small enough. Thus, $\lnor r(H) \rnor_{\tau,s} \le C_M C^s \lnor g\rnor_{\tau,s}$ for $\tau$ small enough and $C_M>0$ depending on $\Phi_M$. In the same manner, according to Lemma \ref{le_ana_zero_avrg}, $b$ is uniquely defined as well with

\begin{align*}
\lnor \widetilde{b}\circ\widetilde{\kappa}|_{\diag} \rnor_{\tau,s-}
	\le & C_M \lnor \frac{1}{\mu'(|z|^2)} \rnor_{\tau,s-} \lnor \widetilde{g}\circ\widetilde{\kappa}|_{\diag} \rnor_{\tau,s-}\\
	\le & \frac{C_M}{1-\lnor (\mu'-1)(|z|^2)\rnor_{\tau,s-}} \lnor \widetilde{g}\circ\widetilde{\kappa}|_{\diag} \rnor_{\tau,s-}\\
	\le & \frac{C_M}{1-\lnor (\mu-\id)(|z|^2)\rnor_{\tau,s-}} \lnor \widetilde{g}\circ\widetilde{\kappa}|_{\diag} \rnor_{\tau,s-}\\
	\le & \frac{C_M}{1-C^s\lnor (\mu-\id)(H)\rnor_{\tau,s-}} \lnor \widetilde{g}\circ\widetilde{\kappa}|_{\diag} \rnor_{\tau,s-}
\end{align*}
for $\tau$ small enough according to Lemma \ref{le_ana_inverse} and \ref{le_ana_sympl}.
\end{proof}

\begin{lemma}
\label{le_ana_bracket}
Let $f,g\in S(B_{\rho})$ then $\lnor [f,g]_{\#} \rnor_{\rho,s} = 0$  for $s=0$ and $1$, and if $s\ge 2$
\[
\lnor [f,g]_{\#} \rnor_{\rho,s} \le 2\sum\limits_{r=1}^{s-1} \lnor f \rnor_{\rho,r} \lnor g \rnor_{\rho,s-r}
\]
furthermore,
\[
\lnor [f,g]_{\#} - i\{f,g\} \rnor_{\rho,s-} \le 2 \lnor f \rnor_{\rho,(s-2)-} \lnor g \rnor_{\rho,(s-2)-}
\]
in particular, for $g = \mu(H)$ with $\mu$ analytic and $H$ quadratic
\[
\lnor [f,g]_{\#} - i\{f,g\} \rnor_{\rho,s-} \le \lver \partial\overline{\partial}H \rver \rho^2 \lnor \mu' \rnor_{\rho,(s-4)-} \lnor f \rnor_{\rho,(s-2)-}
\]
\end{lemma}

\begin{proof}
It is similar to the proof of Lemma \ref{le_ana_prod} but with sums over more specific sets. Let $\epsilon$ be equal to $0$ for the first inequality and $1$ for the second,

\begin{align*}
	& \lnor [f,g]_{\#} - \epsilon i\{f,g\} \rnor_{\rho,r}\\
	\le & \rho^r \sum\limits_{S_i\cap\{j\ge 1+\epsilon\}}  \frac{2(2d)^{-k}k!}{(k+|\alpha|)!(k+|\beta|)!} \binom{\alpha}{m} \binom{\beta}{n} \frac{1}{j!} \lver \partial^{j+m} \overline{\partial}^{n}f_l(0) \rver\lver \partial^{\alpha-m} \overline{\partial}^{j+\beta-n} g_{k-j-l}(0) \rver\\
	& + \rho^r \sum\limits_{S_i\cap\{j\ge 1+\epsilon\}}  \frac{2(2d)^{-k}k!}{(k+|\alpha|)!(k+|\beta|)!} \binom{\alpha}{m} \binom{\beta}{n} \frac{1}{j!} \lver \partial^{j+m} \overline{\partial}^{n}g_l(0) \rver\lver \partial^{\alpha-m} \overline{\partial}^{j+\beta-n} f_{k-j-l}(0) \rver\\
	\le & 2\sum\limits_{t=1+\epsilon}^{r-1-\epsilon} \sum\limits_{2k_1+\alpha_1+\beta_1 = t} \rho^t \sum\limits_{2k_2+\alpha_2+\beta_2 = r-t} \rho^{r-t} \mu_{1,2} C_{\alpha_1,\beta_1,k_1} \lver \partial^{\alpha_1} \overline{\partial}^{\beta_1} f_{k_1}(0) \rver C_{\alpha_2,\beta_2,k_2} \lver \partial^{\alpha_2} \overline{\partial}^{\beta_2} g_{k_2}(0) \rver\\
	\le & 2\sum\limits_{t=1+\epsilon}^{r-1-\epsilon} \lnor f \rnor_{\rho,t} \lnor g \rnor_{\rho,r-t}.
\end{align*}
because we saw that $\mu_{1,2} \le 1$ and the injection is now given by

\begin{align*}
S_i\cap\{j\ge 1+\epsilon\} \hookrightarrow & S_f\cap\{1+\epsilon\le 2k_1+\alpha_1+\beta_1\le r-1-\epsilon\}\\
(k,\alpha,\beta,j,l,m,n) \mapsto & (k-j-l,\alpha-m,j+\beta-n,l,j+m,n,j)\\
				& = (k_1,\alpha_1,\beta_1,k_2,\alpha_2,\beta_2,\gamma)
\end{align*}
where we still have

\begin{align*}
S_i & = \lacc (k,\alpha,\beta,j,l,m,n)\in\Nm^7 /\; 2k+\alpha+\beta=r,j\le k,l+j\le k,m\le \alpha, n\le \beta \racc\\
S_f & = \lacc (k_1,\alpha_1,\beta_1,k_2,\alpha_2,\beta_2,\gamma)\in\Nm^7 /\; 2k_1+2k_2+\alpha_1+\alpha_2+\beta_1+\beta_2 = r, \gamma \le \beta_1,\alpha_2 \racc.
\end{align*}
Then summing the inequality over $0\le r\le s$ gives
\[
\lnor [f,g]_{\#} - \epsilon i\{f,g\} \rnor_{\rho,s-}
	\le 2\sum\limits_{1+\epsilon\le t+r \le s-1-\epsilon} \lnor f \rnor_{\rho,t} \lnor g \rnor_{\rho,r}
	\le 2\lnor f \rnor_{\rho,(s-1-\epsilon)-} \lnor g \rnor_{\rho,(s-1-\epsilon)-}.
\]

If $g=\mu(H)$ we compute

\begin{align*}
\lnor \mu(H) \rnor_{\rho,s} 
	& = \rho^s \sum\limits_{2k+\alpha+\beta=s} C_{\alpha,\beta,k} \lver \partial^{\alpha} \overline{\partial}^{\beta} \mu_k(H)(0) \rver\\
	& = |\partial\overline{\partial}H| \rho^s \sum\limits_{2k+\alpha+\beta=s} C_{\alpha,\beta,k} \lver \partial^{\alpha-1} \overline{\partial}^{\beta-1} \mu'_k(H)(0) \rver\\
	& = |\partial\overline{\partial}H| \rho^2 \rho^{s-2} \sum\limits_{2k+\alpha+\beta=s-2} C_{\alpha+1,\beta+1,k} \lver \partial^{\alpha} \overline{\partial}^{\beta} \mu'_k(H)(0) \rver\\
	& \le |\partial\overline{\partial}H| \rho^2 \lnor \mu'(H) \rnor_{\rho,s-2}
\end{align*}
because
\[
\frac{C_{\alpha+1,\beta+1,k}}{C_{\alpha,\beta,k}} = \frac{1}{(k+\alpha+1)!(k+\beta+1)!} \le 1.
\]
\end{proof}

\begin{proposition}
Let $a\in S(B_{\rho})$ be a symbol of order $1$, meaning $a_0 = 0$, then there exists a symbol $(1+a)^{\#-1} = 1+a^*$ such that $(1+a)\#(1+a)^{\#-1} = 1$. Furthermore, for $\rho$ small enough
\begin{align*}
\lnor a^* \rnor_{\rho,s-} & \le 2 \lnor a \rnor_{\rho,s-}\\
\lnor (1+a)^{\#-1} \rnor_{\rho} & \le 2
\end{align*}
\end{proposition}

\begin{proof}
By iteration, $a^*$ is defined by
\[
a^*_k = -a_k - (a\#a^*)_k = -a_k - \sum\limits_{0\le l\le k-2} \sum\limits_{1\le j\le k-l-1} \frac{\hbar^l}{l!} \partial^l a_j \overline{\partial}^l a^*_{k-l-j}
\]
thus $\lnor a^* \rnor_{\rho,(s+1)-} \le \lnor a \rnor_{\rho,(s+1)-} + \lnor a \rnor_{\rho,s-} \lnor a^* \rnor_{\rho,s-}$. Since $a$ is of order $1$, $\lnor a \rnor_{\rho} \xrightarrow[\rho \to 0]{} 0$, so we can take $\rho$ such that $\lnor a \rnor_{\rho} \le \frac{1}{2}$ and by iteration $\lnor a^* \rnor_{\rho,s-} \le 2 \lnor a \rnor_{\rho,s-}$. This inequality implies the second one.
\end{proof}

\begin{lemma}
\label{le_func_operator}
Consider $\mu = \sum\limits_{k=0}^{\frac{1}{A\hbar}} \hbar^k \mu_k$ an analytic symbol on a neighbourhood of $0$ in $\Cm$. Then, there exists $\mu_b \in S(B_{\rho})$ such that $\mu(T_{\hbar}(|z|^2)) = T_{\hbar}(\mu_b(|z|^2)) + O(e^{-\frac{c}{\hbar}})$. Conversely, if $\mu_b \in S(B_{\rho})$, then there exists an analytic symbol $\mu$ on a neighbourhood of $0$ such that $\mu(T_{\hbar}(|z|^2)) = T_{\hbar}(\mu_b(|z|^2)) + O(e^{-\frac{c}{\hbar}})$.
\end{lemma}

\begin{proof}
For a fixed $\rho$, there exists $c>0$ such that for all $k\in\Nm$, by iteration $T_{\hbar}(|z|^2)^k = T_{\hbar}\lpar (|z|^2)^{\# k} \rpar + kO_{\barg(B_{\rho})\circlearrowleft}(e^{-\frac{c}{\hbar}})$. By hypothesis, there exists $\rho,C>0$ such that each $\mu_k$ has a convergence radius larger than $\rho$ and $|\mu_k(x)|\le Ck!A^k$ for all $|x|\le\rho$. Then, according to Lemma \ref{le_Cauchy_ana}

\[
\lver \mu_k(x) - \sum\limits_{j=0}^{\frac{1}{\hbar}}\frac{\mu_k^{(j)}(0)}{j!} x^j \rver \le C k! A^k \lpar\frac{|x|}{\rho}\rpar^{\frac{1}{\hbar}}
\]
and, since $\lnor |z|^2\rnor_{\rho} =2\rho^2$, taking $\rho<\frac{1}{2}$

\[
\mu(T_{\hbar}(|z|^2)) = \sum\limits_{k=0}^{\frac{1}{A\hbar}} \sum\limits_{j=0}^{\frac{1}{\hbar}} \lpar \hbar^k \frac{\mu_k^{(j)}(0)}{j!} T_{\hbar}\lpar (|z|^2)^{\# j} \rpar \rpar + O_{\barg(B_{\rho})\circlearrowleft}(e^{-\frac{c}{\hbar}}).
\]
Hence, $\mu(T_{\hbar}(|z|^2)) = T_{\hbar}(\mu_b(|z|^2))$ where $\mu_b$ is the analytic function given by

\[
\mu_b(|z|^2)
	= \sum\limits_{k=0}^{\frac{1}{A\hbar}} \sum\limits_{j=0}^{\frac{1}{\hbar}} \hbar^k \frac{\mu_k^{(j)}(0)}{j!} (|z|^2)^{\# j}.
\]
It remains to prove that it is an analytic symbol. The partial formal norms are bounded by

\begin{align*}
\lnor \mu_b(|z|^2) \rnor_{\rho,s-}
	\le & \sum\limits_{k=0}^{\frac{1}{A\hbar}} \sum\limits_{j=0}^{\frac{1}{\hbar}} \hbar^k \frac{|\mu_k^{(j)}(0)|}{j!} \lnor(|z|^2)^{\# j}\rnor_{\rho,s-}\\
	\le & \sum\limits_{k=0}^{\frac{1}{A\hbar}} \sum\limits_{j=0}^{\frac{1}{\hbar}} \hbar^k CA^k k!\rho^{-j} \lnor(|z|^2)\rnor_{\rho,s-}^j\\
	\le & C\sum\limits_{k=0}^{\frac{1}{A\hbar}} (\hbar A)^k k! \sum\limits_{j=0}^{\frac{1}{\hbar}} \rho^{-j} (2\rho^2)^j\\
\end{align*}
which is finite for $\hbar$ and $\rho$ small enough. For the second implication, consider $\mu_b$ fixed, then $\mu$ is given by

\[
\sum\limits_{k=0}^{\frac{1}{A\hbar}} \sum\limits_{j=0}^{\frac{1}{\hbar}} \hbar^k \frac{\mu_k^{(j)}(0)}{j!} (|z|^2)^{\# j} = \mu_b(|z|^2)
\]
which means that for any $l\le \frac{1}{A\hbar}$ fixed,

\[
\mu_l(|z|^2) = \mu_{b,l}(|z|^2) - \sum\limits_{k=0}^{l-1} \sum\limits_{j=0}^{\frac{1}{\hbar}} \frac{\mu_k^{(j)}(0)}{j!} (|z|^2)^{\# j}_{l-k}.
\]
The equation defines $\mu$ by iteration, and it is an analytic symbol near $0$.
\end{proof}

\section{Local study}
\label{sec_local}

We consider in this Section the following framework:

\begin{hypothesis}
\label{hyp_local}
~
\begin{itemize}
\item Either $M$ is a real-analytic compact manifold and:
	\begin{itemize}
	\item $(J,\omega)$ is a Kähler structure for $M$.
	\item $L\rightarrow M$ is a prequantised line bundle, and for $N\in\Nm^*$, $\Hc_N = H^0(M,L^{\otimes N})$ is the quantum space on $M$, meaning the space of holomorphic sections of $L^{\otimes N}\rightarrow M$.
	\item $f:M \rightarrow \Cm$ is a real-analytic, complex valued function, and it is asymptotically equivalent to an analytic symbol on a neighbourhood of $x_0$. 
	\end{itemize}
\item Or $M=\Cm$ and:
	\begin{itemize}
	\item $\Cm$ is equipped with the Kähler structure recalled in Example \ref{ex_kahler_Cm}.
	\item There exists an order function $m$ such that $f\in S(\Cm,m)$.
	\end{itemize}
\item $\mathfrak{e}\in\Cm$ and $x_0\in M$ is the unique point such that $f(x_0)=E$.
\item $df(x_0) = 0$.
\end{itemize}
\end{hypothesis}

Further on, we will also suppose that $\hess_{x_0}(f)$ satisfies \eqref{eq_complex_quad}, and an ellipticity condition at infinity if $M=\Cm$, but it is not necessary for now.

In both cases $M=\Cm$ and $M$ compact, we denote $\partial$ and $\overline{\partial}$ the holomorphic and anti-holomorphic components of the exterior derivative, and $\Delta = \frac{1}{2}\Delta_g$ the holomorphic Laplacian, equal to the Laplace-Beltrami operator up to a constant.

\subsection{Normal form}
\label{subsec_morse}
In order to reduce the symbol to its normal form, we need a Morse Lemma to find a symplectomorphism $\widetilde{\kappa}: \widetilde{M}\rightarrow\widetilde{\Cm}$ and an analytic function $\mu$ such that $\widetilde{f}\circ\widetilde{\kappa}^{-1} = \mu\lpar \widetilde{H}\rpar$ with $H$ the Hessian of $f$ at $x_0$, seen as a quadratic function on $\Cm$. Many variants of this result exists, see for instance \cite{vey77,coli79}, here we consider analytic data. We propose a proof with a method similar to the Moser's trick, sometimes called the homotopy method.

\begin{lemma}[Morse]
Let $f$ be a real-analytic function on a neighbourhood of $x_0$ such that $H=\hess_{x_0}f$ is non-degenerate. Then, there exists a local symplectomorphism $\widetilde{\kappa}$ from a neighbourhood of $(x_0,\overline{x_0})$ in $\widetilde{M}$ to a neighbourhood of $(0,0)$ in $\widetilde{\Cm}$ and an analytic function $\mu$ on a neighbourhood of $0$ such that

\[
\widetilde{f} \circ \widetilde{\kappa}^{-1} = \mu\lpar \widetilde{H}\rpar.
\]
\end{lemma}

\begin{proof}
Denote $f_t = (1-t)H + tf$, we look for a Hamiltonian function $F_t : \widetilde{M}\rightarrow \Cm$ depending on the parameter $t$, and a function $\mu_t$ such that, writing $\widetilde{\kappa_t}$ the Hamiltonian flow generated by $F_t$

\begin{equation}
\label{eq_symp_time}
\mu_t\lpar\widetilde{H}\rpar \circ \widetilde{\kappa_t}(x,y) = \widetilde{f_t} (x,y)
\end{equation}
for $(x,y)$ in a neighbourhood of $(x_0,\overline{x_0})$. We consider $x_0=0$ here for simplicity. Using a Taylor expansion, and identifying the first order terms,

\begin{itemize}
\item $\mu_0 = \id$, $\widetilde{\kappa_0} =\id$,
\item $\mu_t(E) = E + O(E^2)$ for all $0\le t\le 1$,
\item $F_t(x,\overline{y}) = O(|(x,\overline{y})|^3)$ for all $0\le t\le 1$.
\end{itemize}
We consider a neighbourhood of $0$ on which there exists a Kähler potential $\Phi_M$. Then, the Hamiltonian vector field of a function $f$ is $X_f = i\lpar\partial\overline{\partial}\Phi_M\rpar^{-1}\lpar-\overline{\partial}f,\partial f\rpar$. Thus, the Hamiltonian flow satisfies, for any function $g$,
\[
\partial_t\lpar\widetilde{g}\circ\widetilde{\kappa_t}\rpar = \omega\lpar\widetilde{X_g},\widetilde{X_{F_t}}\rpar = \lpar\partial_1\partial_2\widetilde{\Phi_M}\rpar^{-1} i\lpar\partial_1 \widetilde{g}\partial_2 F_t-\partial_2\widetilde{g}\partial_1 F_t\rpar = \lpar\partial_1\partial_2\widetilde{\Phi_M}\rpar^{-1} \lacc F_t,\widetilde{g}\racc
\]
with the Poisson bracket defined by \eqref{eq_Poisson}, and the notation $\lacc \widetilde{f},\widetilde{g}\racc = \widetilde{\lacc f,g\racc}$. Hence, differentiating \eqref{eq_symp_time} with respect to $t$ gives

\[
\lpar \lpar\partial_1\partial_2\widetilde{\Phi_M}\rpar^{-1} \lacc F_t,\mu_t\lpar \widetilde{H}\rpar \racc + \dot{\mu_t}\lpar\widetilde{H}\rpar \rpar\circ \widetilde{\kappa_t} = \dot{\widetilde{f_t}} = \widetilde{f}-\widetilde{H}.
\]
Hence, writing $b_t = \lpar\widetilde{f}-\widetilde{H}\rpar \circ \widetilde{\kappa_t}^{-1}|_{\diag}$, the functions must satisfy the equations,

\begin{gather*}
\lpar\partial\overline{\partial}\Phi_M\rpar^{-1}\lacc F_t|_{\diag},\mu_t(H)\racc + \dot{\mu_t}(H) = b_t,\\
\partial_t b_t = -\lpar\partial\overline{\partial}\Phi_M\rpar^{-1} \lacc F_t|_{\diag}, b_t \racc.
\end{gather*}
For $\mu_t$ and $b_t$ fixed, Proposition \ref{prop_poiss_quad} solves the first equation in $F_t$ and $\dot{\mu_t}$, and for $F_t$ fixed, the second equation is a transport equation in $b_t$. Thus, we formally solve the equations by iteration on the coefficients of the Taylors expansions of $F_t,\, \mu_t,\, b_t$. Then, to prove that these solutions are well-defined, and are analytic symbols, we look for a $\rho>0$ such that their formal norms $\lnor\cdot\rnor_{\rho}$ are bounded. Here, they depend on the parameter $t$, so we have to add it in the expression of the formal norms

\[
\lnor a \rnor_{\rho,s} = \rho^s \sum\limits_{2k+|\alpha+\beta| = s} \lpar \frac{2(2n)^{-k}k!}{(k+|\alpha|)!(k+|\beta|)!} \rpar \sup\limits_{0\le t\le 1} \lver \partial^{\alpha} \overline{\partial}^{\beta} a_k(t)(0) \rver
\]
but all the results of Section \ref{subsec_symbol_calcul} stay true. We just have one more result, if $f(t)$ is a formal analytic symbol for all $0\le t\le 1$ then for all $\rho>0$ and $s\in\Nm$

\[
\lnor \int_0^t c(\tau) d\tau \rnor_{\rho,s} \le \lnor c \rnor_{\rho,s}.
\]
If $b_t\in S(B_{\rho})$, according to Proposition \ref{prop_poiss_quad} there exists $C>0$ depending on $\Phi_M$ such that

\[
\lnor \dot{\mu_t}(d|z|^2)\rnor_{\rho,s} \le C \lnor \widetilde{b_t}\circ\widetilde{\kappa_q}|_{\diag} \rnor_{\rho,s},
\]
where $\widetilde{\kappa_q}$ is the local symplectomorphism such that $\widetilde{H}\circ\widetilde{\kappa_q}(z,\overline{v}) = dz\overline{v}$. If $(\mu_t,b_t)\in S(B_{\rho})^2$, according to Proposition \ref{prop_poiss_quad},

\begin{align*}
\lnor \widetilde{F_t}\circ\widetilde{\kappa_q}|_{\diag} \rnor_{\rho,s-}
	\le & \frac{C}{1-\lnor(\mu_t-\id)(d|z|^2)\rnor_{\rho,s-}} \lnor \widetilde{b_t}\circ\widetilde{\kappa_q}|_{\diag}\rnor_{\rho,s-}\\
	\le & \frac{C}{1-\lnor \dot{\mu_t}(d|z|^2)\rnor_{\rho,s-}} \lnor \widetilde{b_t}\circ\widetilde{\kappa}|_{\diag}\rnor_{\rho,s-}\\
	\le & \frac{C}{1-C \lnor \widetilde{b_t}\circ\widetilde{\kappa_q}|_{\diag} \rnor_{\rho,s-}} \lnor \widetilde{b_t}\circ\widetilde{\kappa_q}|_{\diag}\rnor_{\rho,s-}
\end{align*}
for $\rho$ small enough. Now, if $F_t \in S(B_{\rho})$, using Lemma \ref{le_ana_bracket},

\[
\lnor \widetilde{b_t}\circ\widetilde{\kappa_q}|_{\diag} \rnor_{\rho,s-} \le C \lnor \widetilde{b_t}\circ\widetilde{\kappa_q}|_{\diag}\rnor_{\rho,(s-2)-} \lnor \widetilde{F_t}\circ\widetilde{\kappa_q}|_{\diag}\rnor_{\rho, (s-2)-}.
\]
as $\lnor b_t \rnor_{\rho,s} = 0 = \lnor F_t \rnor_{\rho,s}$ for $s\le 2$ and all $t\in[0,1]$. Taking $\rho$ small enough, we can suppose that 

\begin{equation}
\label{eq_rec_coef}
\lnor \widetilde{b_t}\circ\widetilde{\kappa_q}|_{\diag} \rnor_{\rho,s-} \le \frac{1}{C(C+1)},\quad \lnor \widetilde{F_t}\circ\widetilde{\kappa_q}_{\diag}\rnor_{\rho,s-} \le \frac{1}{C}
\end{equation}
is satisfied for $s=3$. Then, if \eqref{eq_rec_coef} is satisfied at a given $s\ge 3$, the recursive inequalities imply that \eqref{eq_rec_coef} is also satisfied at $s+1$. According to Lemma \ref{le_ana_sympl}, $\lnor F_t\rnor_{\tau}$, $\lnor b_t\rnor_{\tau}$ and $\lnor \mu_t(H)\rnor_{\tau}$ are bounded for a $\tau\in ]0,\rho]$, in particular $F_t$ is an analytic symbol for all $0\le t\le 1$. Since $F_t(0)=0$ for all $t$, the Hamiltonian flow $\widetilde{\kappa_t}$ is well-defined for all $0\le t\le 1$ on a small enough neighbourhood of $0$, and $\widetilde{\kappa_1}$, $\mu_1$ are the desired solutions.
\end{proof}

If $M=\Cm$, then $\widetilde{\kappa}$ is close to the identity, but if $M$ is a compact manifold, a geometric point of view is necessary. For this case, we use results from \cite{dele25}, where they build FIOs from Bohr-Sommerfeld Lagrangians $\Gamma \subset \widetilde{M}\times\widetilde{\overline{M}}$ close in analytic topology to a $\widetilde{\Gamma_0}$ with $\Gamma_0$ a totally real Lagrangian. Here, the Lagrangian is $\Gamma = \grph(\widetilde{\kappa})$. The symplectomorphism can be written as $\widetilde{\kappa} = \widetilde{\kappa}_1 \circ \widetilde{\eta}$ with $\eta$ a local chart from a neighbourhood of $x_0$ to a neighbourhood of $0$ in $\Cm$, and $\widetilde{\kappa_1} = id_{\widetilde{\Cm}} + O(|z,\overline{v}|^2)$, thus $\Gamma$ is close in analytic topology to $\widetilde{\Gamma_0}$, where $\Gamma_0 = \grph(\eta)$ is totally real. Moreover, a Lagrangian is always locally Bohr-Sommerfeld, which is enough as we will only need local operators.

Notice that we could have combined the Morse Lemma with the result of Section \ref{subsec_quad_symbol} to get directly $\widetilde{f}\circ\widetilde{\kappa}^{-1} = \mu(|z|^2)$, but this would make the quantisation more difficult. This is due to the fact that the symplectomorphism from Section \ref{subsec_quad_symbol} is not close to the extension of any symplectomorphism of $\Cm$.

\begin{proposition}
\label{prop_Fio_mani}
There exists an analytic function $\mu$ and operators $I^{\Gamma} : \barg(B_{\rho})\rightarrow \Hc_N(U)$, $I^{\Gamma_i}:\Hc_N(U)\rightarrow \barg(B_{\rho})$ such that, in local coordinates, and for $\hbar = \frac{1}{N}$,

\begin{align}
\label{eq_Fio_equi}
T_N(f) I^{\Gamma} = & I^{\Gamma} T_{\hbar}\lpar\mu_2(H)+ O(\hbar)\rpar,\\\nonumber
I^{\Gamma} I^{\Gamma_i} = & \id + O(e^{-\frac{c}{\hbar}}).
\end{align}
Furthermore, for any weights $W$, $W_i$ such that $I_d-\partial\overline{\partial}W$ is positive definite, there exists $\epsilon,\epsilon_i = O(|z|^3)$ such that

\[
I^{\Gamma} = O(1) : \statesp[W_i](B_{\rho})\rightarrow \statesp[W_i+\epsilon_i](U),\quad
I^{\Gamma_i} = O(1) : \statesp[W+\epsilon](U) \rightarrow \statesp[W](B_{\rho}).
\]
\end{proposition}

\begin{proof}
We take $x_0= 0$ to simplify the notations. Thanks to what we saw earlier, we can use \cite{dele25} Proposition 4.3 , thus there exists bounded operators $I^{\Gamma} : \barg(B_{\rho})\rightarrow \Hc_N(U)$ and $I^{\Gamma_i}:\Hc_N(U)\rightarrow \barg(B_{\rho})$ such that 

\begin{align}
T_N(f) I^{\Gamma} = & I^{\Gamma} T_{\hbar}\lpar f\circ\eta^{-1}+ O(\hbar)\rpar,\\\nonumber
I^{\Gamma} I^{\Gamma_i} = & \id + O(e^{-\frac{c}{\hbar}}).
\end{align}
Since they are defined locally, the operators are also bounded on $\statesp[W]\rightarrow \statesp[W]$ for any weight $W$. Then, $\widetilde{\kappa_1}(x,\overline{y}) = (x,\overline{y}) + O(|x,y|^2)$, and the definition \eqref{eq_phase_symp} of a phase associated to $\widetilde{\kappa_1}$ becomes

\[
(\partial_{\overline{y}}\phi(x,\overline{y}),\overline{y}) = (x,\partial_x\phi(x,\overline{y})) + O(|x,\partial_x\phi(x,\overline{y})|^2),
\]
Hence, there exists a function $\phi$ associated to $\widetilde{\kappa_1}$ on a small enough neighbourhood of $(0,0)$, and $\phi(x,\overline{y}) = x\overline{y} + O(|x,\overline{y}|^3)$. In particular, for a any symbol $g$

\begin{align}
T_{\hbar}(g) I_{\phi} = & I_{\phi} T_{\hbar}\lpar \widetilde{g}\circ\widetilde{\kappa_1}^{-1}+ O(\hbar)\rpar,\\\nonumber
I_{\phi} I_{\phi_i} = & \id + O(e^{-\frac{c}{\hbar}}).
\end{align}
and for any weight $W$, the transformed weight by $\phi$ is of the form $W + O(|z|^3)$. In the end, $I_{\Gamma} I_{\phi}$ and $I_{\phi_i} I_{\Gamma_i}$ satisfy \eqref{eq_Fio_equi}. The operators' norms are exponentially big in general, but since $\widetilde{\kappa} = \widetilde{\kappa_1}\circ\widetilde{\eta}$, they will be of size $O(e^{\frac{c\rho^3}{\hbar}})$, thus they can be compensated by weights that grow like $|z|^3$.
\end{proof}

\begin{remark}
The constructions from \cite{dele25} are similar to our FIOs from Section \ref{sec_toolbox}, except that they consider a different manifold in the domain and range of the operator. For instance, if $\Gamma$ is a Lagrangian of $\widetilde{M}\times\widetilde{\overline{M}}$, then $I^{\Gamma} = I_{\Phi_M,\phi}(a)$ for the right $\phi$ and $a$. We decided to use the results from \cite{dele25} for this specific step, and consider FIOs that does not change the ambient manifold in the rest of our study for simplicity.
\end{remark}

\subsection{Condition on the Hessian}
\label{subsec_simplification}

The problem is now reduced to a Toeplitz operator on $B_{\rho} \subset \Cm$ with symbol $f\in S(B_{\rho})$ of the form $f = \mu(H) + \hbar g + O(\hbar^2)$. The next step is to gain one more order in $\hbar$, meaning to simplify the symbol into $\mu_2(H) + O(\hbar^2)$. However, the spectrum of $T_{\hbar}(\mu(H))$ is difficult to write in terms of $\mu$ and $H$, meanwhile the spectrum of $\mu(T_{\hbar}(H))$ is simply $\mu(\sigma(H))$. Thus, we search for $a\in S(B_{\rho})$ of order $0$ such that

\[
T_{\hbar}\lpar \mu(H)+\hbar g + O(\hbar^2) \rpar T_{\hbar}(a) = T_{\hbar}(a) (\mu+\hbar r)\lpar T_{\hbar}(H)\rpar + O_{\barg(B_{\rho})\circlearrowleft}(\hbar^2).
\]
Notice that for all analytic symbol $H$ and analytic function $F$,

\[
F(T_{\hbar}(H)) = T_{\hbar}\lpar F(H)-\frac{\hbar}{2}\partial H\overline{\partial}HF''(H)\rpar + O_{\barg\circlearrowleft}(\hbar^2).
\]
hence, the symbol $a$ must satisfy

\begin{align*}
	& \lpar \mu(H) + \hbar g + O(\hbar^2) \rpar \# a\\
	= & a \# \lpar \mu(H) + \hbar \lpar r(H)- \partial H\overline{\partial}H\mu''(H)\rpar + O(\hbar^2) \rpar.
\end{align*}
This equality is always satisfied at order $0$, we look at the order $1$ term

\[
\partial \lpar\mu(H)\rpar \overline{\partial}a + ag = \overline{\partial} \lpar\mu(H)\rpar \partial a + \lpar r(H)-\partial H\overline{\partial}H\mu''(H) \rpar a
\]
which can be written

\[
\{\mu(H),a\} = i\lpar g+\partial H\overline{\partial}H\mu''(H)-r(H)\rpar a.
\]
According to Proposition \ref{prop_poiss_quad}, there exists $\rho'\in ]0,\rho[$, $r\in S(B_{\rho'})$ and $b\in S(B_{\rho'})$ such that

\[
\{\mu(H),b\} = g+\partial H\overline{\partial}H\mu''(H)-r(H).
\]
Actually, since $g$ does not depend on $\hbar$, $r$ and $b$ are simply analytic functions on $B_{\rho'}$, and $a=e^{ib}$ is a well-defined analytic function on $B_{\rho'}$ which satisfies the equation. Notice that this computation only uses that $H$ is non-degenerate, but not the ellipticity yet.

The studied operator is then of the form $\mu(T_{\hbar}(H)) + \hbar^2 T_{\hbar}(g)$, and we suppose from now on that $H$ satisfies \eqref{eq_complex_quad}. We are looking for a symbol $a$ such that

\begin{equation}
\label{eq_Moser_base}
\lpar \mu(T_{\hbar}(H)) + \hbar^2 T_{\hbar}(g) \rpar T_{\hbar}(1+a) = T_{\hbar}(1+a) (\mu+\hbar^2 r)(T_{\hbar}(H)) + O_{\barg(B_{\rho})\circlearrowleft}(e^{-\frac{c}{\hbar}}).
\end{equation}
For the computations, we use the operator $I_{\phi} = I_{\phi}(1) = O(1) : \barg \rightarrow \statesp[W_l]$ of Theorem \ref{th_FIO_normal}, which satisfies
\[
I_{\phi} T_{\hbar}(H) = T_{\hbar}(|z|^2 + C\hbar) I_{\phi} + O_{\barg(B_{\rho})\rightarrow\statesp[W_l](B_{\rho})}(e^{-\frac{c}{\hbar}})
\]
for any $\rho>0$. Therefore, by definition of $\mu(T_{\hbar}(H))$, equation \eqref{eq_Moser_base} conjugated by $I_{\phi}$ gives

\begin{multline*}
\lpar \mu(T_{\hbar}(|z|^2+C\hbar)) + \hbar^2 T_{\hbar}(g_2) \rpar T_{\hbar}(1+a_c)\\
	= T_{\hbar}(1+a_c) (\mu+\hbar^2 r)(T_{\hbar}(|z|^2+C\hbar)) + O_{\statesp[W_l](B_{\rho})\circlearrowleft}(e^{-\frac{c}{\hbar}})
\end{multline*}
where $\rho'\in ]0,\rho[$, and $1+a_c \in S(B_{\rho'})$ is defined by $T_{\hbar}(1+a_c) = I_{\phi} T_{\hbar}(1+a) I_{\phi_i}$. Then, according to Lemma \ref{le_func_operator}, this equality is equivalent to

\begin{equation}
\label{eq_Moser_normal}
T_{\hbar}(\mu_b(|z|^2) + \hbar^2 g_2) T_{\hbar}(1+a_c) = T_{\hbar}(1+a_c) T_{\hbar}(\mu_b(|z|^2)+\hbar^2 r_b(|z|^2)) + O_{\statesp[W_l](B_{\sigma})\circlearrowleft}(e^{-\frac{c}{\hbar}}).
\end{equation}
Thanks to this simplification, finding $a$ and $r$ satisfying \eqref{eq_Moser_base} is equivalent to finding $a_c$ and $r_b$ satisfying \eqref{eq_Moser_normal}. Notice that for two symbols $a$,$b$ such that $a\sim b$, by definition $T_{\hbar}(a) = T_{\hbar}(b) + O_{\barg(B_{\rho})\circlearrowleft}(e^{-\frac{c}{\hbar}})$, but the $O_{\barg(B_{\rho})\circlearrowleft}(e^{-\frac{c}{\hbar}})$ can be replaced by $O_{\statesp[W_l](B_{\rho})\circlearrowleft}(e^{-\frac{c}{\hbar}})$ as $(x,\overline{y}) \mapsto x\overline{y}$ is a phase for any weight $W_l$ with transformed $W_l$. Hence, we can solve \eqref{eq_Moser_normal} with symbolic calculus.

\subsection{Moser's trick}

In order to solve \eqref{eq_Moser_normal}, we use a Moser's trick like in Section \ref{subsec_morse}. For $\mu$, $g$ analytic symbols on a neighbourhood of $0$, we are looking for time dependant symbols $a(t)$, $r(t)$ such that
\begin{equation}
\label{eq_symb_calc}
\lpar \mu(|z|^2) +t\hbar^2 g \rpar\# (1+a(t)) = (1+a(t) ) \#\lpar \mu(|z|^2)+\hbar^2 r(t)(|z|^2) \rpar.
\end{equation}

\begin{remark}
Notice that this method cannot be applied if the remainder is of order $\hbar$, as the symbolic calculus would impose to fix $r_0$ such that $r_0(|z|^2) = g_0$, which is not possible in general.

By solving the equation order by order in $\hbar$, we notice that we get terms with higher orders in $t$, thus the necessity to balance it with $r(t)(|z|^2)$ and not just a linear term in $t$. We will write $r(t)$ instead of $r(t)(|z|^2)$ for simplicity.
\end{remark}

We search $a(t)$ as an analytic symbol solution to
\begin{equation}
\label{eq_time_symb}
\left\{
\begin{array}{l}
\partial_t a(t) = \hbar ib(t) \# (1+a(t)),\\
a(0) = 0\\
\end{array}
\right.
\end{equation}
and $r$ will be fixed by the equation and $r(0) = 0$. If equation \eqref{eq_symb_calc} is satisfied, differentiating it with respect to $t$ gives
\begin{align*}
\hbar^2 g & = \lbra \hbar ib, \mu(|z|^2)+t\hbar^2 g \rbra_{\#} + \hbar^2(1+a(t))\# \dot{r}(t) \#(1+a(t))^{\#-1} \\
\Leftrightarrow \hbar^{-1} \lbra \mu(|z|^2),ib \rbra_{\#} & = -g + \hbar it \lbra b,g\rbra_{\#} + R(t)
\end{align*}
with

\[
R(t) = (1+a(t))\# \dot{r}(t)(|z|^2) \#(1+a(t))^{\#-1} = \dot{r} + a\#\dot{r} + \dot{r}\# a^* + a\#\dot{r}\# a^*.
\]
We notice that the coefficients of order $k$ of the symbols

\[
\hbar^{-1}[\mu(|z|^2),ib]_{\#} - \lacc b,\mu(|z|^2) \racc,\quad \hbar it \lbra g,b\rbra_{\#},\quad \dot{r} - R
\]
only depends on $b_j$ and $r_j$ for $j<k$. Thus, the equation can be written like

\begin{equation}
\label{eq_normal_form}
\mu'(|z|^2)\partial_{\theta} b
	= g + \hbar it \lbra g,b\rbra_{\#} +(\dot{r} - R) + \lpar \hbar^{-1}[\mu(|z|^2),ib]_{\#} - \lacc b,\mu(|z|^2) \racc \rpar - \dot{r}
\end{equation}
and be solved by Proposition \ref{prop_poiss_quad}. Then, the functions are formally defined by the iterative equations

\begin{multline*}
	\dot{r_k}(|z|^2) = \int_{|z|\Sm} g_k + it \lbra g,b\rbra_{\#,k-1} +(\dot{r} - R)_k + \lpar \hbar^{-1}[\mu(|z|^2),ib]_{\#} - \lacc b,\mu(|z|^2) \racc \rpar_k \frac{d\theta}{2\pi},\\
	b_k = \frac{1}{\mu'(|z|^2)} \int g_k + it \lbra g,b\rbra_{\#,k-1} +(\dot{r} - R)_k + \lpar \hbar^{-1}[\mu(|z|^2),ib]_{\#} - \lacc b,\mu(|z|^2) \racc \rpar_k - \dot{r}_k \frac{d\theta}{2\pi}.
\end{multline*}
and the formal norms are bounded for all $s$ by

\begin{align*}
	\lnor \dot{r} \rnor_{\rho,s-} \le & \lnor g \rnor_{\rho,s-} + \rho^2 \lpar t \lnor g \rnor_{\rho,(s-1)-} + \lnor \mu'(|z|^2) \rnor_{\rho,(s-3)-} \rpar \lnor b \rnor_{\rho,(s-1)-}\\
	& + \lpar \lnor a \rnor_{\rho,(s-1)-} + \lnor a^* \rnor_{\rho,(s-1)-} + \lnor a \rnor_{\rho,(s-1)-} \lnor a^* \rnor_{\rho,(s-1)-} \rpar \lnor \dot{r} \rnor_{\rho,(s-1)-},
\end{align*}
and

\begin{align*}
	\lnor b \rnor_{\rho,s-} \le & \frac{\lnor g \rnor_{\rho,s-} + \lnor R \rnor_{\rho,s-} + \rho^2 \lpar t \lnor g \rnor_{\rho,(s-1)-} + \lnor \mu'(|z|^2) \rnor_{\rho,(s-3)-} \rpar \lnor b \rnor_{\rho,(s-1)-}}{1-\lnor(\mu-\id)(|z|^2)\rnor_{\rho,(s+2)-}}.
\end{align*}
To summarise, the coefficients are given by the equations

\begin{align}
\label{eq_formal}\nonumber
b_k & = \frac{1}{\mu'(|z|^2)} \int g_k + it \lbra g,b\rbra_{\#,k-1} - R_k + \lpar \hbar^{-1}[\mu(|z|^2),ib]_{\#} - \lacc b,\mu(|z|^2) \racc \rpar_k \frac{d\theta}{2\pi},\\ \nonumber
R_k & = \dot{r}_k + \lpar a\#\dot{r} + \dot{r}\# a^* + a\#\dot{r}\# a^* \rpar_k,\\
\dot{r_k} & = \int_0^{2\pi} g_k + it \lbra g,b\rbra_{\#,k-1} +(\dot{r} - R)_k + \lpar \hbar^{-1}[\mu(|z|^2),ib]_{\#} - \lacc b,\mu(|z|^2) \racc \rpar_k \frac{d\theta}{2\pi},\\ \nonumber
a_k(t) & = i \int_0^t \lpar b\# (1+a(s)) \rpar_{k-1} ds,
\end{align}
As in the beginning of the Section, we consider formal norms depending on the parameter $t$, and we use the results of \ref{subsec_symbol_calcul} to bound them. Since $a$ is of order $1$, if $\rho$ is small enough then $\lnor a \rnor_{\rho} \le 1$, and

\begin{align*}
\lnor R \rnor_{\rho,s-}
	& \le \lpar 1+2\lnor a \rnor_{\rho,s-}+2\lnor a \rnor_{\rho,s-}^2 \rpar \lnor \dot{r} \rnor_{\rho,s-} \le \lpar 1+4\lnor a \rnor_{\rho,s-} \rpar \lnor \dot{r} \rnor_{\rho,s-}\\
\lnor \dot{r} \rnor_{\rho,s-}
	& \le C \lpar 1+2\rho^2 \lnor b \rnor_{\rho,(s-1)-} \rpar + 3 \lnor a \rnor_{\rho,s-} \lnor \dot{r} \rnor_{\rho,(s-1)-}\\
\lnor a \rnor_{\rho,s-}
	& \le 2 \rho^2 \lnor b \rnor_{\rho,(s-1)-} \lnor 1+a \rnor_{\rho,(s-1)-}
	\le 2 \rho^2 \lnor b \rnor_{\rho,(s-1)-} \lpar 1+\lnor a \rnor_{\rho,(s-1)-}\rpar\\
\lnor b \rnor_{\rho,s-}
	& \le C \lpar 1 + \lnor R \rnor_{\rho,s-} + 2\rho^2 \lnor b \rnor_{\rho,(s-1)-} \rpar
\end{align*}
with $C = \max\lpar 1,\frac{1}{1-\lnor\mu-\id\rnor_{\rho}} \rpar \max\lpar \lnor g \rnor_{\rho}, \lnor \mu'(|z|^2) \rnor_{\rho} \rpar$. Let us prove that for $\rho$ small enough, these formal norms are bounded. We prove by induction on $s\in\Nm$ that

\[
\lnor a \rnor_{\rho,s-} \le \frac{1}{8},\quad \lnor R \rnor_{\rho,s-} \le \frac{8}{3}C, \quad \lnor \dot{r} \rnor_{\rho,s-} \le \frac{16}{9} C, \quad \lnor b \rnor_{\rho,s-} \le 4C(1+C).
\]
For $s=0$ we have 

\[
\lnor a \rnor_{\rho,0} = 0, \quad \lnor R \rnor_{\rho,0} = \lnor \dot{r} \rnor_{\rho,0},\quad \lnor \dot{r} \rnor_{\rho,0} = \lnor g \rnor_{\rho,0},\quad \lnor b \rnor_{\rho,0} = 0
\]
so the inequalities hold by hypothesis on $C$. We now consider $s\ge0$, and we suppose them to hold for all $s'<s$. We get that

\[
\lnor a \rnor_{\rho,s-} \le 2 \rho^2 \lnor b \rnor_{\rho,(s-1)-} \lpar 1+\lnor a \rnor_{\rho,(s-1)-}\rpar \le 8C(1+C)\rho^2 \frac{9}{8}
\]
so by taking $\rho^2 \le \frac{1}{9}\frac{1}{8C(1+C)}$ we get $\lnor a \rnor_{\rho} \le \frac{1}{8}$. Then we just have to check the remainders

\begin{tabular}{l r}
~\\
$\begin{aligned}[c]
\lnor b \rnor_{\rho,s-}
	& \le C \lpar 1 + \lnor R \rnor_{\rho,s-} + \rho^2 \lnor b \rnor_{\rho,(s-2)-} \rpar\\
	& \le C \lpar 1+\frac{8}{3}C +\rho^2 4C(1+C) \rpar\\
	& \le C \lpar 1+\frac{8}{3}C +\frac{1}{18} \rpar\\
	& \le 4C(1+C)
\end{aligned}$
&
$\begin{aligned}
\lnor R \rnor_{\rho,s-}
	& \le \lpar 1+4\lnor a \rnor_{\rho,s-} \rpar \lnor \dot{r} \rnor_{\rho,s-}\\
	& \le \lpar 1+\frac{1}{2} \rpar \frac{16}{9}C\\
	& \le \frac{8}{3}C
\end{aligned}$
\end{tabular}

\begin{align*}
\lnor \dot{r} \rnor_{\rho,s-}
	& \le C \lpar 1+2\rho^2 \lnor b \rnor_{\rho,(s-1)-} \rpar + 3 \lnor a \rnor_{\rho,s-} \lnor \dot{r} \rnor_{\rho,(s-1)-}\\
	& \le C \lpar 1+\frac{1}{9} \rpar + \frac{3}{8} \frac{16}{9}C\\
	& \le \frac{16}{9}C
\end{align*}
Notice that $\lnor a \rnor_{\rho}$ can be arbitrarily small by taking $\rho$ close enough to $0$. Using the same notations as Section \ref{subsec_simplification}, we thus proved that there exists $\rho>0$ and $a_c,r_b(|z|^2)\in S(B(0,\rho))$ that satisfies \eqref{eq_symb_calc}, therefore they also satisfy \eqref{eq_Moser_normal}. Finally, the symbols satisfy
\[
\lpar \mu(T_{\hbar}(H)) + \hbar^2 T_{\hbar}(g) \rpar T_{\hbar}(1+a) = T_{\hbar}(1+a) (\mu+\hbar^2 r)(T_{\hbar}(H)) + O_{\statesp[W_l](B_{\sigma})\circlearrowleft}(e^{-\frac{c}{\hbar}}).
\]

\subsection{Action integral}

This subsection is devoted to the spectrum of the normal form $\mu(T_{\hbar}(|z|^2))$. The eigenvalues are $\mu(\hbar\Nm)$, and if the norm of its resolvent at $\lambda\in\Cm$ is bigger than $e^{\frac{c}{\hbar}}$, then $\lambda$ is exponentially close to an eigenvalue. Moreover, the eigenvalues of $T_N(f)$ will be given by a quantisation condition similar to Bohr-Sommerfeld, but in the complexified space. In particular, the function $\mu_0$ is the inverse of the action integral of $f_0$, the principal term of the studied symbol. Recall that in the self-adjoint case, the action integral is defined by $A(E) = \int_{\gamma_E}\xi dx$ where $\gamma_E = \lacc (x,\xi) /\; f_0(x,\xi)=E\racc$. In the non self-adjoint case, more steps are necessary as the functions are defined on the complexified spaces. Recall that there exists a symplectormophism $\widetilde{\kappa}$ from a neighbourhood of $(x_0,\overline{x_0})$ in $\widetilde{M}$ to a neighbourhood of $0$ in $\widetilde{\Cm}$ such that $\widetilde{f_0}\circ\widetilde{\kappa}^{-1} = \mu_0\lpar\widetilde{H}\rpar$, where $H$ is the Hessian of $f$ at $x_0$, $\mu_0$ is holomorphic and $\mu_0(0)=f_0(x_0),\, \mu_0'(0)=1$. We follow the strategy of \cite{hitr24}.

\begin{lemma}
\label{le_spectrum_normal}
Let $\mu$ be an analytic symbol on $\Cm$, then the spectrum of $\mu(T_{\hbar}(|z|^2))$ is $\mu(\hbar\Nm^*)$, and the eigenvalues are simple with eigenfunctions $\lpar e^{-\frac{|z|^2}{2\hbar}} z^k\rpar_{k\in\Nm}$. Moreover, for $c>0$ there exists $c'>0$ such that, for all $\lambda\in\Cm$ and $u\in\barg$, if $\lpar \mu(T_{\hbar}(|z|^2))-\lambda\rpar u = O_{\barg}(e^{-\frac{c}{\hbar}})$ and $\lnor u\rnor_{\barg} \geq C>0$, then there exists $j\in\Nm^*$ such that $\lambda = \mu(j\hbar) + O(e^{-\frac{c'}{\hbar}})$.
\end{lemma}

\begin{proof}
By definition, there exists $C,A>0$, and analytic functions $\mu_j$ such that $|\mu_j|\le CA^jj!$ on a neighbourhood of $0$, and $\mu = \sum\limits_{j=0}^{\frac{1}{A\hbar}}\hbar^j\mu_j$. Let $u\in\barg$, it can be written $u(z) = e^{-\frac{|z|^2}{2\hbar}} \sum\limits_{k=0}^{\infty} \frac{u^{(k)}(0)}{k!} z^k$ near $0$, and

\begin{align*}
\mu(T_{\hbar}(|z|^2))u
	= & \sum\limits_{j=0}^{\frac{1}{A\hbar}} \sum\limits_{k=0}^{\infty} \sum\limits_{l=0}^{\infty} \hbar^j \frac{\mu_j^{(k)}(0)}{k!} \frac{u^{(l)}(0)}{l!} T_{\hbar}(|z|^2)^k \lpar e^{-\frac{|z|^2}{2\hbar}} z^l\rpar\\
	= & \sum\limits_{j=0}^{\frac{1}{A\hbar}} \sum\limits_{k=0}^{\infty} \sum\limits_{l=0}^{\infty} \hbar^j \frac{\mu_j^{(k)}(0)}{k!} \frac{u^{(l)}(0)}{l!} (\hbar(l+1))^k \lpar e^{-\frac{|z|^2}{2\hbar}} z^l\rpar\\
	= & \sum\limits_{j=0}^{\frac{1}{A\hbar}} \sum\limits_{l=0}^{\infty} \hbar^j \mu_j(\hbar(l+1)) \frac{u^{(l)}(0)}{l!} \lpar e^{-\frac{|z|^2}{2\hbar}} z^l\rpar\\
	= & \sum\limits_{l=0}^{\infty} \mu(\hbar(l+1)) \frac{u^{(l)}(0)}{l!} \lpar e^{-\frac{|z|^2}{2\hbar}} z^l\rpar.
\end{align*}
Since $\lpar e^{-\frac{|z|^2}{2\hbar}} z^l\rpar$ is a Hilbert basis of $\barg$, $\mu(T_{\hbar}(|z|^2))u = \lambda u$ is equivalent to $\lambda = \mu((l+1)\hbar)$ for a $l\in\Nm$ and $u = e^{-\frac{|z|^2}{2\hbar}} z^l$. Now, the equation $\lpar \mu(T_{\hbar}(|z|^2))-\lambda\rpar u = O_{\barg}(e^{-\frac{c}{\hbar}})$ is equivalent to

\begin{align*}
	& \sum\limits_{l=0}^{\infty} \lpar\mu(\hbar(l+1))-\lambda\rpar \frac{u^{(l)}(0)}{l!} \lpar e^{-\frac{|z|^2}{2\hbar}} z^l\rpar = O_{\barg}(e^{-\frac{c}{\hbar}})\\
	\Leftrightarrow & \sum\limits_{l=0}^{\infty} \lver \mu(\hbar(l+1))-\lambda\rver^2 \frac{|u^{(l)}(0)|^2}{l!} = O(e^{-\frac{c}{\hbar}})
\end{align*}
and since $\sum\limits_{l=0}^{\infty} \frac{|u^{(l)}(0)|^2}{l!} = \lnor u\rnor_{\barg}^2 \geq C^2>0$, there exists $l\in\Nm$ such that $\lver \mu(\hbar(l+1))-\lambda\rver = O(e^{-\frac{c'}{\hbar}})$ for $c'<c$.
\end{proof}

The definition of the action integral is more convenient with a diffeomorphism that reduces $\widetilde{f_0}$ to $\widetilde{H} = \widetilde{\hess_{x_0}f_0} = \hess_{x_0,\overline{x_0}}\widetilde{f_0}$, and multiplies the symplectic form by a function of $\widetilde{H}$.

\begin{lemma}
There exists a local diffeomorphism $\widetilde{\eta}$ from a neighbourhood of $(x_0,\overline{x_0})$ in $\widetilde{M}$ to a neighbourhood of $(0,0)$ in $\widetilde{\Cm}$ such that $\widetilde{f_0}\circ\widetilde{\eta}^{-1} (z,\overline{v}) = \widetilde{H}(z,\overline{v})$ and $(\widetilde{\eta}^{-1})^* \widetilde{\omega} = G\lpar\widetilde{H}(z,\overline{v})\rpar idz\wedge d\overline{v}$ where $G$ is holomorphic near $0$. Furthermore, $\mu'_0(\xi)=\frac{1}{G\circ\mu_0(\xi)}$ near $0$.
\end{lemma}

\begin{proof}
We look for a local diffeomorphism $\widetilde{\delta}$ on a neighbourhood of $0$ in $\widetilde{\Cm}$ such that $\widetilde{H}\circ \widetilde{\delta} = \mu_0\lpar\widetilde{H}\rpar$ and $\widetilde{\delta}^*\lpar G\lpar\widetilde{H}\rpar idz\wedge d\overline{z}\rpar= idz\wedge d\overline{v}$. Let

\[
\widetilde{\delta}(x,\overline{y}) = \lpar\sqrt{a\lpar\widetilde{H}(x,\overline{y})\rpar}x,\sqrt{a\lpar\widetilde{H}(x,\overline{y})\rpar}\overline{y}\rpar,
\]
then $\widetilde{\delta}^* \widetilde{H} = a\lpar\widetilde{H}\rpar\widetilde{H}$ and

\[
\widetilde{\delta}^*\lpar G\lpar\widetilde{H}\rpar idz\wedge d\overline{v}\rpar = G\lpar a\lpar\widetilde{H}\rpar\widetilde{H}\rpar\lpar a\lpar\widetilde{H}\rpar+\widetilde{H}a'\lpar\widetilde{H}\rpar\rpar idz\wedge d\overline{v}.
\]
Since $\mu_0$ is holomorphic with $\mu_0(0)=0$ and $\mu_0'(0)=1$, we can take $a(\xi) = \frac{\mu_0(\xi)}{\xi}$ and there exists $G$ holomorphic such that $G(\mu_0(\xi))\mu'_0(\xi) = 1$. The solution is thus $\widetilde{\eta} = \widetilde{\delta}\circ\widetilde{\kappa}$.
\end{proof}

According to Section \ref{subsec_quad_symbol}, there exists $\alpha\in\Cm$ such that $\Re(\alpha H)$ is positive definite, and there exists a local symplectomorphism $\widetilde{\kappa_q}$ such that $H\circ\widetilde{\kappa_q}^{-1} = d_0x\overline{y}$, with $d_0 = \frac{\sqrt{\det(\alpha H)}}{\alpha}$. Then, for $E\neq f(x_0)$ close to $f(x_0)$, we define

\[
\gamma_E = \lpar\widetilde{\kappa_q}\circ\widetilde{\eta}\rpar^{-1}\lacc x=\sqrt{\frac{E}{d_0}} e^{it},\; \overline{y}=\sqrt{\frac{E}{d_0}} e^{-it} \middle/\; t\in [0,2\pi] \racc = \lpar\widetilde{\kappa_q}\circ\widetilde{\eta}\rpar^{-1}\lpar\sqrt{\frac{E}{d_0}}\Sm\rpar,
\]
and
\begin{equation}
\label{eq_action_int}
A(E) = \int_{\gamma_E} \alpha,
\end{equation}
where $\alpha$ is a primitive of the symplectic form $\widetilde{\omega}$ on $\widetilde{M}$.

\begin{proposition}
\label{prop_action_int}
The action integral defined by \eqref{eq_action_int} satisfies $A\circ\mu_0(z) = \frac{2\pi}{d_0} z$ for $z$ in a neighbourhood of $x_0$, distinct from $x_0$.
\end{proposition}

\begin{proof}
Using the change of variable formula and Stokes' theorem we get for $t\neq 0$ close to $0$
\begin{align*}
A(td_0)
	& = \int_{t\Sm} \lpar\lpar\widetilde{\kappa_q}\circ\widetilde{\eta}\rpar^{-1}\rpar^* (\alpha)\\
	& = \int_{t\mathbb{B}} \lpar\lpar\widetilde{\kappa_q}\circ\widetilde{\eta}\rpar^{-1}\rpar^* (\widetilde{\omega})\\
	& = \int_{t\mathbb{B}} G(d_0x\overline{y}) idx\wedge d\overline{y}\\
	& = 2\pi \int_0^t G(d_0r)dr
\end{align*}
hence $d_0A'(td_0) = 2\pi G(td_0)$. Then, according to Lemma \ref{le_holo_ext}, $A$ is holomorphic near $0$ with $d_0 A'(z) = 2\pi G(z)$ for all $z\in\Cm$ near $0$. Hence,
\[
d_0 (A\circ\mu_0)'(z) = d_0 \mu'_0(z) A'\circ \mu_0(z) = 2\pi \lpar \mu'_0 G\circ\mu_0\rpar(z) = 2\pi
\]
thus $A\circ\mu_0(\xi) = \frac{2\pi}{d_0}\xi$.
\end{proof}

Actually, the definition of $A$ does not depend on the choice of contour in $\lacc f_0=E\racc$, as it can be deformed with the results from Section \ref{sec_toolbox}. For instance, if $f(x,\overline{y}) = dx\overline{y}$, any closed contour $\gamma(t)$ in $\Cm$ gives a possible contour $(x,\overline{y})= \lpar\gamma(t),\frac{E}{d\gamma(t)}\rpar$. In which case

\begin{align*}
A(E)
	= & -i \int_0^{2\pi} \frac{E\gamma'(t)}{d\gamma(t)} dt\\
	= & \frac{E}{di} \int_{\gamma} \frac{1}{z} dz\\
	= & \frac{2\pi E}{d} \wind(0,\gamma)
\end{align*}
using the residue theorem. Two contours $(x_1(t),\overline{y_1}(t))$, $(x_2(t),\overline{y_2}(t))$ gives the same result if and only if they are in the same homotopy class in $\lacc (x,\overline{y})\in\Cm^2 /\; dx\overline{y}=E\racc$. Indeed, these two loops are homotopic if and only if $x_1$ and $x_2$ are homotopic in $\Cm\backslash\{0\}$, which is equivalent to $\wind(0,x_1) = \wind(0,x_2)$.

\begin{proposition}
Let $E\neq f(x_0)$ be close to $f(x_0)$, and $\gamma_{E,1}$ be a closed contour in $\widetilde{f}^{-1}(E)$, homotopic to $\gamma_E$, then

\[
A(E) = -i\int_{\gamma_{E,1}}\alpha
\]
\end{proposition}

\begin{proof}
Let $\rho_{E,0} = \widetilde{\kappa_q}\circ\widetilde{\eta}(\gamma_E)$, and $\rho_{E,1} = \widetilde{\kappa_q}\circ\widetilde{\eta}(\gamma_{E,1})$, then $\rho_{E,0},\rho_{E,1}$ are homotopic in $\lacc (x,\overline{y})\in\Cm^2 /\; dx\overline{y}=E \racc$. Therefore, we can choose a homotopy $\gamma_{E,t}$ such that $dx\overline{y} = E$ for all $t\in [0,1]$, and for all $(x,\overline{y}) \in \rho_{E,t}$. Now, applying the change of variable $\widetilde{\kappa_q}\circ\widetilde{\eta}$ and Stokes' theorem in two ways

\begin{align*}
	  \int_{\widetilde{\kappa_q}\circ\widetilde{\eta}(\rho_{E,1})} \alpha - \int_{\widetilde{\kappa_q}\circ\widetilde{\eta}(\rho_{E,0})} \alpha
	= & \int_{\rho_{E,1}} \lpar\lpar\widetilde{\kappa_q}\circ\widetilde{\eta}\rpar^{-1}\rpar^*(\alpha) - \int_{\rho_{E,0}} \lpar\lpar\widetilde{\kappa_q}\circ\widetilde{\eta}\rpar^{-1}\rpar^*(\alpha)\\
	= & \int_{\bigcup_t\gamma_{E,t}} \lpar\lpar\widetilde{\kappa_q}\circ\widetilde{\eta}\rpar^{-1}\rpar^*(\omega)\\
	= & \int_{\bigcup_t\gamma_{E,t}} G\lpar dx\overline{y}\rpar d\overline{y}\wedge dx\\
	= & G(E) \int_{\bigcup_t\gamma_{E,t}} d\overline{y}\wedge dx\\
	= & G(E)\lpar \int_{\rho_{E,1}} \overline{y}dx - \int_{\rho_{E,0}} \overline{y}dx\rpar = 0.\\
\end{align*}
\end{proof}

In the next subsection, we will get an asymptotic expression of the eigenvalues of $T_N(f)$ similar to the Bohr-Sommerfeld quantisation conditions, and we will use the previous results. We already know $\mu$ at order $0$ in $\hbar$, but we will describe how to compute it to order $1$ using a WKB method.

\subsection{Spectral description and resolvent estimate}

Let us summarise what we have done so far. Recall that we consider a symbol $f$ on $M$ such that $f'(x_0)=0$ and the Hessian $H$ of $f$ at $x_0$ satisfies \eqref{eq_complex_quad}. We reduced the symbol with four conjugations by operators:

\begin{itemize}
\item $I^{\Gamma}$ from Proposition \ref{prop_Fio_mani} gives $\mu_0(H) + O(\hbar)$,
\item $T_{\hbar}(a_1)$ gives $\mu_2(H)+O(\hbar^2)$,
\item $T_{\hbar}(1+a_2)$ gives $\mu(H) + O\lpar e^{-\frac{c}{\hbar}}\rpar$,
\item $I_{\phi}: \statesp[W](B_{\rho}) \rightarrow \barg(B_{\rho})$ from Theorem \ref{th_FIO_normal}, with inverse $I_{\phi_i}: \barg(B_{\rho}) \rightarrow \statesp[W_i](B_{\rho})$, conjugates $H$ and $|z|^2$.
\end{itemize}
In the end, according to Proposition \ref{prop_Fio_mani}, there exists weights $\mathbf{w},\mathbf{w}_i = O(|z|^3)$ such that $\Fio = I^{\Gamma} T_{\hbar}(a_1) T_{\hbar}(1+a_2) I_{\phi_i}$ is bounded from $\barg(B_{\rho})$ to $\statesp[W_i+\mathbf{w}_i](U)$, it has a local inverse $\Fio_i$ bounded from $\statesp[W+\mathbf{w}](U)$ to $\barg(B_{\sigma})$, and they satisfy

\begin{align*}
T_N(f)\Fio = & \Fio\mu\lpar T_{\hbar}\lpar\frac{\sqrt{\det (\alpha H)}}{\alpha} \lpar |z|^2-\frac{\hbar}{2} \rpar + \frac{\hbar\partial\overline{\partial}H}{2} \rpar\rpar + O_{\barg(B_{\rho})\rightarrow\statesp[W_i+\epsilon_i](U)}\lpar e^{-\frac{c}{\hbar}}\rpar\\
	= & \Fio \mu\lpar T_{\hbar}(H_0) \rpar + O_{\barg(B_{\rho})\rightarrow\statesp[W_i+\epsilon_i](U)}\lpar e^{-\frac{c}{\hbar}}\rpar,\\
\Fio\Fio_i = & id + O_{\statesp[W+\epsilon](U)\rightarrow\statesp[W_i+\epsilon_i](U)}(e^{-\frac{c}{\hbar}}).
\end{align*}

We already saw that the eigenfunctions of $\mu(T_{\hbar}(|z|^2))$ are $\lpar e^{-\frac{|z|^2}{2\hbar}} z^k\rpar_{k\in\Nm^*}$ which are localised on $\lacc |z|=\sqrt{k\hbar} \racc$ in $\barg$. For $T_N(f)$, we use Hypothesis \ref{hyp_local} to apply the following results.

\begin{proposition}[\cite{dele21} Theorem C]
\label{prop_concentr}
Let $f$ be a real-analytic function on $M$ and $E\in\Cm$. Let $u_N$ be a sequence of eigenfunctions for $T_N(f)$ with eigenvalues $\lambda_N$ such that $\lambda_N \rightarrow E$. Then for all open set $V$ at positive distance from $\lacc f=E\racc$
\[
\lnor u_N \rnor_{L^2(V)} = O\lpar e^{-cN}\rpar.
\]
\end{proposition}

The result is a bit different when $M=\Cm$ as it is not compact. Recall that in this case, we supposed that $f\in S(\Cm,m)$ with $m$ and order function, we now require that $f$ is elliptic, in the sense that there exists $\rho,C,R >0$ such that $|\widetilde{f}(x,\overline{y})| \ge \frac{1}{C}m(x)$ for all $(x,\overline{y})\in \diag+\widetilde{B_{\rho}}$ with $|x|\ge R$.

\begin{proposition}
\label{prop_link_weyl}
Let $f \in S(\Cm,m)$, $m'$ an order function, and $W$ be a weight such that $|z|^2+W$ is uniformly plurisubharmonic, $W,\, \nabla^2 W	\in L^{\infty}$, and $\lnor \nabla W\rnor_{L^{\infty}} \le \epsilon$ with $\epsilon>0$ small enough depending on $f$. Then, there exists $p_{\hbar}\in S(\Cm,m)$ such that

\[
T_{\hbar}(f)u(x) = Op^W(p_{\hbar})u(x) = \int_{\Gamma} e^{\frac{-|x|^2+2(x-y)\overline{v} + |y|^2}{2\hbar}} \widetilde{p_{\hbar}}\lpar \frac{x+y}{2},\overline{v}\rpar u(y) \frac{dy\wedge d\overline{v}}{2\pi\hbar}
\]
for all $u\in \statesp[W,m']$, with $\Gamma = \lacc \lpar y, \frac{\overline{x}+\overline{y}}{2}-\frac{\delta(\overline{x}-\overline{y})}{2\langle x-y\rangle} \rpar \middle/\; y\in\Cm \racc$ a contour in $\Cm^2$, and this operator is bounded on $\statesp[W,m'] \rightarrow \statesp[W,\frac{m'}{m}]$ uniformly in $\hbar$. Moreover, $p_{\hbar} = e^{\frac{\hbar}{2}\partial\overline{\partial}} f_{\hbar} \sim \sum \frac{\hbar^k}{2^k k!} \lpar \partial\overline{\partial}\rpar^k f_{\hbar}$, where $e^{\frac{\hbar}{2}\partial\overline{\partial}}$ is the heat kernel at time $1$, and $p_{\hbar}$ is elliptic if $f_{\hbar}$ is elliptic.
\end{proposition}

\begin{proof}
The bound on $Op^W(p)$ for $p\in S(\Cm,m)$ is proved in \cite{hitr24} Proposition 4.1, we prove the relation between this quantisation and Toeplitz operators, following the scheme of \cite{sjos96} Proposition 1.3. Let $u(y) = e^{-\frac{|y|^2}{2\hbar}} r(y) \in \statesp[W,m']$ with $r$ holomorphic,
\begin{align*}
Op^W_t(p_{\hbar,t})u(x)
	= & e^{-\frac{|x|^2}{2\hbar}} \int_{\Gamma_t} e^{\frac{(x-y)\overline{v}}{\hbar}} \widetilde{p_{\hbar,t}}\lpar y+t\frac{x-y}{2},\overline{v}\rpar r(y) \frac{dy\wedge d\overline{v}}{\pi\hbar}\\
	= & e^{-\frac{|x|^2}{2\hbar}} \int_{\Cm} e^{\frac{(x-y)\overline{v_t}(y)}{\hbar}} \widetilde{p_{\hbar,t}}\lpar y+t\frac{x-y}{2},\overline{v_t}(y)\rpar r(y) \overline{\partial}\overline{v_t}(y) \frac{dy\wedge d\overline{y}}{\pi\hbar}
\end{align*}
with $\Gamma_t = \lacc \lpar y, \overline{y}+t\frac{\overline{x}-\overline{y}}{2}-\frac{\delta(\overline{x}-\overline{y})}{2\langle x-y\rangle} \rpar \middle/\; y\in\Cm \racc = \lacc \lpar y,\overline{v_t}(y) \rpar \middle/\; y\in\Cm \racc$ and $p_{\hbar,0} = f_{\hbar}$, in particular using Proposition \ref{prop_general_contours} between $\Gamma_0$ and $\diag$ gives that $Op_0^W(p_{\hbar,0}) = T_{\hbar}(f_{\hbar})$. We thus look for $p_{\hbar,t}$ such that the operator is independent of $t$, and the solutions will be $p_{\hbar} = p_{\hbar,1}$. Differentiating $e^{\frac{|x|^2}{2\hbar}} Op^W_t(p_{\hbar,t})u(x)$ with respect to $t$ gives

\begin{multline*}
	\int_{\Cm} e^{\frac{(x-y)\overline{v_t}(y)}{\hbar}} \lbra \overline{\partial}\overline{v_t} \lpar \frac{\lpar \overline{x}-\overline{y}\rpar (x-y)}{2\hbar} \widetilde{p_{\hbar,t}} + \partial_t\widetilde{p_{\hbar,t}} \right.\right.\\
	\left. \left. + \frac{x-y}{2}\partial_1 \widetilde{p_{\hbar,t}} + \frac{\overline{x}-\overline{y}}{2} \partial_2 \widetilde{p_{\hbar,t}} \rpar - \frac{1}{2} \widetilde{p_{\hbar,t}} \rbra \frac{dy\wedge d\overline{y}}{\pi\hbar}.
\end{multline*}
Now, notice that

\begin{gather*}
\partial_{\overline{y}}\lpar e^{\frac{(x-y)\overline{v_t}}{\hbar}} \frac{\overline{x}-\overline{y}}{2} \widetilde{p_{\hbar,t}} \rpar = e^{\frac{(x-y)\overline{v_t}}{\hbar}} \lpar (x-y)\overline{\partial}\overline{v_t}(y) \frac{\overline{x}-\overline{y}}{2\hbar} \widetilde{p_{\hbar,t}} - \frac{1}{2} \widetilde{p_{\hbar,t}} + \overline{\partial} \overline{v_t} \frac{\overline{x}-\overline{y}}{2} \partial_2 \widetilde{p_{\hbar,t}} \rpar,\\
\partial_{\overline{y}} \lpar \frac{\hbar}{2} e^{\frac{(x-y)\overline{v_t}}{\hbar}} \partial_1 \widetilde{p_{\hbar,t}} \rpar =  e^{\frac{(x-y)\overline{v_t}}{\hbar}} \overline{\partial}\overline{v_t} \lpar \frac{x-y}{2} \partial_1\widetilde{p_{\hbar,t}} + \frac{\hbar}{2} \partial_1\partial_2 \widetilde{p_{\hbar,t}} \rpar,
\end{gather*}
and since $\int_{\Cm} \partial_{\overline{y}} F(x,y,\overline{y}) dy\wedge d\overline{y} = 0$ for any analytic function $F$ such that the integral converges, we can make integrations by parts in the formula of the derivative in $t$, which gives

\[
\int_{\Cm} e^{\frac{(x-y)\overline{v_t}}{\hbar}} \lpar \partial_t \widetilde{p_{\hbar,t}} - \frac{\hbar}{2} \partial_1\partial_2 \widetilde{p_{\hbar,t}} \rpar \overline{\partial}\overline{v_t} \frac{dy\wedge d\overline{y}}{\pi\hbar}.
\]
We now prove that $e^{t\partial\overline{\partial}}f_{\hbar}\in S(\Cm,m)$ for all $t\in \lbra 0,\frac{\hbar}{2} \rbra$, which implies that $Op_t^W \lpar e^{t\partial\overline{\partial}}f_{\hbar} \rpar$ is well-defined and the formal computations we did make sense, thus the result. Denote $\Kc_t(x,y) = \frac{e^{-\frac{x^2+y^2}{4t}}}{4\pi t}$ the heat kernel on $\Rm^2 \simeq \Cm$, we prove that $\lver \widetilde{ \Kc_t \star f_{\hbar}}(x,\overline{y}) \rver \le C m(x)$ for all $(x,\overline{y})\in \diag + \widetilde{B_{\rho}}$,

\begin{align*}
\lver \frac{1}{m(x)} \widetilde{ \Kc_t \star f_{\hbar}}(x,\overline{y}) \rver
	\le & \int_{\Rm^2} \frac{1}{m(x)} \lver \Kc_t(p,q)\rver \lver \widetilde{f_{\hbar}}\lpar x-(p+iq),\overline{y}-(p-iq)\rpar\rver dpdq\\
	\le & C \int_{\Rm^2} (1+p^2+q^2)^{\frac{N}{2}} e^{-\frac{p^2+q^2}{4t}} \frac{dpdq}{4\pi t}\\
	\le  & C \int_{\Rm^2} (1+2t(p^2+q^2))^{\frac{N}{2}} e^{-\frac{p^2+q^2}{2}} \frac{dpdq}{2\pi}\\
	\le & C \frac{N!}{1-4t}
\end{align*}
with $N\in\Nm$ by definition of $m$, and for $t$ small enough, so $p_{\hbar}$ is of order $m$. Now, let $\rho>0$ be such that $\widetilde{f_{\hbar}} \le Cm$ on $\diag + \widetilde{B_{\rho}}$, and $\sigma \in ]0,\rho[$. Let $K\subset \diag$ be a compact, then

\[
p_{\hbar}(x,\overline{y}) = \int_{B_{\sigma}} \Kc_t(p,q) \widetilde{f_{\hbar}}\lpar x-(p+iq),\overline{y}-(p-iq)\rpar dpdq + O(e^{-\frac{\sigma^2}{2\hbar}}),
\]
uniformly for $(x,\overline{y})\in K + \widetilde{B_{\rho}}$. Indeed, the difference of the two expression gives

\begin{align*}
	& \lver \int_{B_{\sigma}^c} \Kc_{\frac{\hbar}{2}}(p,q) \widetilde{f_{\hbar}}\lpar x-(p+iq),\overline{y}-(p-iq)\rpar dpdq \rver\\
	\le & \int_{B_{\sigma}^c} e^{-\frac{p^2+q^2}{2\hbar}} \lver \widetilde{f_{\hbar}}\lpar x-(p+iq),\overline{y}-(p-iq)\rpar\rver dpdq\\
	\le & C m(x) \int_{\Rm^2} (1+p^2+q^2)^{\frac{N}{2}} e^{-\frac{p^2+q^2}{2\hbar}} \frac{dpdq}{2\pi\hbar}\\
	\le  & C_{\sigma} m(x) e^{-\frac{\sigma^2}{2\hbar}},
\end{align*}
and since $m$ is $C^{\infty}$, $m(x)$ is uniformly bounded for $x$ in a compact. Then,

\begin{align*}
p_{\hbar}(x,\overline{y})
	= & \int_{B_{\sigma}} e^{-\frac{p^2+q^2}{2\hbar}} \widetilde{f_{\hbar}}\lpar x-(p+iq),\overline{y}-(p-iq)\rpar dpdq + O(e^{-\frac{\sigma^2}{2\hbar}})\\
	= & \sum\limits_{0\le k \le \frac{\min(\rho-\sigma,\sigma)}{\hbar}} \frac{\hbar^k}{k!} (\partial\overline{\partial})^k \widetilde{f_{\hbar}}\lpar x,\overline{y} \rpar + O(e^{-\frac{c}{\hbar}}),
\end{align*}
according to Lemma \ref{le_int_gauss}, uniformly for all $(x,\overline{y}) \in K+\widetilde{B_{\sigma}}$. Then, with similar arguments as in the proof of Proposition \ref{prop_prod_symbol}, we get that $p_{\hbar} \sim \sum \hbar^k p_k$ uniformly on $K+\widetilde{B_{\sigma}}$ with $\widetilde{p_k} \le B^{k+1} k!$. Hence, $p_{\hbar}$ is a global analytic symbol of order $m$.

According to the previous computations, $\widetilde{p_{\hbar}}(x,\overline{y}) = \widetilde{f_{\hbar}}(x,\overline{y}) + m(x)O(\hbar)$ uniformly on $\diag+\widetilde{B_{\sigma}}$, and $f_{\hbar}$ is elliptic by hypothesis, so $p_{\hbar}$ is elliptic too for $\hbar$ small enough.
\end{proof}

As we saw earlier, the operators we built are defined on spaces with weights like $W+\mathbf{w}$ and $W_i+\mathbf{w}_i$. Although, these two weights are only defined locally, and we will need specific bounds at infinity, thus the next result.

\begin{lemma}
\label{le_exist_weights}
There exists weights $\Wc,\, \Wc_i$ on $M$ such that
\begin{itemize}
\item $\Wc = W + \mathbf{w}$ and $\Wc_i = W_i + \mathbf{w}_i$ on $B_{\rho}(x_0)$ for $\rho$ small enough, with $W,\, W_i$ given by Theorem \ref{th_FIO_normal}, and $\mathbf{w},\, \mathbf{w}_i = O(|x|^3)$.
\item $|x|^2+\Wc(x)$ and $|x|^2+\Wc_i(x)$ are uniformly plurisubharmonic.
\item $\Wc,\, \nabla^2\Wc ,\, \Wc_i,\, \nabla^2\Wc_i  \in L^{\infty}(\Cm)$
\item $\lnor \nabla\Wc \rnor_{L^{\infty}},\, \lnor \nabla\Wc_i \rnor_{L^{\infty}} \le \epsilon$. with $\epsilon>0$ small enough depending on $\delta$ in Hypothesis \ref{hyp_local}.
\end{itemize}
\end{lemma}

\begin{proof}
Consider a function $\chi: \Cm \rightarrow [0,1]$ such that $\chi = 1$ on $B_{\rho}$, $\chi = 0$ outside $B_{\tau\epsilon}$ for $\tau>0$, then $\Wc(x) = \chi(x) (W(x)+\epsilon(x)),\, \Wc(x) = \chi(x) (W_i(x)+\epsilon_i(x))$ satisfy the hypothesis for $\rho$ and $\tau$ small enough.
\end{proof}

\begin{corollary}
\label{cor_ellip_Barg}
With the hypothesis and notations of Proposition \ref{prop_link_weyl}, suppose that $f_{\hbar}$ satisfies the ellipticity condition outside $\widetilde{B_R}$ with $R>0$, then there exists $C>0$ such that for all $u\in \statesp[\Wc,m]$

\[
\lnor u \rnor_{\statesp[\Wc,m](B_{2R}^c)} \le C\lnor T_{\hbar}(f_{\hbar})u \rnor_{\statesp[\Wc,m](B_{2R}^c)} + \lnor u \rnor_{\statesp[\Wc,m](\Cm)}O(e^{-\frac{c}{\hbar}}),
\]
for $\hbar$ small enough.
\end{corollary}

\begin{proof}
According to Proposition \ref{prop_link_weyl}, the problem is the same replacing $f_{\hbar}$ by $p_{\hbar} = e^{\frac{\hbar}{2}\partial\overline{\partial}} f_{\hbar}$ which is still elliptic,
and Proposition 4.3 of \cite{hitr24} gives the result.
\end{proof}

\begin{theorem}
\label{th_elliptic}
Assume Hypothesis \ref{hyp_local} are satisfied. If $M=\Cm$, it means that there exists an order function $m$ such that $f\in S(\Cm,m)$, and we also suppose that $f$ is elliptic, meaning that there exists $\rho,\, C,\, R >0$ such that $\lver \widetilde{f}(x,\overline{y})\rver \ge Cm(x)$ for all $(x,\overline{y}) \in \diag+ \widetilde{B_{\rho}}$ with $|x|\ge R$. Moreover, assume that the Hessian of $f_0$ at $x_0$, denoted $H$, satisfies \eqref{eq_complex_quad} as a quadratic form. Therefore, there exists $\alpha\in\Sm$ such that $\alpha H$ is positive definite. Then, there exists $C,A,r>0$ and real-analytic functions $\mu_k$ such that $|\mu_k(x)|\le Ck!A^k$ for all $x$, and such that the spectrum of $T_N(f)$ satisfies

\begin{gather*}
\sigma(T_N(f)) \cap B_r(f_0(x_0)) = (\lambda_l)_{l\in\Nm^*}\bigcap B_r(f_0(x_0))\\
\lambda_l = \sum\limits_{k=0}^{\frac{1}{A\hbar}} \hbar^k \mu_k\lpar\hbar \lpar \frac{\sqrt{\det (\alpha H)}}{\alpha} \frac{2l+1}{2} + \frac{\partial\overline{\partial}H}{2} \rpar\rpar + O(e^{-\frac{c}{\hbar}}),
\end{gather*}
where each $\lambda_l$ is a simple eigenvalue. Furthermore, there exists $C>0$ such that for any $\lambda\in B_r(f_0(x_0))\cap\sigma(T_N(f))^c$

\[
\lnor (T_N(f)-\lambda)^{-1} \rnor_{\statesp[\Wc](M)\rightarrow\statesp[\Wc_i](M)} \le \frac{C}{d(\lambda,\sigma(T_N(f))}
\]
where the weights $\Wc,\, \Wc_i$ are described in Lemma \ref{le_exist_weights}, in particular on a neighbourhood of $0$

\begin{align*}
\Wc(x) = & \frac{r-1}{r}|x|^2 + \frac{2j}{r}\Re(x)\Im(x) + O(|x|^3),\\
\Wc_i(x) = & \frac{1-r}{r}|x|^2 + \frac{2j}{r}\Re(x)\Im(x) + O(|x|^3),
\end{align*}
where $r,j\in\Rm_+^*\times\Rm$ are such that $(r+ij)^2 = \frac{\zeta}{|\zeta|}$, with

\[
\zeta = \det(\Re (\alpha H)) + \det(\Im (\alpha H)) - i\sqrt{\lver \Im(\det (\alpha H))^2-4\det(\Re (\alpha H))\det(\Im (\alpha H))\rver} \notin \Rm_-.
\]
\end{theorem}

\begin{proof}
We recall that there exists $\rho>0$, a neighbourhood $U$ of $x_0$ and bounded operators $\Fio$, $\Fio_i$ such that

\begin{align*}
T_N(f)\Fio = & \Fio \mu\lpar T_{\hbar}(H_0) \rpar + O_{\barg(B_{\rho})\rightarrow\statesp[\Wc_i](U)}\lpar e^{-\frac{c}{\hbar}}\rpar,\\
\Fio\Fio_i = & id_{\statesp[\Wc](U)} + O_{\statesp[\Wc](U)\rightarrow\statesp[\Wc_i](U)}(e^{-\frac{c}{\hbar}}).
\end{align*}
Hence, there exists $c'>0$ such that the analytic $c'$-pseudo spectrum of $T_N(f)$ is of the form $\cup_{l\in\Nm}V_l$, where $V_l$ is a neighbourhood of

\[
\nu_l = \sum\limits_{k=0}^{\frac{1}{A\hbar}} \hbar^k \mu_k \lpar \hbar \lpar \frac{\sqrt{\det (\alpha H)}}{\alpha} \lpar j+\frac{1}{2}\rpar + \frac{\partial\overline{\partial}H}{2}\rpar\rpar
\]
with size $O\lpar e^{-\frac{c}{\hbar}}\rpar$. This means that the $V_l$ are disconnected, since the distance between two successive $\nu_l$ is $C\hbar + O(\hbar^2)$ with $C>0$, according to Proposition \ref{prop_action_int}. According to Lemma \ref{le_pseudo_spec}, there exists at least one eigenvalue $\lambda_l$ in each $V_l$. We first prove that $\lambda_l$ is the only eigenvalue in $V_l$, and its geometric multiplicity is $1$. Let $u_l,v_l \in\barg$ be eigenfunctions of $T_N(f)$ for $\lambda_l,\lambda'_l\in V_j$. Then $\lambda'_l=\lambda_l + O(e^{-\frac{c}{\hbar}}) = \nu_l + O(e^{-\frac{c}{\hbar}})$ and $u_l,v_l$ are both analytic quasimodes for $\nu_l$. We will make a proof by contradiction, suppose that $u_l\neq v_l$. Since $u_l,v_l$ are small outside a neighbourhood of $x_0$ by Proposition \ref{prop_concentr} and Corollary \ref{cor_ellip_Barg}, we can see them as functions in $\statesp[\Wc](U)$, and orthonormalise them in this space while keeping the quasimode property. Then, denoting $u'_l=\Fio_i u_l$ and $v'_l=\Fio_i v_l$,

\begin{equation}
\label{eq_quasi}
\mu(T_{\hbar}(H_0))u'_l = \nu_l u'_l + O_{\barg(B_{\rho})}(e^{-\frac{c}{\hbar}}), \quad \mu(T_{\hbar}(H_0))v'_l = \nu_l v'_l + O_{\barg(B_{\rho})}(e^{-\frac{c}{\hbar}}).
\end{equation}
By definition of $\barg$, there exists $(\alpha_j,\beta_j)_{j\in\Nm}$ such that $u'_l = e^{-\frac{|z|^2}{2\hbar}}\sum\limits_{j\in\Nm}\alpha_j z^j + O(e^{-\frac{c}{\hbar}})$, $v'_l = e^{-\frac{|z|^2}{2\hbar}}\sum\limits_{j\in\Nm}\beta_j z^j + O(e^{-\frac{c}{\hbar}})$, and according to Lemma \ref{le_spectrum_normal}, \eqref{eq_quasi} becomes

\[
\sum\limits_{j\in\Nm} |\nu_j-\nu_l|^2 |\alpha_j|^2 = O(e^{-\frac{c}{\hbar}}), \quad \sum\limits_{j\in\Nm} |\nu_j-\nu_l|^2 |\beta_j|^2 = O(e^{-\frac{c}{\hbar}}).
\]
Since the distance between two different $\nu_j$ is of size at least $\hbar$, we compute

\[
\lnor u'_l - \alpha_l e^{-\frac{|z|^2}{2\hbar}}z^l \rnor_{\barg} = \sum\limits_{j\neq l} |\alpha_j|^2 = O(\hbar^{-1}) \sum\limits_{j\in\Nm} |\nu_j-\nu_l|^2 |\alpha_j|^2 = O(e^{-\frac{c'}{\hbar}}),
\]
with $c'\in ]0,c[$, and the same applies to $v'_l$, so $u'_l = v'_l + O_{\barg}(e^{-\frac{c'}{\hbar}})$. Then, using $\Fio$,

\[
u_l = v_l + O_{\statesp[\Wc_i](U)}(e^{-\frac{c'}{\hbar}}).
\]
Although $\Wc \le \Wc_i$, for $c>0$ there exists a neighbourhood of $x_0$ on which $\Wc_i-c' \le \Wc$. Hence, if we take $U$ small enough,

\[
u_l = v_l + O_{\statesp[\Wc](U)}(e^{-\frac{c'}{\hbar}}),
\]
thus Proposition \ref{prop_concentr} and Corollary \ref{cor_ellip_Barg} give $u_l = v_l + O_{\statesp[\Wc]}(e^{-\frac{c'}{\hbar}})$, which is absurd as $u_l,v_l$ are orthonormal. Then, $u_l=v_l$ and $\lambda_l=\lambda'_l$, so $\lambda_l$ is the unique eigenvalue in $V_l$, and it has geometric multiplicity $1$. Now, we prove that it has algebraic multiplicity equal to $1$. We also prove it by contradiction: for $l\in\Nm^*$, let $a\in\Nm^*$ and $v_l\in\statesp[\Wc]$ be such that $\lpar T_N(f)-\lambda_l \rpar^a v_l = 0$. We suppose that $v_l$ is different from the eigenfunction $u_l$ of eigenvalue $\lambda_l$, we can thus build an orthonormal family $v_l,w_l$ such that $\lpar T_N(f)-\lambda_k \rpar^a w_l = 0$. Then, the functions $v'_l = I_i v_l$, $w'_l = I_i w_l$ satisfy

\[
\lpar \mu(T_{\hbar}(H_0))-\lambda_l \rpar^a u'_l = O_{\barg}(e^{-\frac{c}{\hbar}}), \quad \lpar \mu(T_{\hbar}(H_0))-\lambda_l \rpar^a w'_l = O_{\barg}(e^{-\frac{c}{\hbar}}).
\]
By the same argument as before, it implies that $v'_l = w'_l + O_{\barg}(e^{-\frac{c}{\hbar}})$ and then $v_l = w_l + O_{\statesp[\Wc_i]}(e^{-\frac{c}{\hbar}})$, which is absurd. Thus, $v_l=u_l$ and the eigenvalue $\lambda_l$ is simple. This proves that the discrete spectrum is of the predicted form. If $M$ is a compact manifold, the quantum space has finite dimension, so the spectrum is pure point anyway.

Now, if $M=\Cm$ we prove that the spectrum contains only eigenvalues. We take $x_0=0=f_0(x_0)$ here to simplify the notations. $H$ is elliptic by hypothesis, so it grows at least like a constant times $|z|^2$. Moreover, $H(\Rm^2) \neq \Cm$ so there exists $\theta$ and $\delta,\epsilon>0$ such that $|H(z)+\frac{\epsilon}{2}e^{i\theta}| > \epsilon m(z)$ for all $|z|\le \delta$. Hence, the same inequality is satisfied with $H$ replaced by $f$, for $\delta$ possibly smaller. Then using the ellipticity of $f$, and taking $\epsilon$ smaller if necessary, $|f_0(z)+\frac{\epsilon}{2}e^{i\theta}| > \epsilon m(z)$ for all $z\in\Cm$. In the same manner, if $\chi\in C^{\infty}_c(\Cm)$ is such that $\chi(z)=\frac{\epsilon}{2}e^{i\theta}$ for all $z\in B_{\delta}$, and $|\chi|\le \frac{\epsilon}{2}$, then  $|f_0(z)+\chi(z)| > \epsilon m(z)$ for all $z\in\Cm$. According to \cite{hitr24} Section 5, an operator $Op^W(p)$ with $p\in S(\Cm,m)$ elliptic is invertible, then it is also true for Toeplitz operators according to Proposition \ref{prop_link_weyl}. Thus, $(T_{\hbar}(f) + T_{\hbar}(\chi) -z)^{-1}$ is holomorphic in a neighbourhood of $0$. It implies that for $z$ in a neighbourhood of $0$ and in the resolvent set of $T_{\hbar}(f)$,

\[
(T_{\hbar}(f)-z)^{-1} = (T_{\hbar}(f) + T_{\hbar}(\chi) -z)^{-1}\lpar \id - T_{\hbar}(\chi) (T_{\hbar}(f) + T_{\hbar}(\chi) -z)^{-1} \rpar^{-1}.
\]
Moreover, $F_{\hbar} (z) = \id - T_{\hbar}(\chi) (T_{\hbar}(f) + T_{\hbar}(\chi) -z)^{-1}$ is Fredholm because $T_{\hbar}(\chi)$ is compact. Since $T_{\hbar}(f)+\frac{\epsilon}{2}e^{i\theta}$ is invertible, so is $F_{\hbar}\lpar\frac{\epsilon}{2}e^{i\theta}\rpar$. Hence, using the results of Fredholm theory, see for instance \cite{zwor12} Theorem D.4, this means that $F_{\hbar}(z)^{-1}$ is meromorphic on a neighbourhood of $0$, and so is $(T_{\hbar}(f)-z)^{-1}$. Thus, the spectrum of $T_{\hbar}(f)$ is only made of eigenvalues with finite multiplicity.

As for the resolvent estimate, according to Lemma \ref{le_spectrum_normal}, $\mu(T_{\hbar}(H_0))$ is diagonalisable, so it satisfies the usual estimate on $L^2$. For all $u\in\statesp[\Wc_i]$, using the information on $\Fio,\Fio_i$,

\begin{align*}
\frac{\lnor (T_N(f)-\lambda)u\rnor_{\statesp[\Wc](U)}}{\|u\|_{\statesp[\Wc_i](U)}}
	= & \frac{\lnor \Fio(\mu(T_{\hbar}(H_0))-\lambda)\Fio_i u\rnor_{\statesp[\Wc](U)}}{\|u\|_{\statesp[\Wc_i](U)}} + O(e^{-\frac{c}{\hbar}})\\
	\ge & \frac{\lnor \Fio_i\Fio(\mu(T_{\hbar}(H_0))-\lambda)\Fio_i u\rnor_{\barg(B_{\rho})}}{\|\Fio_i u\|_{\barg(B_{\rho})}} + O(e^{-\frac{c}{\hbar}})\\
	\ge & \frac{\lnor (\mu(T_{\hbar}(H_0))-\lambda)\Fio_i u\rnor_{\barg(B_{\rho})}}{\|\Fio_i u\|_{\barg(B_{\rho})}} + O(e^{-\frac{c'}{\hbar}})\\
\end{align*}

with $c'>0$ if $\rho$ is small enough as seen before. Then,
\begin{align*}
\frac{\lnor (T_N(f)-\lambda)u\rnor_{\statesp[\Wc](U)}}{\|u\|_{\statesp[\Wc_i](U)}}
	\ge & \frac{1}{\lnor(\mu(T_{\hbar}(H_0))-\lambda)^{-1}\rnor} + O(e^{-\frac{c'}{\hbar}})\\
	\ge & d(\lambda,\sigma(\mu(T_{\hbar}(H_0)))) + O(e^{-\frac{c'}{\hbar}})\\
	& = d(\lambda,\sigma(T_N(f))) + O(e^{-\frac{c'}{\hbar}})
\end{align*}
which proves the estimates near $x_0$. Outside $U$, $f$ stays at positive distance from $\lambda$.For $M$ compact, it is due to the continuity of $f$, and for $M=\Cm$ it is due to the growth hypothesis at infinity. In both case, the estimate outside $U$ is given by usual symbolic calculus.
\end{proof}

For $M=\Cm$, we adapted arguments from the proof of \cite{hitr24} Proposition 4.5.

Actually, the equation of Theorem \ref{th_elliptic} can be written as a quantisation condition similar to Bohr-Sommerfeld. Unlike for self-adjoint operators, the level sets of the symbol $\lacc \widetilde{f_0}=E\racc$ are subsets of $\widetilde{M}$, and the action integrals are defined by Proposition \ref{prop_action_int}. For $E\in\Cm$ close but distinct to $f_0(x_0)$, we can build a quasimode $S$ by a WKB method like in the usual case, but along a closed curve in the complexified space $\gamma_E: [0,1] \to \lacc \widetilde{f_0}=E\racc$. Then, $S$ is a section of $L\boxtimes\overline{L}\rightarrow \widetilde{M}$ along $\gamma_E$, and $S(\gamma_E(1)),\, S(\gamma_E(0))$ are equal if and only if $E$ is an eigenvalue. In general the quotient of the two is of the form $e^{2i\pi I(E)}$, and the condition $I(E)\in\Zm$ gives a necessary condition for $E$ to be an eigenvalue. This equation involves the \emph{Bohr-Sommerfeld class} of some curves, for which we recall the definition.

We refer to \cite{lef14a} and \cite{dele25} for the WKB method applied to self-adjoint and non self-adjoint operators respectively, both near a non-critical value.

\begin{definition}
Let $L\rightarrow M$ be a line bundle, $\Gamma$ be a Lagrangian submanifold of $\widetilde{M}$, then for any closed curve $\gamma: [0,1] \rightarrow \Gamma$, and section $S$ of $L\boxtimes\overline{L}\rightarrow \widetilde{M}$ along $\gamma$ such that $\widetilde{\nabla}_{\gamma'}S=0$, the quotient $\frac{S(1)}{S(0)}$ is the Bohr-Sommerfeld class of $\gamma$ relative to $L$. Moreover, the number $I(\gamma)$ such that the class is written $e^{2i\pi I(\gamma)}$ is the generator of the Bohr-Sommerfeld class.
\end{definition}

Notice that the action integral functional from Proposition \ref{prop_action_int} corresponds to the Bohr-Sommerfeld index of closed curves in $\lacc \widetilde{f_0}=E\racc$ relative to the prequantum line bundle of $M$.

In order to use the WKB method, we now compute the shape of the image of a general section by a Toeplitz operator in local coordinates.

\begin{proposition}
\label{prop_WKB}
Let $x_0\in M$, $\rho>0$, $f\in S(B_{\rho}(x_0))$, $\psi$ be a holomorphic function on $B_{\rho}(x_0)$ independent of $\hbar$, and $a$ such that $e^{-\frac{\Phi_M}{2\hbar}}e^{\frac{\psi}{\hbar}}a \in \statesp[\Phi_M]$. Denote $\overline{v_c}(x)$ the function such that locally $\partial\psi(x) = \partial_1 \widetilde{\Phi_M}\lpar x,\overline{v_c}(x)\rpar$. Then there exists $\eta>0$, and $b\in S(B_{\eta}(x_0))$, such that $T_N(f)\lpar e^{-\frac{\Phi_M}{2\hbar}}e^{\frac{\psi}{\hbar}}a\rpar(x) = e^{-\frac{\Phi_M(x)}{2\hbar}}e^{\frac{\psi(x)}{\hbar}}b(x)$ for all $x\in B_{\eta}(x_0)$. Moreover,

\begin{align*}
b_0(x) = & \widetilde{f}_0(x,\overline{v_c}(x))a_0(x),\\
b_1(x) = & \lpar f_1 + Q\lpar\overline{\partial}\widetilde{f}_0\rpar \rpar(x,\overline{v_c}(x))a_0(x) + i\widetilde{X_{f_0}}(x,\overline{v_c}(x))\cdot a_0(x) + \widetilde{f_0}(x,\overline{v_c}(x))a_1(x),
\end{align*}
where $Q$ is a linear order $1$ differential operator, and $X_{f_0}$ is the Hamiltonian vector field of $f_0$, which is defined by $df_0(\cdot) = \omega(X_{f_0},\cdot)$. If $\widetilde{X_{f_0}}$ is tangent to the Lagrangian $\Lc_{\psi} = \lacc (x,\overline{v_c}(x)) \racc$, then

\begin{multline*}
b_1(x)
	= \lpar \lpar\widetilde{f}a\rpar_1 + i\widetilde{X_{f_0}}a_0\rpar \lpar x,\overline{v_c}(x)\rpar\\
	+ a_0(x) \biggl[\lpar\partial_1\partial_2\widetilde{\Phi_M}\rpar^{-1}\lpar\frac{1}{2}\partial_2\log \widetilde{B_0} \partial_1 \widetilde{f_0} - \partial_1\log \widetilde{B_0} \partial_2 \widetilde{f_0}\rpar\\
	+ \frac{1}{2}\widetilde{\Delta f_0} + \frac{i}{2}\dive_{\Lc_{\psi}} \lpar\widetilde{X_{f_0}}\rpar \biggr] \lpar x,\overline{v_c}(x)\rpar.
\end{multline*}

with $B_0$ the principal term of the Bergman kernel $B$ on $M$.
\end{proposition}

\begin{proof}
To simplify the notations, we take $x_0=0$. According to Proposition \ref{prop_local_berg}, there exists $\rho,\eta>0$, and $B\in S\lpar\widetilde{B_{\rho}}\rpar$ such that

\[
T_N(f)u(x) = \int_{B_{\rho}} e^{{\frac{F(x,z,\overline{z})}{\hbar}}} \Pc(x,\overline{z}) f(z) a(z) \frac{dz}{\pi\hbar} + O_{\statesp[\Phi_M](B_{\rho})}(e^{-cN})
\]
for all $u\in \statesp[\Phi_M](B_{\rho})$ and $x\in B_{\eta}$, with

\[
F(x,z,\overline{v}) = \widetilde{\Phi_M}(x,\overline{v})-\widetilde{\Phi_M}(z,\overline{v})+\psi(z)-\frac{1}{2}\Phi_M(x).
\]
Then, for all $x$, the unique critical point of $(z,\overline{v}) \mapsto F(x,z,\overline{v})$ is $(x,\overline{v_c}(x))$. We will compute this integral with a different method than in Lemma \ref{le_WKB}, instead of diagonalising $F$, for $x\in B_{\eta}$ we consider the diffeomorphism $K_x$ such that

\[
\hess_{z\overline{v}}F\lpar x,x,\overline{v_c}(x)\rpar \Bigr( K_x(z,\overline{v})-\lpar x,\overline{v_c}(x)\rpar\Bigr) = -z\overline{v}.
\]
thus $K_x$ is affine, and $\nabla K_x$ is the constant matrix such that $K_x(z,\overline{v}) = (x,\overline{v_c}(x)) + \nabla K_x \times (z,\overline{v})$. Denote

\begin{align*}
G(x,z,\overline{v}) = & F(x,z,\overline{v}) - F(x,x,\overline{v_c}(x))\\
	& -\hess_{z\overline{v}}F(x,x,\overline{v_c}(x)) \lpar z-x,\overline{v}-\overline{v_c}(x)\rpar,
\end{align*}
then $G(x,\cdot)$ vanishes at order $3$ at $0$. For $x\in B_{\eta}$, using Lemma \ref{le_holo_ext} and the change of variable $K_x$

\begin{align*}
	& \int_{B_{\rho}} e^{\frac{\phi_1(x,z)+\phi_2(x,z)}{\hbar}} \Pc(x,\overline{z})f(z)a(z) dz\\
	= & e^{\frac{F(x,z_c(x),\overline{v_c}(x))}{\hbar}} s(x) \int_{\Gamma} e^{-\frac{y\overline{v}}{\hbar}} e^{\frac{G\lpar x,K_x(y,\overline{v})\rpar}{\hbar}} \Pc\lpar x,K_x(y,\overline{v})\rpar\widetilde{f}\lpar K_x(y,\overline{v})\rpar a\lpar K_x(y,\overline{v})\rpar dy\wedge d\overline{v},
\end{align*}
where $\Gamma = K_x^{-1}\lpar\diag[\rho] - \lpar x,\overline{v_c}(x)\rpar\rpar$ and $s(x) = \det\nabla K_x$. Though, for $\eta$, $\rho$ small enough, there exists $\tau$ such that for all $x\in B_{\eta}$,

\[
\Gamma \supset K_x^{-1}\lpar\diag- \lpar x,\overline{v_c}(x)\rpar\rpar \cap \widetilde{B_{\tau}},
\]
and since $G$ vanishes at order $3$ at $0$, it is arbitrarily small compared to $z\overline{v}$ near $0$, so if $\tau$ is small enough, then the integral is negligible outside $\widetilde{B_{\tau}}$. Since $\diag$ is a good contour for $-z\overline{v}$, we can apply Proposition \ref{prop_linear_contours}

\begin{align*}
	& \int_{\Gamma} e^{-\frac{y\overline{v}}{\hbar}} e^{\frac{G\lpar x,K_x(y,\overline{v})\rpar}{\hbar}} B\lpar x,K_x(y,\overline{v})\rpar\widetilde{f}\lpar K_x(y,\overline{v})\rpar a\lpar K_x(y,\overline{v})\rpar dy\wedge d\overline{v}\\
	= & \int_{B_{\tau}} e^{-\frac{|y|^2}{\hbar}} e^{\frac{G\lpar x,K_x(y,\overline{y})\rpar}{\hbar}} B\lpar x,K_x(y,\overline{y})\rpar\widetilde{f}\lpar K_x(y,\overline{y})\rpar a\lpar K_x(y,\overline{y})\rpar dy + \lnor e^{\frac{G}{\hbar}} \rnor_{L^{\infty}(B_{\tau})} O(e^{-\frac{c\tau^2}{\hbar}}).
\end{align*}
Then, we apply Lemma \ref{le_int_gauss}, and the remainder is $\lnor e^{\frac{G}{\hbar}} \rnor_{L^{\infty}(B_{\tau})} O\lpar e^{-\frac{c\tau}{\hbar}}\rpar$. Since $G$ vanishes at order $3$ at $0$, the two remainders are $O\lpar e^{-\frac{c'\tau}{\hbar}}\rpar$ if $\tau$ is small enough. Hence, $T_{\hbar}(f)\lpar e^{-\frac{\Phi_M}{\hbar}}e^{\frac{\psi}{\hbar}}a\rpar = e^{-\frac{\Phi_M}{2\hbar}}e^{\frac{\psi}{\hbar}}\lpar b + O(e^{-\frac{c'}{\hbar}}) \rpar$ on $B_{\eta}$ with

\[
b(x) \sim s(x) \sum\limits_k \frac{\hbar^k}{k!} \lpar\partial_y\overline{\partial_y}\rpar^k  \lpar e^{\frac{G\lpar x,K_x(y,\overline{y})\rpar}{\hbar}} B\lpar x,K_x(y,\overline{y})\rpar\widetilde{f}\lpar K_x(y,\overline{y})\rpar a\lpar K_x(y,\overline{y})\rpar\rpar(0,0).
\]
In particular,

\begin{align}
\label{eq_WKB_cond}\nonumber
b_0(x) = & s(x)B_0(x,\overline{v_c}(x))\widetilde{f_0}(x,\overline{v_c}(x))a_0(x)\\
b_1(x) = & s(x) \biggr[\lpar B\widetilde{f}a\rpar_1 + \partial_{K_x}\overline{\partial}_{K_x}\lpar B_0\widetilde{f_0}a_0\rpar + \frac{1}{2}(\partial_{K_x}\overline{\partial}_{K_x})^2(G) B_0\widetilde{f_0}a_0\\\nonumber
	& + \partial_{K_x}\overline{\partial}_{K_x}^2(G) \partial_{K_x}(B_0\widetilde{f_0}a_0) + \partial_{K_x}^2\overline{\partial}_{K_x}(G) \overline{\partial}_{K_x}(B_0\widetilde{f_0}a_0) \biggr]\lpar x,\overline{v_c}(x)\rpar,
\end{align}
where we denote $\partial_{K_x} = \partial K_x^{\intercal} \cdot D$, $\overline{\partial}_{K_x} = \overline{\partial} K_x^{\intercal} \cdot D$, and $D = \begin{pmatrix}
\partial_1\\
\partial_2
\end{pmatrix}$. For $f=1$ we know that $b=a$, therefore $b_0=a_0$ and $s(x)B_0(x,\overline{v_c}(x))=1$. Actually, we can also compute the Hessian of $F$ to find that $s(x) = \lpar\partial_1\partial_2\widetilde{\Phi_M}\rpar^{-1}$, and we recognise the principal term $B_0$ of the Bergman kernel. Now, we determine the expression of $\partial_{K_x}$, by definition

\begin{equation}
\nabla K_x^{\intercal} \hess F \nabla K_x = \begin{pmatrix}
0 & -1\\
-1 & 0 
\end{pmatrix},
\end{equation}
with,

\[
\hess F \lpar x,\overline{v_c}(x)\rpar= \begin{pmatrix}
\partial^2\psi(x) -\partial_1^2\widetilde{\Phi_M} & -\partial_1\partial_2\widetilde{\Phi_M}\\
-\partial_1\partial_2\widetilde{\Phi_M} & 0
\end{pmatrix} \lpar x,\overline{v_c}(x)\rpar = \begin{pmatrix}
iAd^2 & -d\\
-d & 0
\end{pmatrix},
\]
where $d=\partial_1\partial_2\widetilde{\Phi_M}\lpar x,\overline{v_c}(x)\rpar$ and $A = -id^{-2}\lpar\partial^2 \psi(x) -\partial_1^2\widetilde{\Phi_M}\lpar x,\overline{v_c}(x)\rpar\rpar$. Then, one solution is

\[
\nabla K_x = \begin{pmatrix}
1 & 0\\
\frac{1}{2}iAd & d^{-1}
\end{pmatrix}
\]
which gives $\partial_{K_x} = \partial_1 + \frac{1}{2}iAd \partial_2$, $\overline{\partial}_{K_x} = d^{-1}\partial_2$, in particular $\overline{\partial}_{K_x} a = 0$. Considering again the case $f=1$, we know that $b_1 = a_1$, thus

\begin{align*}
0 = & \biggr[B_1a_0 + \partial_{K_x}\overline{\partial}_{K_x}\lpar B_0a_0\rpar + \frac{1}{2}(\partial_{K_x}\overline{\partial}_{K_x})^2(G) B_0a_0\\
	& + \partial_{K_x}\overline{\partial}_{K_x}^2(G) \partial_{K_x}(B_0a_0) + \partial_{K_x}^2\overline{\partial}_{K_x}(G) \overline{\partial}_{K_x}(B_0a_0) \biggr]\lpar x,\overline{v_c}(x)\rpar.
\end{align*}
So from now on, we only consider the terms where at least one derivative hits $\widetilde{f}_0$ in the expression of $b_1$ in \eqref{eq_WKB_cond}. The formula becomes

\begin{multline*}
	b_1 = \biggr[ (fa)_1 + s d^{-1}\lpar\partial_1 + \frac{1}{2}iAd \partial_2\rpar\partial_2\lpar B_0\widetilde{f}_0a_0\rpar\\
	+ d^{-2}\lpar\partial_1 + \frac{1}{2}iAd \partial_2\rpar\partial_2^2G \lpar\partial_1 + \frac{1}{2}iAd \partial_2\rpar\widetilde{f_0} a_0\\
	+ d^{-2}\lpar\partial_1 + \frac{1}{2}iAd \partial_2\rpar^2\partial_2G \partial_2\widetilde{f_0} a_0 \biggr]\lpar x,\overline{v_c}(x)\rpar.
\end{multline*}
Moreover, if $f$ holomorphic we know that $b_1 = (fa)_1$ thus

\[
\lpar sd^{-1}\partial_2B_0 + d^{-2}\lpar\partial_1 + \frac{1}{2}iAd \partial_2\rpar\partial_2^2 G \rpar\lpar x,\overline{v_c}(x)\rpar = 0.
\]
So the equation becomes

\begin{multline*}
b_1(x) = \lpar\widetilde{f}a\rpar_1 \lpar x,\overline{v_c}(x)\rpar + d^{-1}\partial_2\widetilde{f}_0 \partial_1 a_0 \lpar x,\overline{v_c}(x)\rpar\\
	+ a_0(x)\biggr[\widetilde{\Delta f_0} + iA\partial_2\log B_0 \partial_2\widetilde{f_0} + \frac{1}{2}iA\partial_2^2\widetilde{f_0} + \frac{1}{2}iAd^{-1}\lpar\partial_1 + \frac{1}{2}iAd \partial_2\rpar\partial_2^2 G \partial_2 \widetilde{f_0}\\
	+ d^{-2}\lpar\partial_1 + \frac{1}{2}iAd \partial_2\rpar^2\partial_2 G \partial_2\widetilde{f_0} \biggr]\lpar x,\overline{v_c}(x)\rpar
\end{multline*}
where $\Delta = d^{-1}\partial\overline{\partial}$ is the holomorphic Laplacian by definition of $d$, and

\[
\partial_2\log(B_0)\lpar x,\overline{v_c}(x)\rpar = \frac{\partial_2 B_0}{B_0}\lpar x,\overline{v_c}(x)\rpar = s(x)\partial_2 B_0\lpar x,\overline{v_c}(x)\rpar.
\]

Now, suppose that $\widetilde{X_{f_0}} = \lpar\partial_1\partial_2\widetilde{\Phi_M}\rpar^{-1}\lpar -i\partial_2\widetilde{f_0},i\partial_1 \widetilde{f_0}\rpar$ is tangent to $\Lc_{\psi}$, since its tangent bundle is $T\Lc_{\psi} = \lacc \lpar 1,\partial\overline{v_c}(x)\rpar\racc$, the coordinate of $X_{f_0}|_{\Lc_{\psi}}$ is

\[
-i\lpar\partial_1\partial_2\widetilde{\Phi_M}\rpar^{-1}\partial_2\widetilde{f_0}\lpar x,\overline{v_c}(x)\rpar,
\]
and the divergence computes,

\begin{align*}
\dive_{\Lc_{\psi}} (X_{f_0})
	= & \frac{-i}{B_0} \partial \lpar B_0 \lpar\partial_1\partial_2\widetilde{\Phi_M}\rpar^{-1} \partial_2f_0 \lpar x,\overline{v_c}(x)\rpar\rpar\\
	= & \frac{-i}{B_0} \lpar\partial_1+\partial\overline{v_c}(x)\partial_2\rpar \lpar B_0 \lpar\partial_1\partial_2\widetilde{\Phi_M}\rpar^{-1} \partial_2 f_0 \rpar\lpar x,\overline{v_c}(x)\rpar.
\end{align*}
Notice that, by definition of $\overline{v_c}(x)$,

\[
\partial^2\psi(x) = \partial_1^2\widetilde{\Phi_M}(x,\overline{v_c}(x)) + \partial\overline{v_c}(x) \partial_1\partial_2\widetilde{\Phi_M}(x,\overline{v_c}(x))
\]
so $\partial\overline{v_c}(x) = iAd$, and

\begin{multline*}
\dive_{\Lc_{\psi}} (X_{f_0})
	= -i\biggl( \partial_1\log B_0 d^{-1}\partial_2f_0 + iAd \partial_2\log B_0 d^{-1}\partial_2f_0\\
	- \partial_1^2\partial_2\widetilde{\Phi_M} d^{-2}\partial_2f_0 - iAd\partial_1\partial_2^2\widetilde{\Phi_M} d^{-2}\partial_2f_0\\
	+ d^{-1}\partial_1\partial_2f_0 + iA \partial_2^2 f_0 \biggr)\lpar x,\overline{v_c}(x)\rpar.
\end{multline*}
Furthermore, by definition of $G$

\begin{align*}
	& \partial_2^3 G = \partial_2^3 F = 0,\\
	& \partial_1 \partial_2^2 G = \partial_1 \partial_2^2 F = -\partial_1 \partial_2^2 \widetilde{\Phi_M},\\
	& \partial_1^2 \partial_2 G = \partial_1^2 \partial_2 F = -\partial_1^2 \partial_2 \widetilde{\Phi_M},\\
\end{align*}
so combining the last equations, we get the following expression of $b_1$,

\begin{multline*}
b_1(x) = \lpar \lpar\widetilde{f}a\rpar_1 + i\widetilde{X_{f_0}}a_0 \rpar \lpar x,\overline{v_c}(x)\rpar\\
	+ a_0(x)\biggr[ \frac{1}{2}\widetilde{\Delta f_0} + \frac{1}{2}iA \partial_2\log B_0 \partial_2\widetilde{f_0} - \frac{d^{-1}}{2}\partial_1\log B_0 \partial_2\widetilde{f_0}\\
	+ \frac{i}{2}\dive_{\Lc_{\psi}} \lpar\widetilde{X_{f_0}}\rpar - \lpar\frac{d^{-2}}{2}\partial_1^2 + iAd^{-1} \partial_2\partial_1\rpar \partial_2 \lpar\widetilde{\Phi_M}\rpar \partial_2\widetilde{f_0}\biggr] \lpar x,\overline{v_c}(x)\rpar.
\end{multline*}
Notice that $\partial_1\log B_0 = d^{-1}\partial_1^2\partial_2\widetilde{\Phi_M}$, $\partial_2\log B_0 = d^{-1}\partial_1\partial_2^2\widetilde{\Phi_M}$, and since $\widetilde{f_0}(x,\overline{v_c}(x))$ is constant

\[
\partial_1\widetilde{f_0}\lpar x,\overline{v_c}(x)\rpar = -\partial\overline{v_c}(x)\partial_2\widetilde{f_0}\lpar x,\overline{v_c}(x)\rpar = -iAd\partial_2\widetilde{f_0}\lpar x,\overline{v_c}(x)\rpar.
\]
In the end, combining these equations with the expression of $b_1$ gives

\begin{multline*}
b_1(x)
	= \lpar \lpar\widetilde{f}a\rpar_1 + i\widetilde{X_{f_0}}a_0\rpar \lpar x,\overline{v_c}(x)\rpar\\
	+ a_0(x) \biggl[ \frac{1}{2}\widetilde{\Delta f_0} + \frac{d^{-1}}{2}\partial_2\log \widetilde{B_0} \partial_1 \widetilde{f_0} - d^{-1}\partial_1\log \widetilde{B_0} \partial_2 \widetilde{f_0}\\
	+ \frac{i}{2}\dive_{\Lc_{\psi}} \lpar\widetilde{X_{f_0}}\rpar \biggr] \lpar x,\overline{v_c}(x)\rpar.
\end{multline*}
\end{proof}

\begin{proposition}
Let $\widetilde{\kappa}$ be the local symplectormorphism such that $\widetilde{f_0}\circ\widetilde{\kappa}(z,\overline{v}) = \mu_0(d_0z\overline{v})$ with $d_0\in\Cm$ and the contour $L_E  = \widetilde{\kappa}^{-1}\lpar\sqrt{\frac{E}{d_0}}\lacc (e^{it},e^{-it}),\; 0\le t<2\pi\racc\rpar$. For $E$ small enough, $L_E$ is close to the elliptic critical point of $f_0$ but does not cross it, therefore $\widetilde{X_{f_0}}$ the Hamiltonian vector field of $\widetilde{f_0}$ does not vanish on this contour. Denote $\beta$ the unique $1$-form on $L_E$ such that $\beta\lpar\widetilde{X_{f_0}}\rpar=1$. With the notations of Theorem \ref{th_elliptic}, the analytic symbol

\[
\mu^c(\xi) = \sum\limits_{k=0}^{\frac{1}{A\hbar}} \hbar^k \mu_k\lpar\hbar \lpar \frac{\sqrt{\det (\alpha H)}}{\alpha} \lpar\xi+\frac{1}{2}\rpar + \frac{\partial\overline{\partial}H}{2} \rpar\rpar
\]
satisfies, for all $E\neq f(x_0)$ in a neighbourhood of $f(x_0)$

\[
(\mu^c)^{-1}(E) = \frac{A(E)}{2\pi\hbar} + \int_{L_E} \lpar \widetilde{f_1}+\frac{1}{2}\widetilde{\Delta f_0} \rpar\frac{\beta}{2\pi} - I_{\sub}\lpar L_E\rpar - \frac{1}{2} + O(\hbar).
\]
Here, $A$ is the action integral as in Proposition \ref{prop_action_int}, and $I_{\sub}(L_E)$ is the generator of the Bohr-Sommerfeld class of $L_E$ relative to a topologically trivial half-form bundle $\delta$. Moreover, notice that the $-\frac{1}{2}$ is the so-called \emph{Maslov's index}. Hence, the difference between two eigenvalues is

\[
\frac{\sqrt{\det(\alpha H)}}{\alpha A'(0)}\hbar + O(\hbar^2).
\]

Furthermore, since $f$ is an analytic symbol, $\frac{\sqrt{\det(\alpha H)}}{\alpha}$ has an expansion in $\hbar$, we write $d_0$ its principal term. Then, in a neighbourhood of $f(x_0)$, the elements of $\mu_0(d_0\Rm)$ are the numbers $E$ such that $\lacc (z,\overline{v})\in M^2 /\; \widetilde{f_0}(z,\overline{v})=E\racc$ is a totally real Lagrangian. In other words, for any $E$ close to $f(x_0)$ in $\mu_0(d_0\Rm)$, the Bohr-Sommerfeld class of $\lacc \widetilde{f_0} = E\racc$ along the prequantised bundle $L$ is unitary.
\end{proposition}

\begin{proof}
We follow the scheme of proof of \cite{dele25} Proposition 6.7. Let $E$ be close but distinct to an eigenfunction $\lambda_l$, and $\rho : [0,1] \to M$ the closed curve such that $L_E(t)$ is of the form $(\rho(t),\rho_2(t))$. We consider a quasimode $u$ of $T_N(f)$ for the value $E$, seen as an element of $\statesp[\Wc](U)$ with $U$ a neighbourhood of $x_0$ thanks to Proposition \ref{prop_concentr} and Corollary \ref{cor_ellip_Barg}. Then, $u(\rho(1))$ and $u(\rho(0))$ can differ by a number of the form $e^{2i\pi s}$, where we call $s\in\Cm$ the phase shift. In practice, $u$ is well-defined along $\rho$ so $s$ is an integer, but computing it in two ways will give an equality involving $(\mu^c)^{-1}(E)$. First, $v=\Fio^{-1}u$ is a quasimode for $\mu^c\lpar T_{\hbar}\lpar\frac{|z|^2}{\hbar}-1\rpar\rpar$ with value $(\mu^c)^{-1}(E)$ by construction. It is equivalent to $v$ being a  quasimode for $T_{\hbar}(|z|^2)$ with value $\hbar\lpar(\mu^c)^{-1}(E)+1\rpar$. According to Lemma \ref{le_spectrum_normal}, $v(x) = e^{-\frac{|x|^2}{2\hbar}}x^j$ with $j\in\Nm$ and $j=(\mu^c)^{-1}(E) + O(e^{-\frac{c}{\hbar}})$. Then, the phase shift of $v$ along $\sqrt{\frac{E}{d_0}}e^{it}$ is $(\mu^c)^{-1}(E) + O(e^{-\frac{c}{\hbar}})$.

Using the notations of Proposition \ref{prop_action_int}, $\sqrt{\frac{E}{d_0}}(e^{it},e^{-it})$ is the Hamiltonian flow of $\mu_0(d_0 z\overline{v})$ but for a complex time. Hence, $L_E$ is the Hamiltonian flow of $f$ for a complex time too. Since $\Fio$ is associated to the symplectomorphism $\widetilde{\kappa}$, applying the pull-back of $L_E^{-1}$ and $\rho$ on each side of $\Fio$ respectively only multiplies it by a time dependant number. So the phase shift of $v$ along $\sqrt{\frac{E}{d_0}}e^{it}$ is the same as $u$ along $\rho(t)$ up to an integer, meaning $s = l + (\mu^c)^{-1}(E)+ O(e^{-\frac{c}{\hbar}})$ with $l\in\Zm$.

Now, writing $u= e^{-\frac{\Phi_M}{2\hbar}}e^{\frac{\psi}{\hbar}}a$, and according to Proposition \ref{prop_WKB},

\[
T_N(f)\lpar e^{-\frac{\Phi_M}{2\hbar}}e^{\frac{\psi}{\hbar}}a\rpar(x) = e^{-\frac{\Phi_M(x)}{2\hbar}}e^{\frac{\psi(x)}{\hbar}}b(x)
\]
with $b_0(x) = \widetilde{f}_0(x,\overline{v_c}(x))a_0(x)$, and $\overline{v_c}(x)$ such that $\partial\psi(x) = \partial_1\widetilde{\Phi_M}(x,\overline{v_c}(x))$. By hypothesis $b_0 = Ea_0$, so there exists $U$ a neighbourhood of $0$ such that

\[
\Lc_{\psi}\cap\lpar U\times\Cm\rpar = \lacc (z,\overline{v_c}(z)) /\; z\in U\racc \subset \lacc \lpar z,\overline{v}\rpar \middle/\; \widetilde{f_0}\lpar z,\overline{v}\rpar = E\racc
\]
meaning $L_E = \lacc\lpar\rho(t),\overline{v_c}(\rho(t))\rpar /\; t\in[0,1]\racc$. Then, using the notations of Proposition \ref{prop_action_int},

\[
A(E) = -i\int_{L_{E}} \alpha = i\int_{L_{E}} \partial_1\widetilde{\Phi_M} dz = i\int_{\rho} \partial_1\widetilde{\Phi_M}\lpar x,\overline{v_c}(x)\rpar dx = i\int_{\rho} \partial\psi(x) dx
\]
so the phase shift of $e^{\frac{\psi}{\hbar}}$ is $\frac{A(E)}{2\pi\hbar}$. Now, since $\widetilde{f_0}$ is constant on $\Lc_{\psi}$, its Hamiltonian vector field $\widetilde{X_{f_0}}$ is tangent to it, so Proposition \ref{prop_WKB} gives

\begin{align*}
b_1
	= & \lpar \lpar\widetilde{f}a\rpar_1 + i\widetilde{X_{f_0}}a_0\rpar \lpar x,\overline{v_c}(x)\rpar\\
	& + a_0(x) \lpar \frac{1}{2}\widetilde{\Delta f_0} + \frac{d^{-1}}{2}\partial_2\log\widetilde{B_0}\partial_1\widetilde{f_0} - d^{-1}\partial_1\log\widetilde{B_0}\partial_2\widetilde{f_0} + \frac{i}{2}\dive_{\Lc_{\psi}} \lpar\widetilde{X_{f_0}}\rpar \rpar \lpar x,\overline{v_c}(x)\rpar.
\end{align*}
Then, $b_1 = Ea_1$ if and only if $a_0$ satisfies the following Transport equation

\begin{multline*}
-i\widetilde{X_{f_0}} a_0 \lpar x,\overline{v_c}(x)\rpar =\\
	\lpar \widetilde{f_1} +\frac{1}{2}\widetilde{\Delta f_0} + \frac{d^{-1}}{2}\partial_2\log B_0 \partial_1\widetilde{f_0} - d^{-1}\partial_1\log B_0  \partial_2\widetilde{f_0} + \frac{i}{2}\dive_{\Lc_{\psi}} \lpar\widetilde{X_{f_0}}\rpar \rpar \lpar x,\overline{v_c}(x)\rpar a_0(x).
\end{multline*}
Since $L_E = \lacc\lpar\rho(t),\overline{v_c}(\rho(t))\rpar /\; t\in[0,1]\racc$ is the Hamiltonian flow of $\widetilde{f_0}$ for a complex time, the previous equation is equivalent to

\begin{multline*}
	-i\partial_t \lpar a_0(L_E(t))\rpar\\
	= \lbra\lpar \widetilde{f_1} +\frac{1}{2}\widetilde{\Delta f_0} + \frac{d^{-1}}{2}\overline{\partial}\log \widetilde{B_0} \partial \widetilde{f_0} - d^{-1}\partial\log \widetilde{B_0} \overline{\partial} \widetilde{f_0} + \frac{i}{2}\dive_{\Lc_{\psi}} \lpar \widetilde{X_{f_0}}\rpar \rpar a_0 \rbra (L_E(t))\\
	= \biggl[\biggl(\widetilde{f_1} +\frac{1}{2}\widetilde{\Delta f_0} + \frac{d^{-1}}{2}\partial_2\log \widetilde{B_0} \partial_1\widetilde{f_0} - \frac{d^{-1}}{2}\partial_1\log \widetilde{B_0} \partial_2\widetilde{f_0}\\
	+ \frac{i}{2}\partial\lpar-i\partial_1\partial_2\lpar 2\widetilde{\Phi_M}\rpar^{-1}\partial_2\widetilde{f_0}\rpar \biggr) a_0 \biggr] (L_E(t))
\end{multline*}
and the solution is of the form

\begin{multline*}
a_0(x,\overline{v_c}(x)) = \alpha_0 \lpar\partial_1\partial_2\lpar 2\widetilde{\Phi_M}\rpar^{-1}\partial_2\widetilde{f_0}(x,\overline{v_c}(x))\rpar^{-\frac{1}{2}} \times\\
	\exp\lpar i\int_{-t}^0 \lpar \widetilde{f_1} + \frac{1}{2}\widetilde{\Delta f_0} + \frac{d^{-1}}{2}\partial_2\log  B_0  \partial_1 \widetilde{f_0} - \frac{d^{-1}}{2}\partial_1\log  B_0  \partial_2 \widetilde{f_0}\rpar (L_E(s)) ds\rpar
\end{multline*}
with $\alpha_0\in\Cm$ fixed. Hence, the phase shift of $a_0$ is

\[
\int_{L_E} \lpar \widetilde{f_1} + \frac{1}{2}\widetilde{\Delta f_0} + \frac{d^{-1}}{2}\partial_2\log\widetilde{B_0}  \partial_1 \widetilde{f_0} - \frac{d^{-1}}{2}\partial_1\log \widetilde{B_0}  \partial_2 \widetilde{f_0}\rpar \frac{\beta}{2\pi} -\frac{k}{2}
\]
with $k\in\Zm$. Now, the integral

\[
-\int_{L_E} \frac{d^{-1}}{2}\lpar\partial_2\log  \widetilde{B_0}  \partial_1\widetilde{f_0} - \partial_2\widetilde{f_0} \partial_1 \log  \widetilde{B_0} \rpar \frac{\beta}{2\pi} = I_{\sub}(\gamma_E)
\]
is the generator of the Bohr-Sommerfeld class of $\Lc_{\psi}$
relative to the half-bundle $\delta$. Indeed, the covariant derivative of the trivialising section $t$ of $\delta$ with respect to $\widetilde{X_{f_0}}$ restricted to $\Lc_{\psi}$ is

\[
-\frac{d^{-1}}{2}\lpar\partial_2\log \widetilde{B_0}  \partial_1 \widetilde{f_0} - \partial_1\log \widetilde{B_0}  \partial_2 \widetilde{f_0}\rpar\lpar x,\overline{v_c}(x)\rpar t(x),
\]
since $B_0=\partial_1\partial_2\widetilde{\Phi_M}$. Finally, combining the phase shifts gives

\[
(\mu^c)^{-1}(E) = l + \frac{A(E)}{2\pi\hbar} + \int_{L_E} \lpar \widetilde{f_1}+\frac{1}{2}\widetilde{\Delta f_0} \rpar\frac{\beta}{2\pi} - I_{\sub}\lpar L_E\rpar -\frac{k}{2} + O(\hbar),
\]
with $k,l\in\Zm$. If $f$ is quadratic, the formula from Section \ref{subsec_quad_fio} is the same with $k=1$, $l=0$, thus $k$ is equal to $1$ and $l$ to $0$ for any $f$ by a perturbation argument, which proves the result.

Let $E$ be close but distinct to $f(x_0)$, and $\gamma: [0,1] \rightarrow \lacc (z,\overline{v})\in\Cm^d /\; \widetilde{f_0}(z,\overline{v})=E\racc$ a closed curve. Then, there exists a closed curve $\rho$ on $\Cm$ such that

\[
\gamma = \widetilde{\kappa_0}^{-1}\lacc \lpar \rho(t),\frac{\xi}{\rho(t)} \rpar /\; t\in [0,1] \racc  \subset \widetilde{\Cm}
\]
with $\xi = \frac{\mu_0^{-1}(E)}{d_0}$, and $\widetilde{\kappa_0}$ is the symplectomorphism such that $\widetilde{f_0}\circ\widetilde{\kappa_0}^{-1}(z,\overline{v}) = \mu_0(d_0z\overline{v})$. $\widetilde{\kappa_0}$ is simply given by the principal term of the symplectormophism $\widetilde{\kappa}$ such that

\[
\widetilde{f}\circ\widetilde{\kappa}(z,\overline{v}) = \mu\lpar\frac{\sqrt{\det(\alpha\hess_{x_0}f)}}{\alpha}z\overline{v}\rpar.
\]
Consider a flat section $S$ of the prequantum line bundle along $\gamma$, then

\begin{gather*}
\lpar \partial S - \partial\widetilde{\Phi_M} S \rpar (\gamma') = 0,\\
\lpar \overline{\partial} S + \overline{\partial}\widetilde{\Phi_M} S \rpar (\gamma') = 0.
\end{gather*}
Thus, writing $\widetilde{\kappa_0} = (\widetilde{\kappa_1},\widetilde{\kappa_2})$,

\begin{align*}
\partial_t \lpar S(\gamma_E(t))\rpar
	= & \lpar \dot{\rho}(t) \partial - \xi \frac{\dot{\rho}(t)}{\rho^2(t)} \overline{\partial} \rpar \widetilde{\kappa_1} \partial  S(\gamma_E(t)) + \lpar \dot{\rho}(t) \partial - \xi \frac{\dot{\rho}(t)}{\rho^2(t)} \overline{\partial} \rpar \widetilde{\kappa_2} \overline{\partial}  S(\gamma_E(t))\\
	= & \lbra \lpar \dot{\rho}(t) \partial - \xi \frac{\dot{\rho}(t)}{\rho^2(t)} \overline{\partial} \rpar \widetilde{\kappa_1} \partial \widetilde{\Phi_M} - \lpar \dot{\rho}(t) \partial - \xi \frac{\dot{\rho}(t)}{\rho^2(t)} \overline{\partial} \rpar \widetilde{\kappa_2} \overline{\partial} \widetilde{\Phi_M} \rbra  S(\gamma_E(t)).
\end{align*}
Hence, $S(z,\overline{v}) = e^{F(z,\overline{v})} S_0$, and the action is real if and only if $F(\gamma_E(2\pi))-F(\gamma_E(0)) \in i\Rm$, with

\[
F(\gamma_E(t)) = \sum_j a_j (\rho(t))^j + \lpar\int_0^t \frac{\dot{\rho}}{\rho}(s)ds\rpar 2\xi \partial\overline{\partial}\widetilde{\Phi_M}(0,0) \det\nabla\widetilde{\kappa} (0,0),
\]
for some $a_j\in\Cm$, writing the functions as power series. Thus, the condition is equivalent to

\[
2\xi \wind_0(\rho) \partial_1\partial_2\widetilde{\Phi_M}(0,0) \det\nabla\widetilde{\kappa} (0,0) \in \Rm.
\]
 By definition, $\wind_0(\rho)$, $\partial_1\partial_2\widetilde{\Phi_M}(0,0)$ are real, and $\det\nabla\widetilde{\kappa} (0,0) = 1$, thus, the energy level $\lacc f_0= E\racc$ is totally real if and only if $\xi\in\Rm \Leftrightarrow E \in \mu_0(d_0\Rm)$.
\end{proof}

Hence, at order $0$, the analytic pseudo-spectrum is close to the lines $\mu_0(\Rm_+)$, with $\mu_0(\Rm_+)=f_0(\Cm)$ if $f_0(\Cm)$ has an empty interior, and else $\mu_0(\Rm_+)$ lies in the interior of $f_0(\Rm)$. Moreover, $\mu_0(\Rm)$ corresponds exactly to the numbers $E$ such that $\lacc f_0=E\racc$ is a totally real Lagrangian.

\section{Multi-well symbol}
\label{sec_global}

\subsection{Normal forms}

In this Section we consider more general hypothesis. We still consider $M$ to be either a compact Kähler manifold or $\Cm$, and $\Hc_N$ the corresponding quantum space. However, we no longer suppose that the energy level of the symbol is reduced to a unique critical point. If $M$ is a compact manifold, $f$ has only a finite number of critical points, but for $M=\Cm$ we will impose this condition. We make the following hypothesis

\begin{hypothesis}
\label{hyp_global}
~
\begin{itemize}
\item $\mathfrak{e}\in\Cm$ and there exists exactly $k\in\Nm^*$ points $x_1,\dots,x_k\in M$ such that $f(x_n)=\mathfrak{e}$, $1\le n\le k$.
\item For $1\le n\le k$, $df(x_n) = 0$ and the quadratic form $H_n = \hess_{x_n}(f_0)$ is elliptic and satisfies $H_n(T_{x_n}M) \neq \Cm$.
\item If $M$ is a compact manifold: $f:M \rightarrow \Cm$ is a real-analytic, complex valued function, and it is asymptotically equivalent to an analytic symbol on neighbourhoods of each $x_n$. 
\item If $M=\Cm$: There exists an order function $m$ such that $f\in S(\Cm,m)$, and $f$ is elliptic away from every $x_n$, in the sense that there exists $\rho,C,R>0$ such that $\lver \widetilde{f}(x,\overline{y})\rver \ge C m(x)$ for all $(x,\overline{y}) \in \diag + \widetilde{B_{\rho}}$ with $\min\limits_{1\le n \le k} |x-x_n| \ge R$. Moreover, $\bigcup\limits_{1\le n\le k} H_n(\Rm^2) \neq \Cm$.
\end{itemize}

\end{hypothesis}
According to Section \ref{sec_local}, for each $1\le n\le k$, there exists bounded FIOs $\Fio_n :\barg(B_{\rho_n})\rightarrow \statesp[\Wc_{n,i}](U_n)$ , and $\Fio_{n,i} :\statesp[\Wc_n](U_n)\rightarrow\barg(B_{\rho_n})$, where $U_n$ is a neighbourhood of $x_n$, and $\Wc_n$, $\Wc_{n,i}$ are the same as in Lemma \ref{le_exist_weights}, but $W,\,W_i$ here are built from Theorem \ref{th_FIO_normal} with $H = \hess_{x_n}f$ for each $1\le n\le k$. They satisfy

\begin{align*}
T_N(f)\Fio_n = & \Fio_n \mu_n\lpar T_{\hbar}(|z|^2) \rpar + O_{\barg(B_{\rho_n})\rightarrow\statesp[\Wc_n](U_n)}\lpar e^{-\frac{c}{\hbar}}\rpar,\\
\Fio_n\Fio_{n,i} = & id + O_{\statesp[\Wc_n](U_n)\rightarrow\statesp[\Wc_n](U_n)}(e^{-\frac{c}{\hbar}}).
\end{align*}
We follow the strategy and notations of \cite{dele25} section 7.

\begin{proposition}
\label{prop_shape_eigen}
For $c>0$, there exists $c'>0$ and a neighbourhood $\Ec$ of $\mathfrak{e}$ in $\Cm$ such that for every $\lambda\in\Ec$ and $u\in\barg(M)$, if $\lnor u\rnor_{\barg} = 1$ and $(T_N(f)-\lambda)u=O_{\barg}(e^{-\frac{c}{\hbar}})$, then there exist $1\le n\le k$ and $j\in\Nm^*$ such that $\lambda = \mu_n(j\hbar)+O(e^{-\frac{c'}{\hbar}})$.
\end{proposition}

\begin{proof}
According to Hypothesis \ref{hyp_global} and Proposition \ref{prop_concentr} or Corollary \ref{cor_ellip_Barg}, $u$ concentrates on $\lacc f=\lambda \racc$, which is a subset of $\cup_{1\le n\le k} U_n$ for $\lambda$ close enough to $\mathfrak{e}$ and $\hbar$ small enough. Hence, if $\Ec$ is small enough, for any such $u$,

\[
u= \sum\limits_{1\le n\le k} u_n + O_{\statesp[\Wc,m]}(e^{-\frac{c}{\hbar}})
\]
where $u_n = \Pi_N(\chi_{U_n}u)$ for $1\le n\le k$. The open sets $(U_n)_{1\le n\le k}$ are at positive distance from each other, and $T_N(f)-\lambda$ is a local operator, so $(T_N(f)-\lambda)u_n = O_{\barg}(e^{-\frac{c}{\hbar}})$ for each $n$. Since $u$ is normalised, there exists $1\le n\le k$ such that $\lnor u_n\rnor_{\statesp[\Wc](U_n)} > \frac{1}{2N}$. Then, the same arguments as in the proof of Theorem \ref{th_elliptic} applies, there exists $c'>0$ and $j\in\Nm^*$ such that $\lambda = \mu_n(j\hbar)+O(e^{-\frac{c'}{\hbar}})$.
\end{proof}

\begin{proposition}
\label{prop_normal_resolvent}
For $c>0$, there exists $c'>0$ and a neighbourhood $\Ec$ of $\mathfrak{e}$ in $\Cm$ such that for every $\lambda\in\Ec$, if $\lver\lambda - \mu_n(l\hbar)\rver > e^{-\frac{c'}{\hbar}}$ for every $1\le n\le k$ and $l\in\Nm^*$, then

\[
(T_N(f)-\lambda)^{-1}\chi_{U_n} = \Fio_{n,i} (\mu_n(T_{\hbar}(H_n))-\lambda)^{-1} \Fio_n\chi_{U_n} + O_{\statesp[\Wc_n]\to\statesp[\Wc_{n,i}]}(e^{-\frac{c}{\hbar}}),
\]
for all $1\le n\le k$.
\end{proposition}

\begin{proof}
We follow the strategy of proof of \cite{dele25} Proposition 7.3. Fix $1\le n\le k$, and denote

\begin{gather*}
P = (T_N(f)-\lambda),\\
Q = \Fio_n (\mu_n(T_{\hbar}(H_n))-\lambda) \Fio_{n,i}\chi_{U_n},\\
Q_i = \Fio_n (\mu_n(T_{\hbar}(H_n))-\lambda)^{-1} \Fio_{n,i}\chi_{U_n}.
\end{gather*}
Let $c>0$, since $\Fio_{n,i} (\mu_n(T_{\hbar}(H_n))-\lambda) \Fio_n$ is a local operator, there exists $U_n\subset W_n$ such that

\[
(1-\chi_{W_n})Q_i = O_{\statesp[\Wc_n]\to\statesp[\Wc_{n,i}]}(e^{-\frac{c}{\hbar}}),
\]
and for $U_n,W_n$ small enough,

\[(P-Q)\chi_{W_n} = O_{\statesp[\Wc_n]\to\statesp[\Wc_{n,i}]}(e^{-\frac{c}{\hbar}}).
\]
By construction,
\begin{align*}
Q Q_i \chi_{U_n} = & \chi_{U_n} + O_{\statesp[\Wc_n]\to\statesp[\Wc_{n,i}]}(e^{-\frac{c}{\hbar}}),\\
Q_i Q \chi_{U_n} = & \chi_{U_n} + O_{\statesp[\Wc_n]\to\statesp[\Wc_{n,i}]}(e^{-\frac{c}{\hbar}}),
\end{align*}
and $P$,$P^{-1}$,$Q$ and $Q_i$ are bounded with norms of size $O(e^{\frac{\epsilon}{\hbar}})$, where $\epsilon$ can be arbitrarily small up to restricting $\Ec$ and $U_n$ according to Proposition \ref{prop_shape_eigen} and using the same arguments as in the proof of Theorem \ref{th_elliptic}. Combining these equalities

\begin{align*}
	& (P^{-1} - Q_i)\chi_{U_n}\\
	= & P^{-1}(1-PQ_i)\chi_{U_n}\\
	= & P^{-1}(1-P\chi_{W_n}Q_i)\chi_{U_n} + O_{\statesp[\Wc_n]\to\statesp[\Wc_{n,i}]}(e^{\frac{2\epsilon-c}{\hbar}})\\
	= & P^{-1}(1-Q\chi_{W_n}Q_i)\chi_{U_n} + O_{\statesp[\Wc_n]\to\statesp[\Wc_{n,i}]}(e^{\frac{2\epsilon-c}{\hbar}})\\
	= & P^{-1}(1-QQ_i)\chi_{U_n} + O_{\statesp[\Wc_n]\to\statesp[\Wc_{n,i}]}(e^{\frac{2\epsilon-c}{\hbar}})\\
	= & O_{\statesp[\Wc_n]\to\statesp[\Wc_{n,i}]}(e^{\frac{2\epsilon-c}{\hbar}}),
\end{align*}
which gives the result with $c'=c-2\epsilon>0$ for $\epsilon$ small enough.
\end{proof}

\begin{theorem}
\label{th_mutliwell}
Assume Hypothesis \ref{hyp_global} are satisfied, and for all $1\le n\le k$, denote $\alpha_n\in\Sm$ a complex number such that $\alpha_n H_n$ is positive definite. Then, there exists $C,A,r>0$ and real-analytic functions $\mu_{j,n}$ such that $|\mu_{j,n}(\xi)|\le Cj!A^j$ for $\xi$ in a fixed neighbourhood of $0$, and such that the spectrum of $T_N(f)$ satisfies

\begin{gather*}
\sigma(T_N(f)) \cap B_r(\mathfrak{e}) = \bigcup_{1\le n\le k} \lacc \lambda_{n,l} \middle/\; l\in\Nm^*\racc \bigcap B_r(\mathfrak{e}),\\
\lambda_{n,l} = \mu_n^c(l) + O(e^{-\frac{c}{\hbar}}),\\
\mu^c_n(\xi) = \sum\limits_{j=0}^{\frac{1}{A\hbar}} \hbar^j\mu_{j,n}\lpar\hbar\lpar\frac{\sqrt{\det (\alpha_nH_n)}}{\alpha_n}\lpar\xi+\frac{1}{2}\rpar + \frac{\partial\overline{\partial}H_n}{2}\rpar\rpar,
\end{gather*}
for all $n,l\in\Nm^2$ and for $\hbar$ small enough. They are counted with multiplicity, as for each $l\in\Nm$, some $\lambda_{n,l}$ can be equal with each other, which corresponds to an exponentially small Jordan block of size at most $k$, which means a matrix of the form

\[
\begin{pmatrix}
\lambda_{n,l} & O\lpar e^{-\frac{c}{\hbar}}\rpar & & (0) & \\
 & \ddots & \ddots & & & \\
 & (0) & \ddots & O\lpar e^{-\frac{c}{\hbar}}\rpar &\\
 & & & \lambda_{n,l} & 
\end{pmatrix}.
\]
Finally, there exists $\rho,C>0$ with $B_r\subset f(B_{\rho})$, such that for any $\lambda\in B_r(\mathfrak{e})\cap\sigma(T_N(f))^c$ and $1\le n\le k$,

\[
\lnor (T_N(f)-\lambda)^{-1} \rnor_{\statesp[\Wc_n](U_n)\rightarrow\statesp[\Wc_{n,i}](U_n)} \le \frac{C}{d(\lambda,\sigma(T_N(f))}.
\]
\end{theorem}

\begin{proof}
According to Proposition \ref{prop_shape_eigen}, there exists $c'>0$ such that the analytic $c'$-pseudo spectrum of $T_N(f)$ in a neighbourhood of $\mathfrak{e}$ is of the form $\sigma_{c'}(f) = \bigcup_{1\le n\le k, l\in\Nm^*} U_{l,n}$ where each $U_{l,n}$ is a neighbourhood of $\lambda_{n,l}$ with size $O\lpar e^{-\frac{c}{\hbar}}\rpar$. It corresponds to the yellow balls in Figure \ref{fig_shape_spectrum}. We already know that each connected component contains at least one eigenvalue. Now, let $W$ be a connected component of $\sigma_{c'}(f)$, we check that the number of eigenvalues counted with multiplicity in $W$ is

\[
\lver\lacc (l,n) \in \Nm^*\times[1,\dots,k] /\; \lambda_{n,l} \in W \racc\rver = \lver \Nc(W)\rver.
\]
We denote

\[
\Pi_{W} = \frac{1}{2i\pi} \int_{\gamma}\lpar T_N(f)-\lambda\rpar^{-1}d\lambda
\]
where $\gamma$ is a closed curve in $\Cm\backslash\sigma_{c'}(f)$ around $W$. For $\lambda$ in a neighbourhood of $\mathfrak{e}$, $f-\lambda$ does not vanish on $M\backslash\lpar U_1\cup\cdots\cup U_k\rpar$, so using the symbolic calculus from Section \ref{subsec_symbol_calcul} if $M$ is compact, and the results from \cite{hitr24} with Proposition \ref{prop_link_weyl} if $M=\Cm$, there exists a symbol $r_{\lambda}(f)$ holomorphic with respect to $\lambda$, such that

\[
T_{\hbar}(r_{\lambda}(f))\chi_{\Cm/\lpar U_1\cup\cdots\cup U_k\rpar} = (T_N(f)-\lambda)^{-1}\chi_{\Cm/\lpar U_1\cup\cdots\cup U_k\rpar} + O(e^{-\frac{c}{\hbar}}),
\]
and then $\Pi_W \chi_{\Cm/\lpar U_1\cup\cdots\cup U_k\rpar} = O(e^{-\frac{c}{\hbar}})$.

Now, let $1\le n\le k$, since $\gamma$ is in $\Cm\backslash\sigma_{c'}(f)$ and according to Proposition \ref{prop_normal_resolvent}

\[
\Pi_W \chi_{U_n} = \Fio_{n,i} \frac{1}{2i\pi}\int_{\gamma} (\mu_n(T_{\hbar}(H_n))-\lambda)^{-1} d\lambda \Fio_n \chi_{U_n} +  O_{\statesp[W_n+\epsilon_n]\to\statesp[W_{n,i}+\epsilon_{n,i}]}(e^{-\frac{c}{\hbar}}).
\]
If there is no element of the form  $\mu_n^c(l) + O(e^{-\frac{c}{\hbar}})$ in $W$, then $\int_{\gamma} (\mu_n(T_{\hbar}(H_n))-\lambda)^{-1} d\lambda = 0$ and $\Pi_W \chi_{U_n} =  O_{\statesp[W_n+\epsilon_n]\to\statesp[W_{n,i}+\epsilon_{n,i}]}(e^{-\frac{c}{\hbar}})$. Otherwise, there is a unique $l\in\Nm^*$ for which it is satisfied, and

\[
\frac{1}{2i\pi}\int_{\gamma} (\mu_n(T_{\hbar}(H_n))-\lambda)^{-1} d\lambda = \Pi_{\Cm e_l}
\]
with $e_l = e^{-\frac{|z|^2}{2\hbar}} z^{l-1}$. So, denoting $d_{n,l} = \Fio_{n,i}e_l$, we get that

\[
\Pi_W \chi_{U_n} = \Pi_{\Cm d_{n,l}} \chi_{U_n} +  O_{\statesp[\Wc_n]\to\statesp[\Wc_{n,i}]}(e^{-\frac{c}{\hbar}}) = \Pi_{\Cm d_{n,l}} +  O_{\statesp[\Wc_n]\to\statesp[\Wc_{n,i}]}(e^{-\frac{c}{\hbar}})
\]
as $d_{n,l}$ already concentrates on $U_n$ by definition. In the end,

\begin{align*}
\Pi_W
	= & \sum\limits_{n\in\Nc(W)} \Pi_W \chi_{U_n} + O_{\statesp[\Wc_n]\to\statesp[\Wc_{n,i}]}(e^{-\frac{c}{\hbar}})\\
	= & \sum\limits_{n\in\Nc(W)} \Pi_{\Cm d_{n,l}} + O_{\statesp[\Wc_n]\to\statesp[\Wc_{n,i}]}(e^{-\frac{c}{\hbar}}).
\end{align*}
Since the operators are projections, the remainder vanishes if $U_n$ is small enough, and the number of eigenvalues in $W$ is $\rank(\Pi_W) = |\Nc(W)|$. hence, we get that the discrete spectrum is of the desired shape. If $M$ is compact, it completes the proof.

If $M=\Cm$, we adapt the argument of the proof of Theorem \ref{th_elliptic}. In order to adapt the proof, it is necessary to prove the existence of $\theta$ and $\delta,\epsilon>0$ independent of $n$ such that $|H_n(z)+\frac{\epsilon}{2}e^{i\theta}| > \epsilon m(z)$ for all $|z|\le \delta$ and $1\le n\le k$. For each $1\le n\le k$, the inequality is satisfied if $\theta \in -H_n(\Rm^2)^c$ and $\delta,\epsilon$ small enough, and we supposed $\bigcup\limits_{1\le n\le k} H_n(\Rm^2) \neq \Cm$. Hence, $\lver f_0(z)-\mathfrak{e}+\frac{\epsilon}{2}e^{i\theta}\rver > \epsilon m(z)$ for all $z$ such that $|z-x_n|\le \delta$ for all $n\in\{1,\dots,k\}$, and it stays true for all $z\in\Cm$ thanks to the last item of Hypothesis \ref{hyp_global}. Then, the arguments are the same as in Theorem \ref{th_elliptic}.
\end{proof}

\subsection{Description of the spectrum}

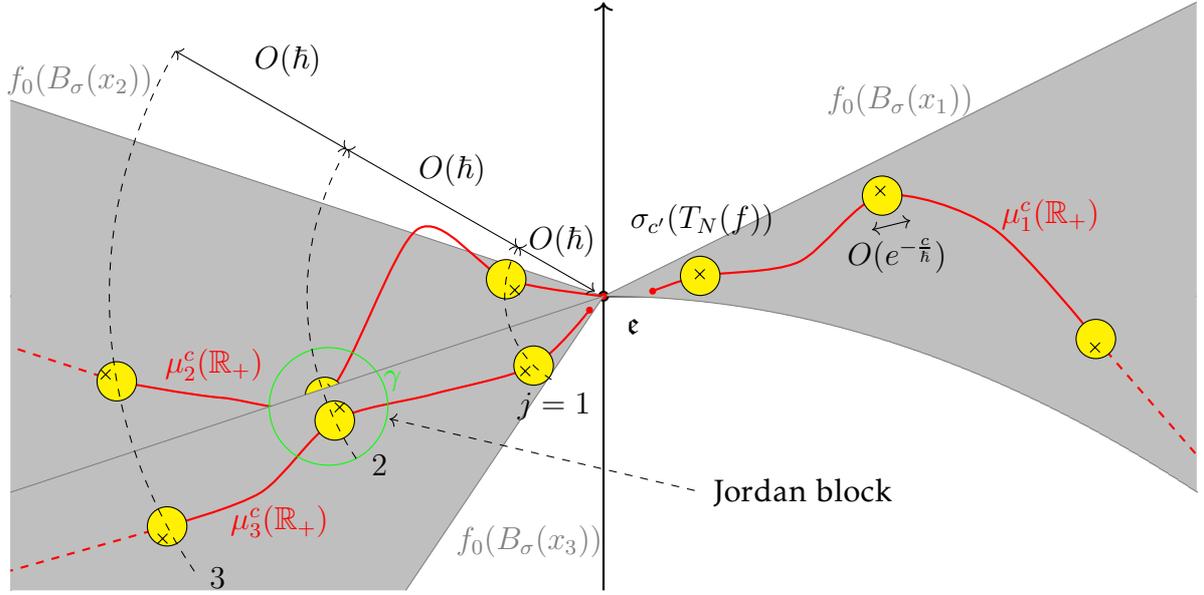
\begin{figure}[h]
\centering
\begin{tikzpicture}[scale = 1.3]
\draw[thick,->] (-6,0)--(6,0);
\draw[thick,->] (0,-3)--(0,3);
\draw[fill=black] (0,0) circle (0.5mm);
\draw (0.3,-0.3) node{$\mathfrak{e}$};
\foreach \x in {0,1,...,60}{
	\node (\x) at (\x/10,-\x^2/1800){};}
\fill[color=lightgray,opacity=0.2]  plot coordinates{(0,0) (6,3) (60)(59)(58)(57)(56)(55)(54)(53)(52)(51)(50)(49)(48)(47)(46)(45)(44)(43)(42)(41)(40)(39)(38)(37)(36)(35)(34)(33)(32)(31)(30)(29)(28)(27)(26)(25)(24)(23)(22)(21)(20)(19)(18)(17)(16)(15)(14)(13)(12)(11)(10)(9)(8)(7)(6)(5)(4)(3)(2)(1)(0,0)};
\draw[gray] (59)--(58)--(57)--(56)--(55)--(54)--(53)--(52)--(51)--(50)--(49)--(48)--(47)--(46)--(45)--(44)--(43)--(42)--(41)--(40)--(39)--(38)--(37)--(36)--(35)--(34)--(33)--(32)--(31)--(30)--(29)--(28)--(27)--(26)--(25)--(24)--(23)--(22)--(21)--(20)--(19)--(18)--(17)--(16)--(15)--(14)--(13)--(12)--(11)--(10)--(9)--(8)--(7)--(6)--(5)--(4)--(3)--(2)--(1)--(0,0)--(6,3);
\draw[gray] (3,2) node{$f_0(B_{\sigma}(x_1))$};
\draw[red,fill=red] (6:0.5) circle(0.3mm);
\draw[red,thick] plot [smooth] coordinates {(6:0.5) (12:1) (10:2) (20:3) (10:4) (-5:5)};
\draw[red] (12:4) node[right]{$\mu^c_1(\Rm_+)$};
\draw[red,thick,dashed] (-5:5) -- (-15:6.2);
\draw[fill=yellow,fill opacity=0.5] (12:1) circle (0.2cm);
\draw (1,0.5) node[above,black,opacity=1]{$\sigma_{c'}(T_N(f))$};
\draw[fill=yellow,fill opacity=0.5] (20:3) circle (0.2cm);
\draw[fill=yellow,fill opacity=0.5] (-5:5) circle (0.2cm);
\draw plot[only marks,mark=x,mark size=2pt] coordinates
 {(13:1) (21:3) (-6:5)};
\draw[black,<->] (14:2.8) -- (14:3.2);
\draw (8:3) node{$O(e^{-\frac{c}{\hbar}})$};

\fill[color=lightgray,opacity=0.2] (0,0) -- (-6,2) -- (-6,-3) -- cycle;
\draw[gray] (-6,2) -- (0,0) -- (-6,-3);
\draw[gray] (-5.3,2.2) node{$f_0(B_{\sigma}(x_2))$};
\draw[red,fill=red] (0,0) circle(0.3mm);
\draw[red,thick] plot [smooth] coordinates {(0,0) (170:1) (160:2) (200:3) (195:4) (190:5)};
\draw[red] (190:4) node{$\mu^c_2(\Rm_+)$};
\draw[red,thick,dashed] (190:5) -- (185:6);
\draw[fill=yellow,fill opacity=0.5] (170:1) circle (0.2cm);
\draw[fill=yellow,fill opacity=0.5] (200:3) circle (0.2cm);
\draw[fill=yellow,fill opacity=0.5] (190:5) circle (0.2cm);
\draw plot[only marks,mark=x,mark size=2pt] coordinates {(176:1-0.1) (199:3-0.1) (189:5+0.1)};

\fill[color=lightgray,opacity=0.2] (0,0) -- (-6,-2) -- (-6,-3) -- (-2,-3) -- cycle;
\draw[gray] (-2,-3) -- (0,0) -- (-6,-2);
\draw[gray] (-0.75,-2.5) node{$f_0(B_{\sigma}(x_3))$};
\draw[red,fill=red] (225:0.2) circle(0.3mm);
\draw[red,thick] plot [smooth] coordinates {(225:0.2) (225:1) (210:2) (205:3) (210:4) (208:5)};
\draw[red] (215:4) node{$\mu^c_3(\Rm_+)$};
\draw[red,thick,dashed] (208:5) -- (-6,-2.8);
\draw[fill=yellow,fill opacity=0.5] (225:1) circle (0.2cm);
\draw[fill=yellow,fill opacity=0.5] (205:3) circle (0.2cm);
\draw[fill=yellow,fill opacity=0.5] (208:5) circle (0.2cm);
\draw plot[only marks,mark=x,mark size=2pt] coordinates {(224:1+0.1) (203:3-0.1) (209:5+0.1)};

\draw[green] (202:3) circle (0.6cm);
\draw[green] (202:2.3) node{$\gamma$};
\draw[dashed,<-] (210:2.5) -- (1,-2) node[right]{Jordan block};

\draw[dashed] (150:1) arc(150:240:1) node[below]{$j=1$};
\draw[dashed] (150:3) arc(150:215:3) node[right]{$2$};
\draw[dashed] (150:5) arc(150:215:5) node[right]{$3$};
\draw[black,<->] (150:0.1) -- (150:1);
\draw (127:0.7) node{$O(\hbar)$};
\draw[black,<->] (150:1) -- (150:3);
\draw (140:2) node{$O(\hbar)$};
\draw[black,<->] (150:3) -- (150:5);
\draw (143:4) node{$O(\hbar)$};
\end{tikzpicture}
\caption{Spectrum of $T_N(f)$: the yellow balls form $\sigma_{c'}(T_N(f))$. For $r\neq s$ a ball at level $j=r$ is disjoint from the balls at level $s$, but on the same level a connected component can consist of multiple balls.}
\label{fig_shape_spectrum}
\end{figure}

The description of the spectrum in Theorem \ref{th_mutliwell} is summarised in Figure \ref{fig_shape_spectrum}. The analytic pseudo-spectrum consists of open sets of size $O(e^{-\frac{c}{\hbar}})$ around each $\mu_n^c(l)$. For a fixed $1\le n\le k$, the elements of $(\mu_n^c(l))_{l\in\Nm^*}$ are separated by a distance $O(\hbar)$, also $\mu_n^c(0)$ is at distance $O(\hbar)$ to $\mathfrak{e}$. But for a fixed $l\in\Nm^*$, two or more balls around the $(\mu_n^c(l))_{1\le n\le k}$ can intersect each other, leading to a Jordan block. Writing $\widetilde{\eta_n}$ the symplectomorphism such that $\widetilde{f}\circ\widetilde{\eta_n} = \widetilde{H_n}$ near $0$, since $\frac{\sqrt{\det (\alpha_nH_n)}}{\alpha_n}\Rm + \hbar\frac{\partial\overline{\partial}H_n}{2} \subset H(\Rm^2)$, we get that for $\sigma$, $\rho$ small enough

\begin{align*}
\lpar\frac{\sqrt{\det (\alpha_nH_n)}}{\alpha_n}\Rm + \hbar\frac{\partial\overline{\partial}H_n}{2}\rpar \cap B_{\rho}
	\subset & \lpar \widetilde{H_n} \circ \widetilde{\eta_n}(\diag)\rpar \cap B_{\rho}\\
	\subset & \widetilde{H_n} \circ \widetilde{\eta_n}\lpar \diag\cap \widetilde{B_{\sigma}}(x_n,\overline{x_n})\rpar.
\end{align*}
Hence, for $r$ small enough

\begin{align*}
\mu_{n,0}\lpar\frac{\sqrt{\det (\alpha_nH_n)}}{\alpha_n}\Rm + \hbar\frac{\partial\overline{\partial}H_n}{2}\rpar \cap B_r
	\subset & \mu_{n,0}\lpar\widetilde{H_n}\rpar \circ \widetilde{\eta_n}\lpar\diag\cap\widetilde{B_{\sigma}}(x_n,\overline{x_n})\rpar\\
	& = \widetilde{f_0}\lpar\diag\cap\widetilde{B_{\sigma}}(x_n,\overline{x_n})\rpar\\
	& = f_0(B_{\sigma}(x_n)),
\end{align*}
so the curve $\mu_n(\Rm_+) \cap B_r$ stays in $f_0(B_{\sigma}(x_n))$ up to a $O(\hbar)$ term. Moreover, for $\sigma$ small enough $f(B_{\sigma}(x_n))\cap B_r(\mathfrak{e})$ is close to $\Sigma(H_n)\cap B_{\sigma}(\mathfrak{e})$, which is a truncated cone.

Like in Theorem \ref{th_elliptic}, it is possible to get expressions for $(\mu^c_{n})_{1\le n\le k}$ with a WKB method, and to link the curves $\mu^c_{n}(\Rm)$ to the Bohr-Sommerfeld index of the corresponding level sets of $f$.

\begin{proposition}
We keep the notations and conclusions of Theorem \ref{th_mutliwell}. For all $1\le n\le k$, consider the function $A_n(E) = -i\int_{\gamma_E}\alpha$ as in \eqref{eq_action_int} but with
\[
\gamma_{E,n} = \widetilde{\kappa_n}^{-1}\lacc x=y= \sqrt{\frac{E}{d_n}} e^{it} /\; t\in [0,2\pi] \racc = \widetilde{\kappa_n}^{-1}\lpar\frac{E}{d_n}\Sm\rpar,
\]
where $\widetilde{\kappa_n}$ is a symplectomorphism such that $\widetilde{f_0}\circ \widetilde{\kappa_n}^{-1}(x,\overline{y}) = d_nx\overline{y}$ near $(x_n,\overline{x_n})$. Let $1\le n\le k$, and $E\neq \mathfrak{e}$ be in a neighbourhood of $\mathfrak{e}$, then

\[
(\mu_n^c)^{-1}(E) = \frac{A_n(E)}{2\pi\hbar} + \int_{\gamma_{E,n}} \lpar \widetilde{f_1}+\frac{1}{2}\widetilde{\Delta f_0} \rpar\frac{\beta}{2\pi} - I_{\sub}(\gamma_{E,n}) -\frac{1}{2} + O(\hbar),
\]
with $\beta$ the $1$-form on $\gamma_{E,n}$ such that $\beta\lpar\widetilde{X_{f_0}}\rpar=1$, and $\widetilde{X_{f_0}}$ the Hamiltonian vector field of $\widetilde{f_0}$.

Since $f$ is an analytic symbol, $\frac{\sqrt{\det(\alpha_n H_n)}}{\alpha_n}$ has an expansion in $\hbar$, we write $d_n$ its principal term. Then, locally near $\mathfrak{e}$, the elements of $\mu_{n,0}(d_n\Rm)$ are the numbers $E$ such that the connected component $P_n(E)$ of $\{f^{-1}(E)\}$ that stays in a neighbourhood of $x_n$ is a totally real Lagrangian. In other words, for any $E$ close to $\mathfrak{e}$ in $\mu_{n,0}(d_n\Rm)$, the Bohr-Sommerfeld class of $P_n(E)$ along the prequantised bundle $L$ is unitary.
\end{proposition}

Theorem \ref{th_mutliwell} consider an energy at which the preimage is only made of elliptic critical points. In all generality, the preimage can also contain hyperbolic critical points, or even regular points. Hence, for a general framework, a mixed study of these different set-ups would be necessary. An idea would be to combine our results with \cite{dele25,dura25a} using the resolvent estimates of each one to glue them together.

\appendix

\section{Contour deformations}

\label{sec_contour}

Computations with FIOs involve changes of contours of integration in order to use the method of stationary phase. These computations are described in \cite{meli75,sjos19}, and relies on Stokes' Theorem applied to suitable families of contours.

In this section, we detail the methods of contour deformation in a general framework, and we give an application for a specific structure. Let $n,d\in\Nm^*$, $U\times V \subset \Cm^n\times\Cm^d$ be neighbourhoods of $0$, and $F$ be a holomorphic function on $U\times V$ such that

\begin{align}
\label{eq_cond_phase}\nonumber
& \forall x\in U,\; \exists! y_0(x)\in V /\; \partial_yF(x,y_0(x))=0,\\
& \det\lpar \hess_y F(x,y_0(x))\rpar \neq 0,\\\nonumber
& y_0(0) = 0
\end{align}

\begin{definition}
\label{def_contour}
Let $U,W$ be open sets of $\Cm^n$ and $\Rm^d$ respectively, we call contour a family $\lpar\Gamma(x)\rpar_{x\in U}$ of real-dimension $d$ sub-manifolds of $\Cm^d$, such that there exists $\gamma$ real analytic such that the function $w\mapsto \gamma(x,w)$ is an embedding for all $x\in U$, and $\Gamma(x) = \lacc \gamma(x,w) / w\in W\racc$. Furthermore, we consider on $\Gamma(x)$ the orientation given by the image of the canonical orientation of $\Rm^d$ by $d_w\gamma(x,w)$.

Let $V$ be an open set of $\Cm^n$ and $F$ be holomorphic on $U\times V$ and satisfying \eqref{eq_cond_phase}. A contour $\lpar\Gamma(x)\rpar_{x\in U}$ is said good for $F$ if there exists $C>0$, such that for all $\rho>0$, there exists $\eta>0$ such that $\forall x \in U\cap B_{\eta}$, $\forall y\in V\cap\Gamma(x)$ with $|y|\ge\rho$

\[
\Re\lpar F(x,y)-F(x,y_0(x))\rpar \le -C| y-y_0(x)|^2.
\]

A homotopy of good contours for $F$ is a family of contours

\[
\Gamma_{t}(x) = \lacc \gamma(t,x,w) /\; w \in\Rm^d \racc
\]
with $\gamma : [0,1]\times \Cm^n \times \Rm^d \rightarrow \Cm^d$ real analytic, such that $\Gamma_{\lambda}$ is good for $F$ for all $0\le t\le 1$ and $\lpar\Gamma_0 \cap B_{\rho}\rpar\bigcup\lpar\Gamma_1 \cap B_{\rho}\rpar\bigcup\lbra\lpar\bigcup\limits_{0\le \lambda\le 1}\Gamma_{\lambda}(x)\rpar \bigcap \partial B_{\rho} \rbra$ is an oriented closed real manifold of dimension $d$ for all $\rho>0$.
\end{definition}

Notice that the definition of good contour could be replaced by

\[
\Re\lpar F(x,y)-F(x,y_0(x))\rpar \le -C| y-y_0(x)|,
\]
but the quadratic formula is more convenient locally, as we will reduce the problem to $F$ quadratic thanks to a Morse Lemma.

This definition of good contour is a variant of the notion used by Sjöstrand, see \cite{sjos82} Chapter 3. We will see that integrals over homotopic good contours are equal, which will simplify computations in the next subsections.

\begin{example}
Let $F$ satisfy \eqref{eq_cond_phase}, by the holomorphic Morse lemma, for $x\in U$ there exists a diffeomorphism $K_x$ such that for $y$ small enough

\[
F(x,y) = F(x,y_0(x)) - \lpar K_x\lpar y-y_0(x)\rpar\rpar^{\intercal}\lpar K_x\lpar y-y_0(x)\rpar\rpar.
\]
This is actually a parametric Morse lemma due to the dependence in $x$, the exact result will be given by Lemma \ref{le_morse}. Now, $\Gamma(x) = K_x^{-1}\Rm+y_0(x)$ is a good contour in a neighbourhood of $y_0(x)$. If the neighbourhood covers all $V$, then $\Gamma(x)$ is a good contour. Moreover, for any good contour $\Gamma(x)$, the contour $\Gamma(x)+b(x)$ is good too for $b(x)$ small enough with $b(0)=0$.
\end{example}

\begin{proposition}
\label{prop_local_contour_defor}
Let $F$ satisfy \eqref{eq_cond_phase}, and $a$ be a holomorphic function on $\Cm^{d+n}$. If $\Gamma_0$, $\Gamma_1$ are homotopic good contours for $F$, then there exists $c>0$ such that for all $\rho>0$ there exists $\eta>0$ and $\rho>\epsilon>0$ such that for all $x\in B_{\eta}$

\[
e^{-\frac{F(x,y_0(x))}{\hbar}} \int_{\Gamma_0(x)\cap B_{\rho}} e^{\frac{F(x,y)}{\hbar}} a(x,y) dy
= e^{-\frac{F(x,y_0(x))}{\hbar}} \int_{\Gamma_1(x)\cap B_{\rho}} e^{\frac{F(x,y)}{\hbar}} a(x,y) dy + O\lpar e^{-\frac{c(\rho-\epsilon)^2}{\hbar}}\rpar.
\]
\end{proposition}

\begin{proof}
By hypothesis, there exists a homotopy of good contours $\lpar\Gamma_t\rpar_{0\le t\le 1}$. We denote

\[
L = \lpar\Gamma_0 \cap B_{\rho}\rpar\bigcup\lpar\Gamma_1 \cap B_{\rho}\rpar\bigcup\lbra\lpar\bigcup\limits_{0\le \lambda\le 1}\Gamma_{\lambda}(x)\rpar \bigcap \partial B_{\rho} \rbra,
\]
then $L$ is an oriented closed piecewise smooth real manifold of dimension $d$. Then, by Stokes' Theorem, for any manifold $\Omega$ of real dimension $n+1$ such that $\partial\Omega = L$,

\[
\int_{\Omega} d\lpar e^{\frac{F(x,y)}{\hbar}}a(x,y)dy\rpar = \int_{\partial\Omega} e^{\frac{F(x,y)}{\hbar}}a(x,y)dy.
\]
But, since $F$ and $a$ are holomorphic

\[
d\lpar e^{\frac{F(x,y)}{\hbar}}a(x,y)dy\rpar = 0.
\]
and the orientation is preserved along $\Gamma_{\lambda}$, so

\[
\int_{\Gamma_0 \cap B_{\rho}} e^{\frac{F(x,y)}{\hbar}}a(x,y)dy - \int_{\Gamma_1 \cap B_{\rho}} e^{\frac{F(x,y)}{\hbar}}a(x,y)dy = \pm\int_{\lpar\cup_{0\le \lambda\le 1}\Gamma_{\lambda}(x)\rpar \bigcap \partial B_{\rho}} e^{\frac{F(x,y)}{\hbar}}a(x,y)dy.
\]
Now, by definition of good contours, there exists $c>0$ such that for all $\rho>0$, there exists $\eta_1>0$ such that for all $x\in B_{\eta_1}$, $|y|\ge \rho$

\[
\Re F(x,y) - \Re F(x,y_0(x)) \le -c|y-y_0(x)|^2.
\]
Furthermore, since $y_0(0)=0$, for all $\epsilon\in ]0,\rho[$ there exists $\eta_2 >0$ such that $|y_0(x)| < \epsilon$ for all $x\in B_{\eta_2}$. Then, taking $\eta = \min(\eta_1,\eta_2)$, for all $x\in B_{\eta}$, $y\in\partial B_{\rho}$

\[
\lver e^{\frac{F(x,y)-F(x,y_0(x))}{\hbar}} \rver \le e^{-C\frac{(\rho-\epsilon)^2}{\hbar}}.
\]
Hence, by triangle inequality, for all $x\in B_{\eta}$

\[
\lver e^{-\frac{F(x,y_0(x))}{\hbar}}\lpar\int_{\Gamma_0 \cap B_{\rho}} e^{\frac{F(x,y)}{\hbar}}a(x,y)dy - \int_{\Gamma_1 \cap B_{\rho}} e^{\frac{F(x,y)}{\hbar}}a(x,y)dy\rpar \rver \le 2\pi\rho \lnor a \rnor e^{-C\frac{(\rho-\epsilon)^2}{\hbar}}.
\]
\end{proof}

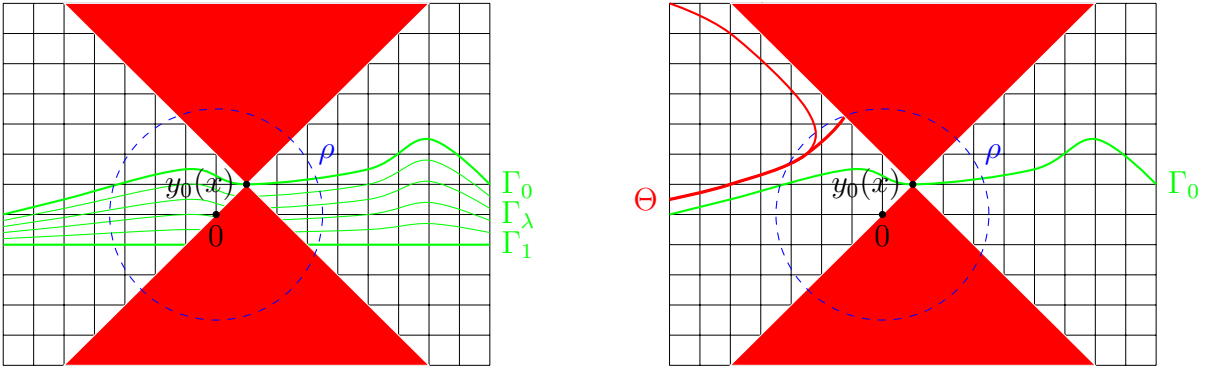
\begin{figure}[h]
\centering
\begin{minipage}[b]{.49\textwidth}
\begin{tikzpicture}[scale=0.4]
\draw[step=1cm,black,very thin,opacity=0.3] (-8,-6) grid (8,6);
\draw[green,thick] plot [smooth, tension=0.6] coordinates {(-8,-1) (-6,-0.5) (-2,0.5) (0,0) (4,0.5) (6,1.5) (8,0)} node[right] {$\Gamma_0$};
\draw[green,opacity=0.5] (8,-1) node[right] {$\Gamma_{\lambda}$};
\foreach \t in {0.2,0.4,...,0.8}{
    \draw[green,opacity=0.5] plot [smooth, tension=0.6] coordinates {(-8,-1-\t) (-6,-0.5-1.5*\t) (-2,0.5-2.5*\t) (0,-2*\t) (4,0.5-2.5*\t) (6,1.5-3.5*\t) (8,-2*\t)};}
\draw[green,thick] (-8,-2) -- (8,-2) node[right] {$\Gamma_1$};
\draw[thin, white, fill=red, fill opacity=0.2] (0,0) -- (-6,6) -- (6,6) -- cycle;
\draw[thin, white, fill=red, fill opacity=0.2] (0,0) -- (-6,-6) -- (6,-6) -- cycle;
\draw[fill=black] (0,0) circle (1mm) node[left] {$y_0(x)$};
\draw[blue,dashed] (0-1,0-1) circle (3.5);
\draw[fill=black] (0-1,0-1) circle (1mm) node[below] {$0$};
\draw[blue] (3-1,2-1) node[above,right]{$\rho$};
\end{tikzpicture}
\end{minipage}
\hfill
\begin{minipage}[b]{.49\textwidth}
\begin{tikzpicture}[scale=0.4]
\draw[step=1cm,black,very thin,opacity=0.3] (-8,-6) grid (8,6);
\draw[green,thick] plot [smooth, tension=0.6] coordinates {(-8,-1) (-6,-0.5) (-2,0.5) (0,0) (4,0.5) (6,1.5) (8,0)} node[right] {$\Gamma_0$};
\draw[red,thick] plot [smooth, tension=0.6] coordinates {(-8,6) (-6,5) (-3.5,2.5) (-3.5,1) (-6,0) (-8,-0.5)} node[left] {$\Theta$};
\draw[red,very thick] plot [smooth, tension=0.6] coordinates {(-5,6) (-4,5) (-2,3) (-3.5,1) (-6,0) (-8,-0.5)};
\draw[thin, white, fill=red, fill opacity=0.2] (0,0) -- (-6,6) -- (6,6) -- cycle;
\draw[thin, white, fill=red, fill opacity=0.2] (0,0) -- (-6,-6) -- (6,-6) -- cycle;
\draw[fill=black] (0,0) circle (1mm) node[left] {$y_0(x)$};
\draw[blue,dashed] (0-1,0-1) circle (3.5);
\draw[fill=black] (0-1,0-1) circle (1mm) node[below] {$0$};
\draw[blue] (3-1,2-1) node[above,right]{$\rho$};
\end{tikzpicture}
\end{minipage}
\caption{Examples of contours, for $F(x,y) = -\lpar y-y_0(x)\rpar^2$ and $y_0(x)\neq 0$.}
\label{fig_contour}
\end{figure}

Figure \ref{fig_contour} shows good contours for $F(x,y) = -\lpar y-y_0(x)\rpar^2$ and $y_0(x)\neq 0$. Here, a contour is good if it stays outside the red part away from $0$. For instance, in the left figure, $\Gamma_t$ is a homotopy of good contours. In the right figure, $\Gamma_0$ and $\Theta$ are both good contours, but they are not homotopic, as any homotopy between them cross the red part.

\begin{definition}[Affine contour]
We say that a contour $\Gamma(x) = \lacc U(x,w)/\; w\in\Rm^{d} \racc$ is affine if for all $x\in\Cm^n$ there exists $A(x)\in Gl_d(\Cm)$ and $b(x)\in\Cm^d$ such that $U(x,w) = A(x)w + b(x)$, for all $w\in\Rm^d$.
\end{definition}

We prove the following lemma in order to give an equivalent condition on $A(x),b(x)$ for the affine contour $\gamma(x) = \lacc A(x)w+b(x) \middle/\; w\in\Rm^d\racc$ to be good for a function $F$.

\begin{lemma}
\label{le_sqrt_matrix}
Let $H\in Gl_d(\Cm)$ be real-symmetric, meaning $H^{\intercal} = H$, then there exists $P\in Gl_d(\Cm)$ symmetric such that $P^{\intercal}P = P^2 = H$.
\end{lemma} 

\begin{proof}
Denote by $(\lambda_j)_{1\le j\le k}$ the eigenvalues of $H$ and $(m_j)_{1\le j\le k}$ their multiplicities, then there exists a polynomial $p$ such that the first $m_j-1$ derivatives of $p$ agree with some branch of $\sqrt{z}$ at $\lambda_j$ for each $j\in \{1,\dots,k\}$. It means that $p^2(X)-X$ vanishes at order $m_j$ at $\lambda_j$ for each $j\in \{1,\dots,k\}$, and according to the Cayley-Hamilton theorem $p(H)^2=H$. Since $H$ is symmetric, $p(H)$ is also symmetric, thus the equality with $P=p(H)$.
\end{proof}

\begin{proposition}
\label{prop_contours_shape}
Let $F$ be a polynomial of degree at most $2$ on $\Cm^{d+n}$ and satisfy \eqref{eq_cond_phase}, notice that the Hessian of $y \mapsto -F(x,y)$ does not depend on $x$ any more, so let  $P\in Gl_d(\Cm)$ be such that $P^{\intercal}P$ is the Hessian. Let $\Gamma(x)=A(x)\Rm^d + b(x)$ be a contour, $\Gamma(x)$ is a good contour for $F$ if and only if, there exists $C<1$, $\eta>0$ such that for all $|x|<\eta$ the real matrices $M(x)$ and $N(x)$ such that $M(x)+iN(x) = PA(x)$ satisfy $\lnor N(x)M(x)^{-1}\rnor \le C$ in operator norm, and $y_0(0) \in \Gamma(0)$. By construction $M$, $N$ are analytic, so it is equivalent to $\lnor N(0)M(0)^{-1}\rnor <1$, and $y_0(0)\in\Gamma(0)$.
\end{proposition}

\begin{proof}
By definition $F(x,y) = F(x,y_c(x)) - \lpar P\lpar y-y_0(x)\rpar\rpar^2$, and recall

\[
\Re (y^2) = \Re\lpar \sum\limits_{1\le j\le d} y_j^2 \rpar = \sum\limits_{1\le j\le d} \Re(y_j)^2 -  \Im(y_j)^2.
\]
Suppose that $\lnor N(0)M(0)^{-1}\rnor \ge 1$, then there exists $t_0\in\Rm^d$ such that $(N(0)t_0)^2\ge(M(0)t_0)^2$, so for $l\in\Rm$,

\begin{align*}
	& \Re\lpar F(0,A(0)l t_0 +b(0)) - F(0,y_0(0)) \rpar\\
	= & -\Re\lpar\lbra \lpar M(0)+iN(0)\rpar l t_0 + P(0)(b(0)-y_0(0)) \rbra^2\rpar\\
	= & l^2 \lpar(N(0)t_0)^2-(M(0)t_0)^2\rpar - l2\Re\Bigl<(M(0)+iN(0))t_0,P(0)(b(0)-y_0(0))\Bigr>\\
	& - \Re\lpar\lbra P(0)(b(0)-y_0(0)) \rbra^2\rpar,\\
\end{align*}
and this quantity is either positive for some $l$, or constant. In both cases, it contradicts the definition of good contour. Now, suppose $\lnor N(0)M(0)^{-1}\rnor <1$, but $y_0(0)\notin\Gamma(0)$, then we can always translate $b(0)$ by elements of $A(x)\Rm^d$ such that $P(0)(b(0)-y_0(0))\neq 0$ is orthogonal to $P(0)A(0)\Rm^d$ in $\Cm^d$. Then, $\Re\lpar\lpar P(0)A(0) t\rpar^2\rpar >0$ for all $t\in\Rm^d$, and the orthogonality implies that $\Re\lpar\lbra P(0)(b(0)-y_0(0)) \rbra^2\rpar<0$, then

\begin{align*}
	& \Re\lpar F(0,A(0)t +b(0)) - F(0,y_0(0)) \rpar\\
	= & -\Re\lpar\lbra P(0)A(0) t + P(0)(b(0)-y_0(0)) \rbra^2\rpar\\
	= & \Re\lpar\lpar P(0)A(0)t\rpar^2\rpar - \Re\lpar\lbra P(0)(b(0)-y_0(0)) \rbra^2\rpar\\
	> & 0.
\end{align*}
It contradicts the definition of good contour. By contraposition, we proved that any good contour satisfies $\lnor N(0)M(0)^{-1}\rnor <1$ and $y_0(0)\in\Gamma(0)$.

Now, suppose that $M$, $N$ satisfy $\lnor N(0)M(0)^{-1} \rnor <1$, and $y_0(0)\in \Gamma(0)$, then since $b$ is analytic, we can write $b(x)=y_0(x)+d(x)$ with $|d(x)|\to 0$ as $|x|\to 0$. We fix

\[
0<C<\frac{1-\lnor N(0)M(0)^{-1} \rnor^2}{8\lpar 1+\lnor N(0)M(0)^{-1} \rnor^2\rpar} \lnor P(0)\rnor^{-2}.
\]
Let $\rho>0$, then for all $t\in\Rm^d$ such that $|A(x)t+b(x)|\ge \rho$,

\begin{align*}
	& \Re\lpar F(x,A(x)t+b(x)) - F(x,y_0(x)) \rpar\\
	= & \Re\lpar(N(x)t)^2-(M(x)t)^2\rpar - 2\Re\langle(M(x)+iN(x))t,P(x)d(x)\rangle\\
	& -\Re\lpar\lbra P(x)d(x) \rbra^2\rpar\\
	\le & -\lpar 1-\lnor N(x)M(x)^{-1} \rnor^2\rpar\lver M(x)t\rver^2\\
	& + 2\lver (M(x)+iN(x))t\rver|P(x)d(x)| -\Re\lpar\lbra P(x)d(x) \rbra^2\rpar\\
	\le & -\frac{1-\lnor N(x)M(x)^{-1} \rnor^2}{1+\lnor N(x)M(x)^{-1} \rnor^2} \lnor P(x)\rnor^{-2} \lver A(x)t\rver^2\\
	& + 2\lnor P(x)\rnor^{-1}\lver A(x)t\rver|P(x)d(x)| -\Re\lpar\lbra P(x)d(x) \rbra^2\rpar.
\end{align*}
Then, since $|A(x)t+b(x)|=|A(x)t+y_0(x)+d(x)|\ge\rho$, and $|y_0(x)|\to 0$, $|d(x)|\to 0$ when $|x|\to 0$, there exists $\eta$ such that for all $|x|\le \eta$,

\begin{itemize}
\item $C\le \frac{1-\lnor N(x)M(x)^{-1} \rnor^2}{8\lpar 1+\lnor N(x)M(x)^{-1} \rnor^2\rpar}\lnor P(x)\rnor^{-2}$,
\item $|A(x)t| \ge \frac{\rho}{2}$,
\item $|d(x)|\le |A(x)t|$,
\item $4C|y|^2 \ge 2\lnor P(x)\rnor^{-1}|P(x)d(x)||y| -\Re\lpar\lbra P(x)d(x) \rbra^2\rpar$ for all $|y|\ge\frac{\rho}{2}$.
\end{itemize}
In the end, we get

\begin{align*}
	& \Re\lpar F(x,A(x)t+b(x)) - F(x,y_0(x)) \rpar\\
	\le & -8C |A(x)t|^2 + 2\lnor P(x)\rnor^{-1}|P(x)d(x)||y| -\Re\lpar\lbra P(x)d(x) \rbra^2\rpar\\
	\le & -4C |A(x)t|^2\\
	\le & -C\lver A(x)t+d(x)\rver^2\\
	& = -C\lver A(x)t+b(x)-y_0(x)\rver^2
\end{align*}
so $\Gamma(x)$ is a good contour.
\end{proof}

\begin{proposition}
\label{prop_linear_contours}
Let $F$ be quadratic on $\Cm^{d+n}$ and satisfy \eqref{eq_cond_phase}, then all affine good contours are homotopic with each other up to a change of orientation. More precisely, let $a$ be holomorphic on $\Cm^{d+n}$, $\Gamma_0(x)=A_0(x)\Rm^d + b_0(x)$ and $\Gamma_1(x)=A_1(x)\Rm^d+b_1(x)$ with $A_0(x),A_1(x)\in Gl_d(\Cm)$. Denote $d_0(x)$ and $d_1(x)$ the signs of $\det (\Re (PA_0(x)))$ and $\det (\Re (PA_1(x)))$ respectively, where $P(x)\in Gl_d(\Cm)$ is such that $P(x)^{\intercal}P(x)$ is the Hessian of $y \mapsto -F(x,y)$. Then, there exists $c>0$ such that for all $\rho>0$ there exists $\eta>0$ such that for all $|x|<\eta$,
\[
d_0 \int_{\Gamma_0\cap B_{\rho}} e^{F(x,y)-F(x,y_0(x))} a(x,y) dy = d_1\int_{\Gamma_1\cap B_{\rho}} e^{F(x,y)-F(x,y_0(x))} a(x,y) dy + O(e^{-\frac{c}{\hbar}}).
\]
\end{proposition}

\begin{proof}
We consider only one affine good contour $\Gamma(x) = A(x)\Rm^d + b(x)$, and we will prove that it is homotopic to $P(x)^{-1}\Rm^d+y_0(x)$. According to Proposition \ref{prop_contours_shape}, $A(x) = P(x)^{-1}(M(x)+iN(x))$ where $\lnor N(0)M(0)^{-1}\rnor <1$ with operator norm, and $y_0(0) \in \Gamma(0)$. The homotopy can then be built in four parts,

\begin{itemize}
\item[1] $A(x)\Rm^d + (1-t)b(x)+\lambda y_0(x)$.
\item[2] $P(x)^{-1}(M(x)+i(1-t)N(x))\Rm^d+y_0(x)$.
\item[3] If $M(x)$ and $I_d$ are not in the same connected component of $Gl_d(\Rm)$, multiply by the diagonal matrix with coefficients $(-1,1,\dots,1)$, else do nothing.
\item[4] $P^{-1}M_t(x)\Rm^d+y_0(x)$ where $M_t(x)$ is the homotopy from $M(x)$ to $\id_d$ in $Gl_d(\Rm)$.
\end{itemize}
Step 1, 2 and 4 are good according to Proposition \ref{prop_contours_shape}. Hence, we get the equality by applying Proposition \ref{prop_local_contour_defor} to the steps 1,2 and 4.  Step 3 is done by the change of variables that multiplies one variable by $\frac{d_1}{d_0}$, which changes the sign of the integral accordingly.

Figure \ref{fig_change_contour} shows steps 1 and 2 in dimension $1$. Step 3 will simply change the direction of integration, and step 4 is only necessary in larger dimensions.

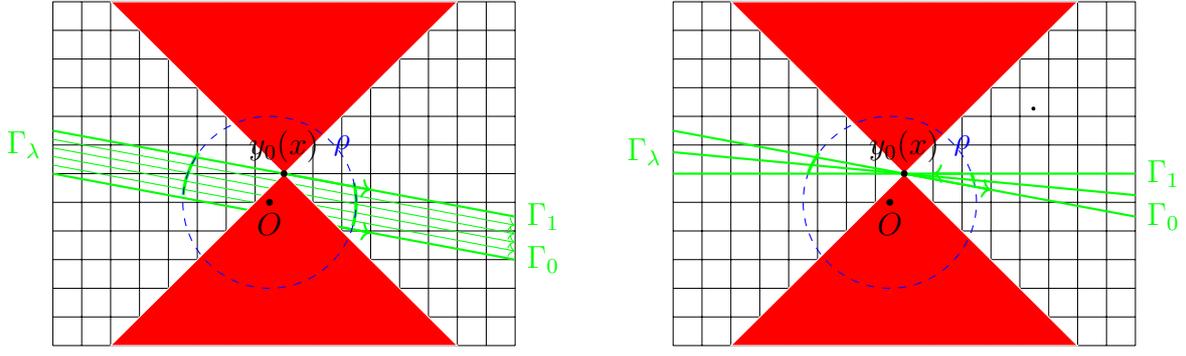
\begin{figure}[t]
\centering
\begin{minipage}[b]{.49\textwidth}
\begin{tikzpicture}[scale=0.38]
\draw[step=1cm,black,very thin,opacity=0.3] (-8,-6) grid (8,6);
\draw[green,thick] (-8,0) -- (8,-3) node[right] {$\Gamma_0$};
\foreach \t in {0.2,0.4,...,0.8}{
    \draw[green,->,opacity=0.5] (-8,1.5*\t) -- (8,-3+1.5*\t);}
\draw[green,thick,opacity=0.5] (-8,1) node[left] {$\Gamma_{\lambda}$};
\draw[green,thick] (-8,1.5) -- (8,-1.5) node[right] {$\Gamma_1$};
\draw[green,opacity=0.5,very thick] (0-0.5,0-1) +(13:3) arc(13:-16:3);
\draw[green,opacity=0.5,very thick] (0-0.5,0-1) +(145:3) arc(145:175:3);
\draw[green,thick,->] (0,0) -- (3,-9/16);
\draw[green,thick,->] (0,-1.5) -- (3,-9/16-1.5);
\draw[thin, white, fill=red, fill opacity=0.2] (0,0) -- (-6,6) -- (6,6) -- cycle;
\draw[thin, white, fill=red, fill opacity=0.2] (0,0) -- (-6,-6) -- (6,-6) -- cycle;
\draw[fill=black] (0,0) circle (1mm) node[above] {$y_0(x)$};
\draw[blue,dashed] (0-0.5,0-1) circle (3);
\draw[blue] (3-1,2-1) node{$\rho$};
\draw[fill=black] (0-0.5,0-1) circle (1mm) node[below] {$O$};
\end{tikzpicture}
\end{minipage}
\hfill
\begin{minipage}[b]{.49\textwidth}
\begin{tikzpicture}[scale=0.38]
\draw[step=1cm,black,very thin,opacity=0.3] (-8,-6) grid (8,6);
\draw[thin, white, fill=red, fill opacity=0.2] (0,0) -- (-6,6) -- (6,6) -- cycle;
\draw[thin, white, fill=red, fill opacity=0.2] (0,0) -- (-6,-6) -- (6,-6) -- cycle;
\draw[blue,dashed] (0-0.5,0-1) circle (3);
\draw[blue] (3-1,2-1) node{$\rho$};
\draw[green,thick] (-8,0) -- (8,0) node[right] {$\Gamma_1$};
\draw[green,thick,->] (2,0) -- (1,0);
\draw[green,thick,opacity=0.5] (8,-0.75) -- (-8,0.75) node[left] {$\Gamma_{\lambda}$};
\draw[green,thick] (-8,1.5) -- (8,-1.5) node[right] {$\Gamma_0$};
\draw[green,thick,->] (0,0) -- (3,-9/16);
\draw[fill=black] (4.47,2.27) circle (0.5mm);
\draw[green,opacity=0.5,very thick] (0-0.5,0-1) +(20:3) arc(20:13:3);
\draw[green,opacity=0.5,very thick] (0-0.5,0-1) +(145:3) arc(145:160:3);
\draw[fill=black] (0,0) circle (1mm) node[above] {$y_0(x)$};
\draw[fill=black] (0-0.5,0-1) circle (1mm) node[below] {$O$};
\end{tikzpicture}
\end{minipage}
\caption{Steps 1 and 2 of the proof of Proposition \ref{prop_linear_contours}.}
\label{fig_change_contour}
\end{figure}
\end{proof}

\begin{lemma}
\label{le_contour_in_affine}
If $F=-y^2$, for any good contour $\Gamma_x = \lacc \gamma_x(t)\in U/\; t\in\Rm^d \racc$ such that $\gamma_0(0)=0$, the contour $\gamma_x'(0)\Rm^d + \gamma_x(0)$ is also good for $F$, and it is homotopic to $\Gamma_x$.
\end{lemma}

\begin{proof}
By hypothesis, there exists $C>0$ such that for all $v\in\Rm^d$ with $\gamma_x(v)\in U$,

\[
-\Re (\gamma_0(v)^2) \le -C|\gamma_0(v)|^2
\]
and since $\gamma_0(0)=0$, for $v$ small enough

\[
\lpar\Im(\gamma_0'(0))v\rpar^2-\lpar\Re(\gamma_0'(0))v\rpar^2 + O(|v|^3) \le 0.
\]
Hence $\lnor \Im(\gamma_0'(0))\Re(\gamma_0'(0))^{-1} \rnor<1$, and by continuity it is still true for $x$ small, so $\gamma_x'(0)\Rm^d+\gamma_x(0)$ is a good contour according to Proposition \ref{prop_contours_shape}. Therefore, there exists $\alpha>1$ such that for all $0\le t\le 1$ and $v\in\Rm^d$, if $y = \lambda\lpar \gamma_0'(0)v + \gamma_0(0) \rpar + (1-t)\gamma_0(t)$ then $\Re(y)\ge \alpha\Im(y)$, and so $-\Re(y^2)\le -\frac{\alpha^2-1}{\alpha^2+1}|y|^2$. Hence, by continuity $\lpar t\lpar \gamma_x'(0)\cdot + \gamma_x(0) \rpar + (1-t)\gamma_x\rpar\Rm^d$ is a good contour for all $0\le t\le 1$ if $x$ is small enough.
\end{proof}

\begin{lemma}[Parametric Morse Lemma]
\label{le_morse}
Let $F$ be holomorphic and satisfy \eqref{eq_cond_phase}. Then, there exists $\eta,\rho>0$ and a holomorphic function $K$ on $B_{\eta}\times B_{\rho}$ such that

\[
F(x,K_x(y)) = F(x,y_0(x))-y^2.
\]
Moreover, for all $x\in B_{\eta}$, the function $K_x$ is a local diffeomorphism.
\end{lemma}

\begin{proof}
By definition, for $x\in U$ fixed, the function $y\mapsto F(x,y+y_0(x))-F(x,y_0(x))$ has a non-degenerate critical point at $0$, thus by Taylor's theorem there exists holomorphic functions $(g_j)_{1\le j\le d}$ such that $F(x,y+y_0(x))-F(x,y_0(x)) = -\sum\limits_{1\le j,k\le d} g_{j,k}(x,y) y_jy_k$ for all $(x,y)\in U\times V$. Furthermore, $g_{j,k}(x,0) = -\frac{1}{2}\partial_{y_j,y_k}^2F(x,y_0(x))$ forms a non-degenerate matrix, so there exists $\eta,\rho>0$ such that the quadratic form $\xi \mapsto \sum\limits_{1\le j,k\le d} g_{j,k}(x,y) \xi_j\xi_k$ is non-degenerate for all $(x,y)\in B_{\eta}\times B_{\rho}$. For such $x,y$, there exists a linear application $M_{x,y}$ that reduces the quadratic form to $\sum\limits_{1\le j\le d} \xi_j^2$, and is holomorphic with respect to $x$ and $y$ for $\eta,\rho$ small enough. Then, for all $x\in B_{\eta}$, using the local inverse theorem,

\[
G_x : y \mapsto M_{x,y}^{-1}(y)
\]
is a local holomorphic change of variable such that $F(x,G_x^{-1}(y)+y_0(x)) = F(x,y_0(x)) - y^2$, hence $K_x(y) = G_x^{-1}(y)+y_0(x)$.
\end{proof}

\begin{proposition}
\label{prop_general_contours}
Let $F$ be a holomorphic function satisfying \eqref{eq_cond_phase} and for $j=0,1$, let $\Gamma_j = \lacc \gamma_j(x,y) \racc$ be good contours such that $\gamma_j(0,0) = y_0(0)$. Then $\Gamma_0$ and $\Gamma_1$ are homotopic up to a change of orientation.
\end{proposition}

\begin{proof}
According to Lemma \ref{le_morse}, there exists $\rho,\eta>0$ and $K$ such that for all $x\in B_{\eta}$ we can apply the change of variable $y\mapsto K_x(y)$ on $B_{\rho}$ which changes $F$ into $-y^2$ and the good contours for $F$ become locally good contours for $-y^2$ of the form $\lacc \mu_j(x,y) \racc$ with $\mu_j(0,0)=0$. We can thus suppose $F=-y^2$ from now on. Then, according to Lemma \ref{le_contour_in_affine} for $j=0,1$, $\Gamma_j$ is homotopic to $\Gamma'_j = \partial_y\gamma_j(x,0)\Rm^d+\gamma_j(x,0)$. Finally, according to Proposition \ref{prop_linear_contours}, the contours $\Gamma'_j$ are in homotopic up to a change of orientation.
\end{proof}

\begin{remark}
We proved that two good contours are homotopic if they go through $y_0(0)$ for $x=0$, hence the equality between the integrals up to a small remainder. Now, if they stay far away from $y_0(x)$, the corresponding integral are exponentially small, so  the integrals are still equal up to a small remainder.
\end{remark}

Actually, the integrals that will appear in the other sections will have a more specific setting. We will mainly consider integrals over $\Cm^d$ of real analytic functions. For that reason, we consider the following framework.

\begin{definition}
Let $(M,J)$ be a complex manifold, its complexification $\widetilde{M}$ is a small neighbourhood of the diagonal in $M\times\overline{M}$, endowed with the complex structure $(J,-J)$. In local coordinates, we write $\widetilde{f}(x,\overline{y})$ the holomorphic functions of $\widetilde{M}$. Let $f$ be a real-analytic function on $M$, and $x_0\in M$, then there exists a unique holomorphic function $\widetilde{f}$ on a neighbourhood of $(x_0,\overline{x_0})$ in $\widetilde{M}$ such that $\widetilde{f}(x,\overline{x}) = f(x)$, called the holomorphic extension of $f$. See \cite{dele25} for a detailed construction of these objects, we can always use local coordinates and the construction on $\Cm^d$ that we describe below.
\end{definition}

\begin{lemma}
\label{le_holo_ext}
Consider $M=\Cm^d$, if $f$ is a real-analytic function on $B_{\rho}$ with $\rho>0$, then there exists $\eta\in ]0,\rho]$ and a unique holomorphic function $\widetilde{f}$ on $\widetilde{B_{\eta}}$ such that $\widetilde{f}|_{\diag[\eta]} = f$.
\end{lemma}

\begin{proof}
By definition, there exists $\eta\in ]0,\rho]$ such that $f(z) = \sum\limits_{k,l=0}^{+\infty} \frac{\partial^k\overline{\partial}^lf(0)}{k!l!} z^k\overline{z}^l$ for all $z\in B_{\eta}$, and the series converges uniformly. Hence, $\widetilde{f}(z,\overline{w}) = \sum\limits_{k,l=0}^{+\infty} \frac{\partial^k\overline{\partial}^lf(0)}{k!l!} z^k\overline{w}^l$ works.
\end{proof}

\begin{lemma}
\label{le_comp_contours}
Let $\phi_1,\phi_2$ be real analytic on $\Cm^{d+n}$ such that for $x$ in a neighbourhood of $0$ and $j=1,2$ the function $z\mapsto \Re \phi_j(x,z)$ has a unique critical point $z_j(x)$ of signature $(0,-2n)$, with $z_j(0)=0$. Suppose furthermore that $\partial_z(\phi_1+\phi_2)(0,0)=0$, $\overline{\partial}_z(\phi_1+\phi_2)(0,0)=0$, then denoting $F(x,z,\overline{v}) = \lpar\widetilde{\phi_1}+\widetilde{\phi_2}\rpar (x,\overline{x},z,\overline{v})$, there exists $\rho,\sigma>0$ such that $F$ restricted to $B_{\sigma}\times B_{\rho}$ satisfies \eqref{eq_cond_phase}.
\end{lemma}

\begin{proof}
By hypothesis $\hess_z\Re(\phi_1+\phi_2)(0,0)$ is non-degenerate, thus $\hess_{z,\overline{v}}F(0,0)$ is too, and by the implicit function theorem there exists $\rho,\sigma>0$ such that for all $|x|\le\sigma$, $F(x,\cdot)$ admits a unique critical point $\lpar z_c(x),\overline{v_c}(x)\rpar$ on $B_{\rho}$, and by hypothesis $z_0(0) = 0$, $\overline{v_c}(0) = 0$. Therefore, $F$ restricted to $B_{\sigma}\times B_{\rho}$ satisfies \eqref{eq_cond_phase}.
\end{proof}

\begin{proposition}
\label{prop_contour_aext}
Let $\phi_1,\phi_2$ be real analytic on $\Cm^{d+n}$ such that for $x$ in a neighbourhood of $0$ and $j=1,2$ the function $z\mapsto \Re \phi_j(x,z)$ has a unique critical point $z_j(x)$ of signature $(0,-2n)$, with $z_j(0)=0$. Suppose furthermore that $\partial_z(\phi_1+\phi_2)(0,0)=0$, $\overline{\partial}_z(\phi_1+\phi_2)(0,0)=0$, and consider the notations of Lemma \ref{le_comp_contours}. Then,

\begin{align*}
	  & \int_{B_{\rho}} e^{\frac{\phi_1(x,z)+\phi_2(x,z)}{\hbar}} a(x,z) dz\\
	= & e^{\frac{F\lpar x,z_c(x),\overline{v_c}(x)\rpar}{\hbar}} \lpar \int_{B_{\rho}} e^{-\frac{|y|^2}{\hbar}} s(y,\overline{y})\widetilde{a}\lpar x,\overline{x},K_x(y,\overline{y})\rpar dy + O(e^{-\frac{c}{\hbar}}) \rpar,
\end{align*}
where $K_x$ is such that $F\lpar x,K_x(z,\overline{v})\rpar = F(x,z_c(x),\overline{v_c}(x)) -z\overline{v}$, in particular $K_x(0,0) = \lpar z_c(x),\overline{v_c}(x)\rpar$, and

\[
s(0,0) = \lpar -\det\hess_{y,\overline{v}}F\lpar x,z_c(x),\overline{v_c}(x)\rpar\rpar^{-\frac{1}{2}}.
\]
\end{proposition}

\begin{proof}
According to Lemma \ref{le_comp_contours}, there exists $\rho,\sigma>0$ such that $F$ satisfies \eqref{eq_cond_phase} on $B_{\sigma}\times B_{\rho}$, which is equivalent to the existence of $\lpar z_c(x),\overline{v_c}(x) \rpar$. Since $z_j(x)$ are critical points of signature $(0,-2n)$, we can consider $\sigma,\rho$ smaller such that

\[
\Re(\phi_j(x,z)-\phi_j(x,z_j(x)))\le -c_j|z-z_j(x)|
\]
for all $|x|\le\sigma$ and $|z|\le\rho$, with $c_j>0$. Then, denote $c=\max(c_1,c_2)/4$, and consider $\rho'\in ]0,\rho[$, then for $|z|\ge\rho'$

\begin{align*}
	& \Re F(x,z,\overline{z})- \Re F(x,z_c(x),\overline{v_c}(x))\\
	\le & -c_1|z-z_1(x)|^2 -c_2|z-z_2(x)|^2 + \Re \phi_1(x,z_1(x)) + \Re \phi_2(x,z_2(x))- \Re F(x,z_c(x),\overline{v_c}(x))\\
	\le & -\frac{c_1}{2}|z-z_c(x)|^2 -\frac{c_2}{2}|z-v_c(x)|^2 +c_1|z_1(x)-z_c(x)|^2 +c_2|z_2(x)-v_c(x)|^2\\
	& + \Re \phi_1(x,z_1(x)) + \Re \phi_2(x,z_2(x))- \Re F(x,z_c(x),\overline{v_c}(x))\\
	\le & -c\lpar |z-z_c|^2+|v-v_c|^2 \rpar
\end{align*}
for $x$ small enough, since $|z_1(x)|,|z_2(x)|,|z_c(x)|,|v_c(x)| \to 0$ when $|x|\to 0$. This implies that $\diag$ is a good contour for $F$. According to Lemma \ref{le_morse}, there exists a diffeomorphism $K_x$ such that for $(z,\overline{v})$ small enough

\[
F\lpar x,K_x(z,\overline{v})\rpar = F(x,z_c(x),\overline{v_c}(x)) -z\overline{v}.
\]
For $x\in B_{\eta}$, using Lemma \ref{le_holo_ext} and the change of variable $K_x$

\begin{align*}
	& \int_{B_{\rho}} e^{\frac{\phi_1(x,z)+\phi_2(x,z)}{\hbar}} a(x,z) dz\\
	= & \int_{\diag[\rho]} e^{\frac{F(x,z,\overline{v})}{\hbar}} \widetilde{a}(x,z,\overline{v}) dz\wedge d\overline{v}\\
	= & e^{\frac{F(x,z_c(x),\overline{v_c}(x))}{\hbar}} \int_{\Gamma} e^{-\frac{y\overline{w}}{\hbar}} s(y,\overline{w}) \widetilde{a}\lpar x,\overline{x},K_x(y,\overline{w})\rpar dy\wedge d\overline{w},
\end{align*}
where $\Gamma = K_x^{-1}\lpar\diag[\rho]\rpar$ and $s(z,\overline{v}) = \det\nabla_{z,\overline{v}}K_x$. Though, for $\sigma,\rho$ small enough, there exists $\rho'$ such that for all $x\in B_{\eta}$

\[
\Gamma \supset K_x^{-1}\lpar\diag- \lpar z_c(x),\overline{v_c}(x)\rpar\rpar \cap B_{\rho'},
\]
and since the integral outside this new set is negligible, and $\diag$ is a good contour for $-z\overline{v}$, we can apply Proposition \ref{prop_general_contours}

\begin{align*}
	& \int_{\Gamma} e^{-\frac{y\overline{w}}{\hbar}} s(y,\overline{w}) \widetilde{a}\lpar x,\overline{x},K_x(y,\overline{w})\rpar dy\wedge d\overline{w}\\
	= & \pm \int_{B_{\rho'}} e^{-\frac{|y|^2}{\hbar}} s(y,\overline{y}) \widetilde{a}\lpar x,\overline{x},K_x(y,\overline{y})\rpar dy + O(e^{-\frac{c}{\hbar}}).
\end{align*}
Then, $K_x$ can be chosen such that the sign is positive, in which case, it is isotopic to the identity, and since

\[
\lpar\nabla_{0,0}K_x\rpar^{\intercal} \hess_{y,\overline{v}}F\lpar x,z_c(x),\overline{v_c}(x)\rpar \nabla_{0,0}K_x = \begin{pmatrix}
0 & -I_d\\
-I_d & 0
\end{pmatrix},
\]
it implies that $\det\nabla_{0,0}K_x = \lpar -\det\hess_{y,\overline{v}}F\lpar x,z_c(x),\overline{v_c}(x)\rpar\rpar^{-\frac{1}{2}}$.
\end{proof}

\section{Références}

\sloppy
\printbibliography[heading=none]

\end{document}